\documentclass[a4paper,11pt]{article}
\usepackage{amsmath, amsfonts, amscd, amssymb, amsthm, enumerate, xypic}

\def\bfB{\mathbf{B}}
\def\calSH{\mathcal{S}\mathcal{H}}

\DeclareMathOperator{\id}{\operatorname{id}}

\DeclareMathOperator{\Mat}{\operatorname{M}}
\DeclareMathOperator{\Hom}{\operatorname{Hom}}
\DeclareMathOperator{\Mata}{\operatorname{A}}
\DeclareMathOperator{\Mats}{\operatorname{S}}
\DeclareMathOperator{\End}{\operatorname{End}}
\DeclareMathOperator{\Gal}{\operatorname{Gal}}

\DeclareMathOperator{\Minor}{\operatorname{Minor}}
\DeclareMathOperator{\catch}{\operatorname{catch}}
\DeclareMathOperator{\NT}{\operatorname{NT}}
\DeclareMathOperator{\GL}{\operatorname{GL}}
\DeclareMathOperator{\Ker}{\operatorname{Ker}}
\DeclareMathOperator{\maxrk}{\operatorname{maxrk}}
\DeclareMathOperator{\sprk}{\operatorname{sprk}}
\DeclareMathOperator{\Alt}{\operatorname{Alt}}
\DeclareMathOperator{\Sym}{\operatorname{Sym}}
\DeclareMathOperator{\Lrad}{\operatorname{Lrad}}
\DeclareMathOperator{\Rrad}{\operatorname{Rrad}}
\DeclareMathOperator{\Rad}{\operatorname{Rad}}
\DeclareMathOperator{\Vect}{\operatorname{span}}
\DeclareMathOperator{\im}{\operatorname{Im}}
\DeclareMathOperator{\tr}{\operatorname{tr}}

\DeclareMathOperator{\car}{\operatorname{char}}
\DeclareMathOperator{\Sp}{\operatorname{Sp}}
\DeclareMathOperator{\rk}{\operatorname{rk}}
\DeclareMathOperator{\Pf}{\operatorname{Pf}}

\renewcommand{\setminus}{\smallsetminus}
\renewcommand{\epsilon}{\varepsilon}

\def\F{\mathbb{F}}
\def\K{\mathbb{K}}
\def\R{\mathbb{R}}
\def\C{\mathbb{C}}

\renewcommand{\L}{\mathbb{L}}

\def\calA{\mathcal{A}}
\def\calB{\mathcal{B}}
\def\calC{\mathcal{C}}

\def\calF{\mathcal{F}}

\def\calI{\mathcal{I}}

\def\calM{\mathcal{M}}
\def\calN{\mathcal{N}}

\def\calQ{\mathcal{Q}}

\def\calS{\mathcal{S}}
\def\calT{\mathcal{T}}
\def\calU{\mathcal{U}}
\def\calV{\mathcal{V}}
\def\calW{\mathcal{W}}
\def\calX{\mathcal{X}}
\def\calY{\mathcal{Y}}

\def\lcro{\mathopen{[\![}}
\def\rcro{\mathclose{]\!]}}

\theoremstyle{definition}
\newtheorem{Def}{Definition}[section]
\newtheorem{Not}[Def]{Notation}

\theoremstyle{plain}
\newtheorem{theo}{Theorem}[section]
\newtheorem{prop}[theo]{Proposition}
\newtheorem{cor}[theo]{Corollary}
\newtheorem{lemma}[theo]{Lemma}
\newtheorem{claim}{Claim}

\theoremstyle{plain}

\theoremstyle{remark}
\newtheorem{Rems}{Remarks}[section]
\newtheorem{Rem}[Rems]{Remark}

\title{Large affine spaces of symplectic forms}
\author{Cl\'ement de Seguins Pazzis\footnote{Universit\'e de Versailles Saint-Quentin-en-Yvelines, Laboratoire de Math\'ematiques
de Versailles, 45 avenue des Etats-Unis, 78035 Versailles cedex, France}
\footnote{e-mail address: clement.de-seguins-pazzis@ac-versailles.fr}}

\begin{document}

\thispagestyle{plain}

\maketitle
\begin{abstract}
Let $\F$ be a field and $V$ be a $2n$-dimensional vector space over $\F$.
In a previous article, we have proved that if $\F$ has more than $2n-2$ elements then the greatest possible dimension for an affine
space of symplectic forms on $V$ is $n(n-1)$. Here, under the same cardinality assumption we study 
the spaces that have the critical dimension $n(n-1)$. In particular, if the characteristic of $\F$
is not $2$ the classification of these spaces up to congruence is reduced to: (1) the classification of nonisotropic quadratic forms over $\F$, up to equivalence
and multiplication with a nonzero scalar; (2) the classification of nonisotropic Hermitian forms over all quadratic extensions of $\F$, 
up to equivalence and multiplication by $-1$. In particular, for quadratically closed fields it is shown that there is exactly one solution up to congruence.
\end{abstract}



\tableofcontents

\section{Introduction}

\subsection{The problem}

Let $\F$ be a field (possibly with characteristic $2$) and $\calA$ be a simple algebra over $\F$ (that is, a nontrivial finite-dimensional unital and associative algebra over $\F$, with no nontrivial two-sided ideal).
Seemingly innocuous questions are the one of the greatest possible dimension for an affine subspace 
$\calS$ of $\calA$ that is included in the unit group $\calA^\times$, which we call \textbf{unital} affine subspaces,
and the one of understanding the solutions of greatest possible dimension, called the \textbf{optimal spaces}, up to equivalence
(two subsets $\calT$ and $\calT'$ being called equivalent when there exist units $a$ and $b$ in $\calA$ such that $\calT'=a\calT b$). 
In fact, these problems are highly nontrivial, and 
their solution deeply depends on the structure of the underlying field $\F$. If $\F$ is algebraically closed, 
these problem are entirely reducible to the problem of linear subspaces of nilpotent endomorphisms of vector spaces, initiated by Gerstenhaber in \cite{Gerstenhaber}
(see \cite{Meshulam} for the reduction process), however the spaces of maximal dimension are completely different if we turn, e.g., to the field of real numbers.
It was shown in \cite{dSPlargeaffinenonsingular} that, in the case where $\calA$ splits (i.e., it is isomorphic to a matrix algebra over $\F$) 
the classification of the optimal spaces can be reduced to the classification, up to congruence (i.e., equivalence and multiplication with a nonzero scalar), of the nonisotropic quadratic forms over $\F$. The result was recently generalized \cite{dSPaffineunits} 
to any simple algebra over $\F$ provided that the cardinality of $\F$ is large enough.
Spectacularly, if $\car(\F) \neq 2$ then the classification is reduced to the one of nonisotropic quadratic forms over $\F$, up to congruence,
and the classification of nonisotropic Hermitian forms over quadratic and quaternionic extensions of $\F$, up to congruence.

Although there remains the issue of fields with small cardinality, for which the techniques used so far are powerless, 
the aforementioned problem is largely settled, but this raises new questions. A natural variation of the problem is to take a 
central simple algebra $\calA$ over $\F$ with an involution $\sigma : x \mapsto x^\star$ (either of the first or of the second kind).
We will take the case of an involution of the first kind, for instance. Then we have the linear subspace $\Sym(\calA,\sigma):=\{x \in \calA : x^\star=x\}$
of all $\sigma$-symmetric elements, as well as the linear subspace $\Alt(\calA,\sigma):=\{x^\star-x\mid x \in \calA\}$ of all $\sigma$-alternating elements.
It is then natural to ask the same problem with respect to these special elements:
What is the greatest dimension for an affine subspace of $\Sym(\calA,\sigma)$ (respectively, of $\Alt(\calA,\sigma)$) that is included in $\calA^\times$? And can we classify the spaces of greatest possible dimension up to the action of the unitary group $\calU(\calA,\sigma):=\{x \in \calA : xx^\star=1_\calA\}$? We doubt that there can be reasonable general answers to these questions: as we will see here, the basic case of split algebras dwarfs the study of \cite{dSPaffineunits} in terms of complexity.
Hence, it is perfectly reasonable to lower the bar and to stick to split central simple algebras over $\F$, i.e.\ to full matrix algebras.
In that case, by the determination of the involutions over split central simple algebras, the problem is easily reduced to the conjunction of the following two problems, where we denote by $\Mats_n(\F)$ (respectively, $\Mata_n(\F)$) the set of all $n$-by-$n$ symmetric 
(respectively, alternating) matrices with entries in $\F$:
\begin{itemize}
\item Given an integer $n>0$, what is the greatest possible dimension for an $\F$-affine subspace of invertible matrices of
$\Mats_n(\F)$, and what are the spaces of greatest possible dimension, up to congruence?
\item Given an integer $n>0$, what is the greatest possible dimension for an $\F$-affine subspace of invertible matrices of
$\Mata_n(\F)$, and what are the spaces of greatest possible dimension, up to congruence?
\end{itemize}

Equivalently, the first (respectively, the second) problem amounts to determining, for an $n$-dimensional vector space $V$ over $\F$, 
the greatest possible dimension for an affine subspace of non-degenerate symmetric (respectively, alternating) bilinear forms on $V$,
and to classify the spaces of greatest possible dimension up to the natural action of the general linear group $\GL(V)$.
As in the problem we started from, over an algebraically closed field these problems are entirely reduced to the study of linear subspaces
of nilpotent elements in $\Mats_n(\F)$ and $\Mata_n(\F)$, for which many advances have been made in the past decades
\cite{MeshulamRadwan,DraismaKraftKuttler,BukovsekOmladic,structuredGerstenhaber1,structuredGerstenhaber2,structuredGerstenhaber3}.
However, the case of general fields behaves completely differently, and the techniques used to study linear subspaces of nilpotent elements are useless 
to deal with them.

Here, we will stick to the second question, and we will give a complete answer to it under a mild cardinality assumption on $\F$.
In fact, the question of the greatest possible dimension has already been answered recently \cite{dSPaffinealt}, and we immediately recall the result.
For a vector space $V$, we denote by $\calA^2(V)$ the vector space of all alternating bilinear forms on $V$.
A subset of $\calA^2(V)$ will be called \textbf{symplectic} when all its elements are non-degenerate, i.e., symplectic forms.

\begin{theo}[Theorem 2 from \cite{dSPaffinealt}]\label{theo:majdim}
Let $V$ be an $\F$-vector space with even dimension $2n$. Assume that $|\F| > 2n-2$.
Then the greatest possible dimension for a symplectic affine subspace of $\calA^2(V)$ is $n(n-1)$.
\end{theo}

This result had been obtained earlier by Rubei over the field of real numbers \cite{Rubei} with a method that uses peculiar properties
of this field.

The aim of this article is to understand the structure of the spaces of greatest possible dimension, which we call
the \textbf{optimal} affine spaces, under the same cardinality assumption as in Theorem \ref{theo:majdim}.

Before going into details, we will give the general spirit of the main result of the article.
Let $\calS$ be an optimal symplectic affine subspace of $\calA^2(V)$.
One key (and non-obvious) notion that we will develop is the one of irreducibility. 
A linear subspace $U$ of $V$ is called $\calS$-adapted when for every (symplectic) form $s$ in $\calS$
it is totally singular (that is $\forall (x,y)\in U^2, \; s(x,y)=0$) and the orthogonal complement of $U$ under 
$s$ does not depend on the choice of $s$ in $\calS$. Of course, the zero subspace is $\calS$-adapted, 
and so is every potential common Lagrangian for the elements of $\calS$.
We say that $\calS$ is \textbf{irreducible} when its only adapted subspace is $\{0\}$.
Now, take a maximal $\calS$-adapted subspace $U$, and denote by $U^\bot$ its common orthogonal complement under the elements of $\calS$.
Every element $s$ of $\calS$ then naturally induces:
\begin{itemize}
\item a vector space isomorphism $\overline{s} : x \in U \mapsto [y  \mapsto s(x,y)] \in \Hom(V/U^\bot,\F)$;
\item a symplectic form $[s]$ on $U^\bot/U$.
\end{itemize}
Then, one can prove, under the cardinality assumption of Theorem \ref{theo:majdim}, that:
\begin{enumerate}[(i)]
\item The set $\calW:=\{\overline{s} \mid s \in \calS\}$ is an affine subspace of isomorphisms from $U$ to $\Hom(V/U^\bot,\F)$
with the greatest possible dimension, which is $\dbinom{\dim U}{2}$;
\item The set $\calC:=\{[s] \mid s \in \calS\}$ is an irreducible optimal symplectic affine subspace of 
$\calA^2(U^\bot/U)$;
\item And finally $\calS$ precisely consists of the alternating bilinear forms $s$ on $V$ for which $U$ is orthogonal to $U^{\bot}$,
the induced mapping $\overline{s} : x \in U \mapsto [y  \mapsto s(x,y)] \in \Hom(V/U^\bot,\F)$ belongs to $\calW$, and the alternating form 
induced by $s$ on $U^\bot/U$ belongs to $\calC$.
\end{enumerate}
We shall say that $\calC$ is a \textbf{core} of $\calS$, and that $\calW$ is a \textbf{wing} of it.
Here, we cannot use a definite article because there are cases in which there are several maximal $\calS$-adapted subspaces.
Nevertheless, in some sense the equivalence class of a wing is uniquely determined by the congruence class of $\calS$,
and the congruence class of a core is uniquely determined by the one of $\calS$.
This essentially reduces the study of optimal symplectic affine subspaces to 
the solution of two separate problems:
\begin{itemize}
\item The classification of optimal affine subspaces of units in matrix spaces (over $\F$);
\item The classification of irreducible optimal symplectic affine subspaces.
\end{itemize}
As stated earlier, in \cite{dSPlargeaffinenonsingular} the former problem has been entirely reduced, over all fields with more than $2$ elements, 
to the classification of nonisotropic quadratic forms over $\F$, up to equivalence and multiplication with a nonzero scalar.
The latter problem will take up the greater part of the present article, and we will show in particular that for fields with characteristic other than $2$
this problem can be reduced to the classification of nonisotropic Hermitian forms over all quadratic extensions of $\F$, up to equivalence and multiplication by $\pm 1$.

\subsection{Examples}

It is useful that we give some examples of affine spaces of symplectic forms with large dimension.
We shall give our examples in matrix terms.
The most basic example is obtained by starting from a large affine subspace $\calU$ of $\Mat_n(\F)$ made of invertible matrices, 
and then by taking the set of all matrices of the form
$$\begin{bmatrix}
[0]_{n \times n} & U \\
-U^T & A
\end{bmatrix} \quad \text{with $U \in \calU$ and $A \in \Mata_n(\F)$.}$$
This is an affine subspace of invertible matrices in $\Mata_{2n}(\F)$.
For instance, we can take $\calU$ as the affine space of all upper-triangular matrices with all diagonal entries equal to $1$.
Then $\dim \calU=\frac{n(n-1)}{2}=\dim \Mata_n(\F)$, and the space we have created has dimension $n(n-1)$, which is the bound
from Theorem \ref{theo:majdim}.

If we reinterpret the previous example in terms of bilinear forms, a remarkable feature is that all the forms have a common Lagragian.
And for a long time we thought that this property was common to all the optimal spaces, i.e., that every optimal affine space of invertible alternating matrices is congruent to a solution of the above form.
However, a reevaluation of the problem lead us to discover a counterexample: In section 6 of \cite{dSPaffinealt} we produced an optimal symplectic affine subspace of $\calA^2(V)$ with no common Lagrangian. In matrix terms, one takes the space $\calM$ of all real alternating matrices of the form
$$\begin{bmatrix}
0 & x & y & 1 \\
-x & 0 & 1 & -y \\
-y & -1 & 0 & x \\
-1 & y & -x & 0
\end{bmatrix} \quad \text{with $(x,y) \in \R^2$.}$$
It can indeed be seen from the Pfaffian that all such matrices are invertible, as well as all the nonzero matrices in the translation vector space of $\calM$.
As a consequence $\Vect_\R(\calM X)$ has dimension $3$ for all $X \in \R^4 \setminus \{0\}$, which clearly bars the existence of a common Lagrangian
for all the matrices in $\calM$.
We also proved in \cite{dSPaffinealt} that similar examples exist over a general field $\F$ provided that $\car(\F) \neq 2$ and there exists a $3$-dimensional nonisotropic quadratic form over $\F$
(i.e., the u-invariant of $\F$ is greater than $2$).

Here is an even more intriguing example:
For the standard symplectic matrix
$$K_{2n}:=\begin{bmatrix}
0 & I_n \\
-I_n & 0
\end{bmatrix} \in \Mat_{2n}(\R),$$
let us consider the linear subspace $W \subseteq \Mata_{2n}(\R)$ consisting of all the alternating matrices $M \in \Mata_{2n}(\R)$ such that $K_{2n}^{-1}M$ is alternating.
\label{example:real}
A straightforward computation reveals that $W$ is the space of all matrices of the form
$$\begin{bmatrix}
A & B \\
B & A^T
\end{bmatrix} \quad \text{with $(A,B)\in \Mata_n(\R)^2$,}$$
and hence $\dim W=n(n-1)$.
It is possible to prove that $K_{2n}^{-1} W$ is irreducible if $n>1$, i.e., that no nontrivial subspace of $\R^{2n}$
is invariant under all its elements (see Section \ref{section:irroptimalintroduction} for a generalization of this).
As a consequence $K_{2n}+W$ has no common Lagrangian.
And we claim that $K_{2n}+W$ consists of symplectic matrices only. To see this, it suffices to note that 
$I_{2n}+K_{2n}^{-1}M$ is invertible for all $M \in W$, which follows from the observation that 
$K_{2n}^{-1} M$ is a real alternating matrix (and hence all its complex eigenvalues belong to the imaginary line $i\R$).

\subsection{Reduction to trivial spectrum subspaces}

Traditionally, the study of affine spaces of units in simple algebras can be reduced to so-called trivial spectrum subspaces.
Given an $\F$-algebra $\calA$, every unital affine subspace of $\calA$ is equivalent to one which contains the unity $1_\calA$.
An affine subspace $\calS$ of $\calA$ that contains $1_\calA$ is unital if and only if no element $x$ of its translation vector space $S$
is such that $1_\calA-x$ is invertible.

\begin{Def}
A linear subspace $S$ of an $\F$-algebra $\calA$ is said to have \textbf{trivial spectrum} whenever
$1_\calA-x$ is invertible for all $x \in S$.
\end{Def}

Hence, a linear subspace $S$ of endomorphisms of a finite-dimensional vector space $V$ has trivial spectrum whenever no element of $S$
has a nonzero fixed vector.

We shall shortly see that this idea can be adapted to study affine subspaces of symplectic forms on $V$, but before we do so we have to introduce an important notion:

\begin{Def}
Given a nondegenerate bilinear form $b$ on $V$, an endomorphism $u$ of $V$ is called \textbf{$b$-alternating}
whenever $\forall x \in V, \; b(x,u(x))=0$, and we denote by $\calA_b$ the set of all such endomorphisms
(it is a linear subspace of $\End(V)$, obviously).
\end{Def}

Now, let $s_0$ be an arbitrary symplectic form on $V$, and $S$ be a linear subspace of $\calA^2(V)$, 
yielding an affine subspace $\calS:=s_0+S$.
For every $s \in \calA^2(V)$, we consider the linear mapping $\widetilde{s} : x \in V \mapsto s(-,x) \in V^\star$,
thereby identifying $\calA^2(V)$ with a subspace of $\Hom(V,V^\star)$. 
Then we obtain a linear subspace of endomorphisms
$$\Phi(S):=\{\widetilde{s_0}^{-1} \circ \widetilde{s} \mid s \in S\} \subseteq \End(V),$$
and we see that $\calS$ is symplectic if and only if every element of $\id_V+\Phi(S)$ is invertible,
i.e.\ $\Phi(S)$ has trivial spectrum.
Finally, the endomorphisms of type $\widetilde{s_0}^{-1} \circ \widetilde{s}$ with $s \in \calA^2(V)$
are exactly the $s_0$-alternating endomorphisms of $V$.
Hence, the problem is essentially moved to the one of determining the trivial spectrum subspaces of $\calA_{s_0}$
with dimension $n(n-1)$: if we have such a trivial spectrum subspace $\calV$ then
$\{(x,y) \mapsto s_0(x,y+u(y)) \mid u \in \calV\}$ is a symplectic affine subspace of $\calA^2(V)$ with dimension $n(n-1)$
(that contain $s_0$).

Hence, we are essentially reduced to trying to classify, for an arbitrary symplectic form $s$ on $V$, 
the trivial spectrum linear subspaces of $\calA_s$ with greatest possible dimension,
and this is exactly how we will solve the problem.
In particular, we can already rephrase Theorem \ref{theo:majdim}:

\begin{theo}\label{theo:majodimTS}
Let $(V,s)$ be a symplectic space with dimension $2n \geq 2$, such that $|\F|> 2n-2$.
Then the greatest possible dimension for a trivial spectrum subspace of $\calA_{s}$ is $n(n-1)$.
\end{theo}

Before we go on, it is useful that we recall the solution to the analogue problem of large affine subspaces of invertible matrices.
This will also allow us to introduce useful additional terminology and notation.

\subsection{A review of the linear operators case}\label{section:reviewstandard}

Let $V$ be a finite-dimensional vector space over $\F$.
Here we review the classification of the optimal affine subspaces of $\End(V)$ that contain only 
automorphisms. In \cite{dSPlargeaffinenonsingular}, the results were stated in matrix terms; here we also state them in geometric terms.

Let $b$ a bilinear form on $V$. A basic remark is that if $b$ is a \emph{nonisotropic} bilinear form on $V$, then the space
$\calA_b$ of all $b$-alternating endomorphisms has trivial spectrum. Indeed, for all $u \in \calA_b$ and all $x \in V \setminus \{0\}$,
the equality $u(x)=x$ is impossible because it would lead to $b(x,x)=b(x,u(x))=0$, thereby violating the assumed nonisotropy of $b$.
Besides, the vector space $\calA_b$ is isomorphic to $\calA^2(V)$ through $u \mapsto [(x,y) \mapsto b(x,u(y))]$, so its dimension equals
$\dbinom{n}{2}$. The following theorem, which reformulates theorem 4 of \cite{dSPlargeaffinenonsingular}, essentially states that these solutions are the only irreducible ones.

\begin{theo}[Classification of optimal trivial spectrum subspaces of $\End(V)$]\label{theo:matrixcase}
Let $\calS$ be an optimal trivial spectrum subspace of $\End(V)$, where $|\F|>2$.
Then the $\calS$-invariant subspaces form a flag $(V_i)_{0 \leq i \leq p}$ of $V$, and there exists
a list $(b_i)_{1 \leq i \leq p}$ of bilinear forms, where $b_i$ is a nonisotropic bilinear form on $V_i/V_{i-1}$, such that
$\calS$ is the set, denoted by $\calA_{b_1,\dots,b_p}$, of all $u \in \End(V)$ that satisfy the following properties:
\begin{enumerate}[(i)]
\item $u$ leaves each subspace $V_i$ invariant;
\item For all $i \in \lcro 1,p\rcro$, the endomorphism of $V_i/V_{i-1}$ induced by $u$ is $b_i$-alternating.
\end{enumerate}
Moreover, the forms $b_1,\dots,b_p$ are uniquely determined by $\calS$ up to multiplication with nonzero scalars.

Conversely, let $(V_i)_{0 \leq i \leq p}$ be a flag of subspaces of $V$, and let
$(b_i)_{1 \leq i \leq p}$, where $b_i$ is a nonisotropic bilinear form on $V_i/V_{i-1}$ for all $i \in \lcro 1,p\rcro$.
Then $\calA_{b_1,\dots,b_p}$ is an optimal trivial spectrum subspace of $\End(V)$, and its invariant subspaces are the $V_i$'s.
\end{theo}

As a consequence, the classification of the optimal trivial spectrum subspaces of endomorphisms
is equivalent to the classification of the nonisotropic bilinear forms over $\F$, up to equivalence and multiplication by nonzero scalars
(i.e., up to \emph{congruence}). It is known, over fields with characteristic other than $2$,
that such a classification can be reduced to the classification of Hermitian forms over algebraic extensions of $\F$
(see \cite{Riehm}).

We now express the previous theorem in matrix terms, but beforehand it is useful to make the terminology clearer.

\begin{Def}
Two subsets $\calX$ and $\calY$ of $\Mat_{n,p}(\F)$ are called \textbf{equivalent}, and we write $\calX \sim \calY$, when there are
invertible matrices $P\in \GL_n(\F)$ and $Q \in \GL_p(\F)$ such that $\calY=P \calX Q$, i.e., when $\calX$ and $\calY$ represent the same set
of linear operators (from a $p$-dimensional vector space to an $n$-dimensional vector space) in a different choice of bases of the source and target spaces.

Two subsets $\calX$ and $\calY$ of $\Mat_n(\F)$ are called \textbf{similar}, and we write $\calX \simeq \calY$, when there is
an invertible matrix $P\in \GL_n(\F)$ such that $\calY=P \calX P^{-1}$, i.e.,  when $\calX$ and $\calY$ represent the same set
of endomorphisms (of an $n$-dimensional vector space) in a different choice of basis.

Two subsets $\calX$ and $\calY$ of $\Mat_n(\F)$ are called \textbf{congruent}, and we write $\calX \cong \calY$, when there is
an invertible matrix $P\in \GL_n(\F)$ such that $\calY=P \calX P^T$, i.e.,  when $\calX$ and $\calY$ represent the same set
of bilinear forms (of an $n$-dimensional vector space) in a different choice of basis.
\end{Def}

\begin{Def}
Two bilinear forms $b$ and $b'$ on respective vector spaces $V$ and $V'$ are called equivalent, and we write $b \simeq b'$, when
there is a vector space isomorphism $\varphi : V \overset{\simeq}{\longrightarrow} V'$ such that $\forall (x,y)\in V^2, \; b(x,y)=b'(\varphi(x),\varphi(y))$.
This means that $b$ and $b'$ have a common Gram matrix.

Two bilinear forms $b$ and $b'$ on respective vector spaces $V$ and $V'$ are called congruent, and we write $b \cong b'$, when
$b \simeq \lambda b'$ for some $\lambda \in \F^\times$.

We adopt similar definitions for quadratic forms instead of bilinear forms (a crucial precision for fields with characteristic $2$): two quadratic forms $q$ and $q'$ on respective vector spaces $V$ and $V'$ are called equivalent, and we write $q \simeq q'$, when there is a vector space isomorphism $\varphi : V \overset{\simeq}{\longrightarrow} V'$ such that
$\forall (x,y)\in V^2, \; q(x)=q'(\varphi(x))$; and they are called congruent if $q \simeq \lambda q'$ for some $\lambda \in \F^\times$.
\end{Def}

We can now state the results:
we take nonisotropic\footnote{A square matrix $M \in \Mat_q(\F)$ is nonisotropic whenever
$\forall X \in \F^q \setminus \{0\}, \; X^TMX \neq 0$.} (and non-void) matrices $P_1 \in \GL_{n_1}(\F),\dots,P_p \in \GL_{n_p}(\F)$,
and we consider the space $\calA_{P_1,\dots,P_p}$ of all block-upper-triangular matrices of the form
$$\begin{bmatrix}
P_1^{-1} A_1 & & [?] \\
& \ddots & \\
[0] & & P_p^{-1} A_p
\end{bmatrix}$$
where $A_1 \in \Mata_{n_1}(\F),\dots,A_p \in \Mata_{n_p}(\F)$, and all the other unspecified upper-diagonal blocks are arbitrary.
Such a subspace is then an optimal trivial spectrum subspace of $\Mat_N(\F)$ for $N:=\sum_{k=1}^p n_k$.
And if $|\F|>2$ then every optimal trivial spectrum subspace $\calM$ of $\Mat_N(\F)$ is similar to $\calA_{P_1,\dots,P_p}$
for some integer $p>0$ and some list $(P_1,\dots,P_p)$ of nonisotropic square matrices, each term in such a list being then uniquely determined
up to matrix congruence and multiplication with a nonzero scalar.

Next, we state the corresponding result for affine spaces of isomorphisms, which reformulates theorem 7 of \cite{dSPlargeaffinenonsingular} in geometric terms.

\begin{theo}[Classification of optimal affine subspaces of isomorphisms in $\Hom(V,V')$]
Let $\calS$ be an optimal affine subspace of isomorphisms in $\Hom(V,V')$, where $|\F|>2$.
Then there exists $u_0 \in \calS$ as well as a flag $\calF=(V_i)_{0 \leq i \leq p}$ of $V$,
and for every $i \in \lcro 1,p\rcro$, a nonisotropic bilinear form $b_i$ on $V_i/V_{i-1}$, such that
$\calS$ is the set, denoted by $\calA_{u_0,b_1,\dots,b_p}$, of all $u \in \Hom(V,V')$ that satisfy the following properties:
\begin{enumerate}[(i)]
\item $u(V_i)=u_0(V_i)$ for each $i \in \lcro 1,p\rcro$;
\item For all $i \in \lcro 1,p\rcro$, the endomorphism $v_i$ of $V_i/V_{i-1}$ induced by $u_0^{-1} u$
satisfies $\forall x \in V_i/V_{i-1}, \; b_i(x,v_i(x))=b_i(x,x)$.
\end{enumerate}
Moreover, the flag $(V_i)$ is uniquely determined by $\calS$, and each quadratic form $x \mapsto b_i(x,x)$
is uniquely determined by $\calS$ up to multiplication with a nonzero scalar.

Conversely, let $(V_i)_{0 \leq i \leq p}$ be a flag of subspaces of $V$, and let
$(b_i)_{1 \leq i \leq p}$, where $b_i$ is a nonisotropic bilinear form on $V_i/V_{i-1}$ for all $i \in \lcro 1,p\rcro$.
Let $u_0$ be an isomorphism from $V$ to $V'$.
Then $\calA_{u_0,b_1,\dots,b_p}$ is an optimal affine subspace of isomorphisms from $V$ to $V'$.
\end{theo}

In matrix terms, the previous theorem can be expressed as follows:
we take nonisotropic matrices $P_1 \in \GL_{n_1}(\F),\dots,P_p \in \GL_{n_p}(\F)$, and we consider the space $\calI_{P_1,\dots,P_p}$ of all block-upper-triangular matrices of the form
$$\begin{bmatrix}
P_1+ A_1 & & [?] \\
& \ddots & \\
[0] & & P_p+ A_p
\end{bmatrix}$$
where $A_1 \in \Mata_{n_1}(\F),\dots,A_p \in \Mata_{n_p}(\F)$, and all the other unspecified upper-diagonal blocks are arbitrary.
Such a subspace is then an optimal affine subspace of invertible matrices of $\Mat_N(\F)$ for $N:=\sum_{k=1}^p n_k$.
And if $|\F|>2$ then every optimal affine subspace $\calM$ of invertible matrices of $\Mat_N(\F)$ is equivalent to $\calI_{P_1,\dots,P_p}$
for some integer $p>0$ and some list $(P_1,\dots,P_p)$ of nonisotropic square matrices,
each (nonisotropic) quadratic form $X \mapsto X^T P_iX$ being then uniquely determined by $\calM$ up to congruence.
In particular, if $\car(\F)\neq 2$ the matrices $P_i$ can be chosen symmetric, and then their congruence classes are uniquely determined by $\calM$
up to multiplication with nonzero scalars. In the general case (still excluding $\F_2$)
we conclude that the classification, up to equivalence, of the optimal affine subspaces of invertible matrices of $\Mat_n(\F)$
amounts to the classification of nonisotropic quadratic forms over $\F$ up to congruence.

\subsection{Reducing the problem to the irreducible case (the matrix viewpoint)}\label{section:intromatrix}

As is also the case in the study of trivial spectrum subspaces of plain endomorphisms, a key role is played by the notion of \emph{irreducibility}.
A subset $\calT$ of $\End(V)$ is irreducible if and only if no nontrivial linear subspace of $V$ is invariant under all the elements of $\calT$.

For symplectic spaces, the definition is different: we say that a (non-empty) set $\calX$ of symplectic forms on $V$
is \textbf{irreducible} when there is no nontrivial subspace $W$ of $V$ for which the orthogonal subspace $W^{\bot_s}$ includes $W$ and does not depend on the choice of $s$ in $\calX$. In matrix fashion, this property means that there exists a basis $(e_1,\dots,e_{2n})$ of $V$ whose first $r$
vectors span $W$ and for which the Gram matrices of the elements of $\calX$ all take the form
$$\begin{bmatrix}
[0]_{r \times r} & [0]_{r \times (2n-2r)} & [?]_{r \times r} \\
[0]_{(2n-2r) \times r} & [?]_{(2n-2r) \times (2n-2r)} & [?]_{(2n-r) \times r} \\
[?]_{r \times r} & [?]_{r \times (2n-2r)} & [?]_{r \times r}
\end{bmatrix}.$$

\begin{Not}\label{not:bullet}
Let $\calW$ and $\calC$ be respective subsets of $\Mat_r(\F)$ and $\Mata_{2s}(\F)$. We set
$$\calW \bullet \calC:=\left\{
\begin{bmatrix}
[0]_{r \times r} & [0]_{r \times 2s} & W \\
[0]_{2s \times r} & C & B \\
-W^T & -B^T & D
\end{bmatrix} \mid (W,C,B,D)\in \calW \times \calC \times \Mat_{2s,r}(\F) \times \Mata_r(\F)\right\},$$
which is a subset of $\Mata_{2(r+s)}(\F)$.
\end{Not}

Note that if $\calW$ and $\calC$ are affine subspaces of invertible matrices then
$\calW \bullet \calC$ is an affine subspace of invertible matrices of $\Mata_{2(r+s)}(\F)$. If in addition
$\dim \calW=\dbinom{r}{2}$ and $\dim \calC=s(s-1)$ then it is easy to compute that
$$\dim(\calW \bullet \calC)=(r+s)(r+s-1).$$
This yields the following result:

\begin{prop}
Let $r$ and $s$ be non-negative integers, and assume that $|\F|>2(r+s-1)$.
Let $\calW$ be an optimal affine space of invertible matrices of $\Mat_r(\F)$, and
$\calC$ be an optimal affine space of invertible matrices of $\Mata_{2s}(\F)$. Then
$\calW \bullet \calC$ is an optimal affine space of invertible matrices of $\Mata_{2(r+s)}(\F)$,
and hence so is every congruent set in $\Mata_{2(r+s)}(\F)$.
\end{prop}

If $\calW'$ and $\calC'$ are respective subsets of $\Mat_r(\F)$ and $\Mata_{2s}(\F)$
with $\calW' \sim \calW$ and $\calC' \cong \calC$, then
$$\calW' \bullet \calC' \cong \calW \bullet \calC.$$
Indeed, let us take invertible matrices $P,Q,R$ in $\GL_r(\F)$, $\GL_r(\F)$ and $\GL_{2s}(\F)$ respectively, such that
$\calW'=P\calW Q$ and $\calC'=R \calC R^T$.
Then $S:=P \oplus R \oplus Q^T$ belongs to $\GL_{2r+2s}(\F)$ and it is straightforward to check that
$$S(\calW \bullet \calC) S^T=\calW' \bullet \calC'.$$

We can now state our first classification result:

\begin{theo}[Core-Wing Decomposition Theorem]\label{theo:separationtheorem}
Let $\calS$ be an optimal symplectic affine subspace of $\Mata_{2n}(\F)$, with $|\F|> 2n-2$.

Then there exists an integer $r \in \lcro 0,n\rcro$, an optimal affine subspace $\calW$ of invertible elements of $\Mat_r(\F)$,
and an irreducible optimal symplectic affine subspace $\calC$ of $\Mata_{2n-2r}(\F)$ such that
$$\calS \cong \calW \bullet \calC.$$
The integer $r$ is then uniquely determined by $\calS$, as well as the equivalence class of $\calW$ and the congruence class of $\calC$.
\end{theo}

In particular, the elements of $\calS$ have a common Lagrangian if and only if the integer $r$ in the previous decomposition equals $\frac{n}{2}$,
or equivalently the core part $\calC$ vanishes (i.e., it is limited to the $0$-by-$0$ matrix).

The geometric interpretation of this matrix construction has been hinted at earlier in this introduction, and it will 
be explained in great depth in Section \ref{section:reduction}.

The bottom line of the Core-Wing Decomposition Theorem is that understanding the optimal symplectic affine subspaces of $\Mata_{2n}(\F)$ amounts to
understanding the optimal affine subspaces of nonsingular square matrices on the one hand, and the \emph{irreducible} optimal affine subspaces of nonsingular alternating square matrices on the other hand. The latter are described in the next section.

Let us final turn to the corresponding decomposition for trivial spectrum subspaces of $\calA_s$, where $s$ is a fixed symplectic form
on the vector space $V$.

First of all, for $k \geq 0$ we denote by $s_{2k}$ the standard symplectic form on $\F^{2k}$,
with Gram matrix $K_{2k}=\begin{bmatrix}
0 & I_k \\
-I_k & 0
\end{bmatrix}$ in the standard basis. Note that the space $\calA_{s_{2k}}$ is represented in the standard basis by
$K_{2k}^{-1} \Mata_{2k}(\F)$.

Let $s$ be a symplectic form on $V$.
A family $(e_1,\dots,e_r,f_1,\dots,f_r)$ of $V$ is called \textbf{symplectic} whenever $s(e_i,f_j)=\delta_{i,j}$ for all $(i,j)\in \lcro 1,r\rcro^2$.

A \textbf{mixed symplectic basis} of $(V,s)$ with index $r$ is a basis of the form $(e_1,\dots,e_r,g_1,\dots,g_{2n-2r},f_1,\dots,f_r)$
in which $(e_1,\dots,e_r,f_1,\dots,f_r)$ and $(g_1,\dots,g_{2n-2r})$ are symplectic
and every $g_i$ is $s$-orthogonal to all the vectors $e_1,\dots,e_r,f_1,\dots,f_r$. This amounts to having the Gram matrix of $s$
in that basis equal to
$$\begin{bmatrix}
0 & 0 & I_r \\
0 & K_{2n-2r} & 0 \\
-I_r & 0 & 0
\end{bmatrix}.$$
In this case we say that this basis is \textbf{adapted} to the subspace $\Vect(e_1,\dots,e_r)$.

In the remainder, it will be useful to note that if we have such a basis $\bfB$, then the $s$-alternating endomorphisms of $V$
are those with matrix in $\bfB$ of the form
$$\begin{bmatrix}
W & (K_{2n-2r} B)^T & D \\
B' & K_{2n-2r}^{-1} C & B \\
D' & -(K_{2n-2r}B')^T & W^T
\end{bmatrix}$$
with $W \in \Mat_r(\F)$, $C \in \Mata_{2n-2r}(\F)$, $B$ and $B'$ in $\Mat_{2n-2r,r}(\F)$ and $D$ and $D'$ in $\Mata_r(\F)$.
Those that leave $\Vect(e_1,\dots,e_r)$ invariant are then the endomorphisms with matrix in $\bfB$ of the form
$$\begin{bmatrix}
W & (K_{2n-2r} B)^T & D \\
[0]_{(2n-2r) \times r} & K_{2n-2r}^{-1} C & B \\
[0]_{r \times r} & [0]_{r \times (2n-2r)} & W^T
\end{bmatrix}$$
with $W \in \Mat_r(\F)$, $C \in \Mata_{2n-2r}(\F)$, $B$ in $\Mat_{2n-2r,r}(\F)$ and $D$ in $\Mata_r(\F)$.
For respective linear subspaces $\calW$ and $\calC$ of $\Mat_r(\F)$ and $\calA_{s_{2n-2r}}$ (which we naturally identify with $K_{2n-2r}^{-1} \Mata_{2n-2r}(\F)$),
we set
\begin{multline*}
\calW \triangle \calC:=\Biggl\{\begin{bmatrix}
W & (K_{2n-2r} B)^T & D \\
[0]_{(2n-2r) \times r} & C & B \\
[0]_{r \times r} & [0]_{r \times (2n-2r)} & W^T
\end{bmatrix} ;
\\
\hskip 10mm (W,C,B,D) \in \calW \times \calC \times \Mat_{2n-2r,r}(\F) \times \Mata_r(\F)\Biggr\}.
\end{multline*}
Note that $\calW \triangle \calC$ has trivial spectrum if and only if both $\calW$ and $\calC$ have trivial spectrum.
If $\dim \calW=\dbinom{r}{2}$ and $\dim \calC=(n-r)(n-r-1)$, one checks that $\dim (\calW \triangle \calC)=n(n-1)$;
if in addition $|\F| > 2n-2$ it follows that $\calW \triangle \calC$ represents, in the basis $\bfB$, an optimal trivial spectrum subspace of
$\calA_s$.

Now, we can finally state the counterpart of Theorem \ref{theo:separationtheorem} for trivial spectrum subspaces:

\begin{theo}[Core-Wing Decomposition Theorem for trivial spectrum subspaces]\label{theo:separationtheoremTspectrum}
Let $(V,s)$ be a symplectic space of dimension $2n$, with $|\F|> 2n-2$.
Let $\calS$ be an optimal trivial spectrum subspace of $\calA_s$.

Then there exists an integer $r \in \lcro 0,n\rcro$, a mixed symplectic basis $\bfB$ of $(V,s)$ with index $r$,
an optimal trivial spectrum subspace $\calW$ of $\Mat_r(\F)$, and an \emph{irreducible} optimal trivial spectrum subspace $\calC$ of
$\calA_{s_{2n-2r}}$ (identified with $K_{2n-2r}^{-1} \Mata_{2n-2r}(\F)$), such that
$\calS$ is represented in $\bfB$ by $\calW \triangle \calC$.

The following data is then uniquely determined by $\calS$: the integer $r$, the conjugacy class of $\calW$
(for the action of $\GL_r(\F)$), and the orbit $\calC$ under conjugation of the symplectic group $\Sp_{2n-2r}(\F)$.
\end{theo}

The geometric meaning of Theorem \ref{theo:separationtheoremTspectrum} will be explained in Section \ref{section:invariantsubspaces}.

\subsection{The structure of irreducible optimal solutions}\label{section:irroptimalintroduction}

It is now time to explain the structure of the irreducible optimal spaces.
We start with the situation of symplectic affine subspaces of $\calA^2(V)$, which is easier to describe than the one of trivial
spectrum subspaces.
Note first that if $\dim V=2$ then $\calA^2(V)$ has dimension $1$, and its optimal symplectic affine subspaces are simply the singletons.
In particular none is irreducible! If $V$ has dimension $0$ then $\{0\}$ is the sole irreducible optimal symplectic affine subspace
of $\calA^2(V)$. Hence, in the remainder of the section we will always assume that $\dim V \geq 4$.

Throughout, we fix a quadratic extension $\L \subset \End(V)$ of the $\F$-algebra $\F \id_V$ (naturally identified with $\F$).
Note already that an extension of this sort exists only if $\F$ is not quadratically closed.
We may naturally view $V$ as an $\L$-vector space, and when we do so we write $V^\L$ for $V$.

\begin{Def}
We denote by $W_\L$ the set of all bilinear forms $s$ on $V^2$ such that
$$\forall x \in V, \; \forall \lambda \in \L, \; s(x,\lambda(x))=0.$$
In particular, the case $\lambda=\id_V$ shows that all the elements of $W_\L$ are alternating forms.
\end{Def}

We will now give some results without proof (the proofs are mostly found in Section \ref{section:WL}).
First of all, $W_\L$ is a linear subspace of $\calA^2(V)$ with dimension $n(n-1)$ (see Proposition \ref{cor:dimWL}).
Next, let $s_0 \in \calA^2(V)$. The following conditions are equivalent (see Proposition \ref{prop:dirWL}):
\begin{enumerate}[(i)]
\item All the forms in $s_0+W_\L$ are symplectic;
\item For all $\lambda \in \L \setminus \F$ and all $x \in V \setminus \{0\}$, one has $s_0(x,\lambda x) \neq 0$;
\item There exists $\lambda \in \L \setminus \F$ such that $(x,y) \mapsto s_0(x,\lambda y)$ is nonisotropic.
\end{enumerate}

In other words, say that we start from $a \in \End(V)$ whose minimal polynomial $p$ is
irreducible with degree $2$ (we say that $a$ is a \textbf{quadratic endomorphism} with irreducible minimal polynomial), and we consider $\L:=\F[a]=\F \id_V\oplus \F a$.
Then $W_\L$ consists of all the alternating bilinear forms $s$ on $V$ such that $(x,y)\mapsto s(x,a(y))$ is alternating.
Moreover, given $s_0 \in \calA^2(V)$, the affine subspace $s_0+W_\L$ is symplectic if and only if $(x,y) \mapsto s_0(x,a(y))$ is nonisotropic.

We also have the important result that the data of $W_\L$ is enough to recover~$\L$.

\begin{prop}\label{prop:uniquenessL}
Let $\L$ and $\L'$ be two quadratic extensions of $\F$ inside $\End(V)$
such that $W_\L=W_{\L'}$. Then $\L=\L'$.
\end{prop}

Now we can state our classification theorems.

\begin{theo}\label{theo:maintheoVL}
Let $V$ be a vector space with dimension $2n$ over $\F$, with $n \geq 2$.
Let $\L=\F[a] \subset \End(V)$ be a quadratic extension of $\F$ and $s_0 \in \calA^2(V)$ be such that
$(x,y) \mapsto s_0(x,a(x))$ is nonisotropic. Then $s_0+W_\L$ is an irreducible optimal symplectic affine subspace of $\calA^2(V)$ with dimension $n(n-1)$.
\end{theo}

Note that a space of the previous type does not exist if $\F$ is finite, and more generally if the u-invariant of $\F$
is less than $3$, as here $\dim V \geq 4$ and the existence would require that some quadratic form on $V$ is nonisotropic.
This explains why we made no cardinality assumption on $\F$ in Theorem \ref{theo:maintheoVL}.

The most difficult part of the present article is of course the converse statement:

\begin{theo}\label{theo:maintheo}
Let $V$ be a vector space with dimension $2n$ over $\F$, with $|\F| > 2n-2$ and $n \geq 2$.
Let $\calS$ be an irreducible optimal symplectic affine subspace of $\calA^2(V)$, and let $s_0 \in \calS$.
Then $\calS=s_0+W_\L$ for some quadratic extension $\L=\F[a] \subset \End(V)$ of $\F$
such that $(x,y) \mapsto s_0(x,a(y))$ is nonisotropic.
\end{theo}

Now, it is high time we gave examples. Say that there exists a (nonisotropic) quadratic form $q$ on an $n$-dimensional $\F$-vector space $E$, with $n>1$.
Assume that there is a scalar $\alpha \in \F \setminus \{0\}$ such that the form $q \bot (-\alpha) q$ is nonisotropic.
In particular $\alpha$ is not a square in $\F$, so $t^2-\alpha$ is irreducible over $\F$.
We consider the space $V:=E^2$ equipped with the endomorphism $a : (x,y) \mapsto (\alpha y,x)$, so that
$a^2=\alpha \id$. Since $\alpha$ is not a square, $\L:=\F[a]$ is a quadratic extension of $\F$ in $\End(V)$.
Now, take an arbitrary bilinear form $B$ on $E$ such that $B(x,x)=q(x)$ for all $x \in V$ (we do not take the polar form of $q$, because of fields with characteristic $2$).
Finally, we take the alternating form
$$s_0 : ((x,y),(x',y')) \mapsto B(x,y')-B(x',y).$$
Then $s_0((x,y),a(x,y))=q(x)-\alpha q(y)$ for all $(x,y)\in V$, and hence the quadratic form $(x,y) \mapsto s_0((x,y),a(x,y))$ is nonisotropic.
Therefore $s_0+W_\L$ is an irreducible affine subspace of $\calA^2(V)$ with dimension $n(n-1)$.

Here is a restatement in matrix terms:
we take $Q \in \GL_n(\F)$ and $\alpha \in \F \setminus \{0\}$ such that, for $q : X \mapsto X^T QX$, the form
$q \bot(-\alpha q)$ is nonisotropic. Then we take
$$A:=\begin{bmatrix}
0 & \alpha I_n \\
I_n & 0
\end{bmatrix} \quad \text{and} \quad K_0:=\begin{bmatrix}
0 & Q \\
-Q^T & 0
\end{bmatrix},$$
and we consider the algebra $\L:=\F[A]$, naturally identified with a subalgebra of $\End(\F^{2n})$.
Then we take the symplectic form $s_0$ whose Gram matrix in the standard basis is $K_0$.
The space $W_\L$ corresponds to the matrix space of all matrices of the form
$$\begin{bmatrix}
B  & C \\
-C^T  & -\alpha B^T
\end{bmatrix} \quad \text{with $B \in \Mata_n(\F)$ and $C \in \Mata_n(\F)$.}$$
Thanks to this shape, we immediately see that $\dim W_\L=n(n-1)$, and the affine subspace of matrices that corresponds to
$s_0+W_\L$ is simply the one of all matrices of the form
$$\begin{bmatrix}
B  & (Q+C) \\
-(Q+C)^T  & -\alpha B^T
\end{bmatrix} \quad \text{with $B \in \Mata_n(\F)$ and $C \in \Mata_n(\F)$.}$$
As an example, take $\F=\R$, $Q=I_n$ and $\alpha=-1$: in that case $q \bot (-\alpha q)$ is the standard quadratic form on $\R^{2n}$
(with matrix $I_{2n}$) in the standard basis, and $W_\L$ corresponds to the space of all matrices that are both
alternating and $K_{2n}$-alternating! The space $s_0+W_\L$ then has the very simple matrix form
$$\calU_n:=\left\{\begin{bmatrix}
B & (I_n+C) \\
-(I_n+C)^T  & B^T
\end{bmatrix} \mid (B,C) \in \Mata_n(\R)^2\right\}.$$
This corresponds to the example displayed on page \pageref{example:real}.

Our last result is quite spectacular, but limited to the case where $\L$ is a separable extension of $\F$.
Remember that in that case a skew-Hermitian form on an $\L$-vector space $W$ is defined as a mapping $h : W \times W \rightarrow \L$
such that $h(x,-)$ is linear for all $x \in W$, and $h(-,x)=-h(x,-)^\sigma$ for all $x \in W$,
where $\sigma$ denotes the non-identity automorphism in $\Gal(\L/\F)$.

\begin{theo}[Refined classification theorem in the separable case]\label{theo:refinedseparable}
Let $V$ be a vector space with dimension $2n \geq 4$ over $\F$.
Let $\L \subset \End(V)$ be a quadratic extension of $\F$.
\begin{enumerate}[(a)]
\item Let $\calS$ be an affine subspace of $\calA^2(V)$ with translation vector space $W_\L$.
Then there is a unique skew-Hermitian form $h$ on the $\L$-vector space $V^\L$
such that
$$\forall s \in \calS, \; \forall x \in V, \; \forall \lambda \in \L,\; s(x,\lambda x)=\tr_{\L/\F}\bigl(h(x,\lambda x)\bigr).$$
\item Conversely, for every skew-Hermitian form $h$ on the $\L$-vector space $V^\L$, the set of all
$s$ in $\calS$ such that $\forall x \in V, \; \forall \lambda \in \L,\; s(x,\lambda x)=\tr_{\L/\F}(h(x,\lambda x))$
is an affine subspace of $\calA^2(V)$ with translation vector space $W_\L$, and it is symplectic if and only if $h$
is nonisotropic.
\item Let $\L' \subseteq \End(V)$ be another quadratic extension of $\F$.
Two affine subspaces $\calS$ and $\calS'$ of $\calA^2(V)$ with respective translation vector spaces $W_\L$ and $W_{\L'}$
are congruent if and only if $\L'$ is isomorphic to $\L$ as an $\F$-algebra and the corresponding skew-Hermitian forms $h$ and $h'$ are equivalent up to multiplication by $-1$.
\end{enumerate}
\end{theo}

Note that since $V$ is even-dimensional every quadratic extension of $\F$ is isomorphic to a subalgebra of $\End(V)$.
Hence, in case $|\F| > 2n-2$ and $\car(\F) \neq 2$, classifying the irreducible optimal symplectic affine subspaces of $\calA^2(V)$
amounts to classifying, over each quadratic extension of $\F$, the nonisotropic Hermitian forms of rank $n$ up to equivalence and multiplication by $-1$
(indeed, classifying skew-Hermitian forms is equivalent to classifying Hermitian forms).
For example, if $\F=\R$ then there is exactly one quadratic extension of $\F$ up to isomorphism (namely $\C$), and on $\C$ there is exactly one nonisotropic
Hermitian form of rank $n$ up to equivalence and multiplication by $-1$. Hence, in $\Mata_{2n}(\R)$ all the irreducible optimal symplectic affine subspaces are
congruent to the space $\calU_n$ described on page \pageref{example:real}.

Here is the matrix version of the previous theorem over fields with characteristic other than $2$:

\begin{theo}[Refined classification theorem: matrix version]\label{theo:refinedmatrix}
Let $n \geq 2$. Assume that $\car(\F) \neq 2$.
\begin{enumerate}[(a)]
\item Let $\alpha$ be a non-square in $\F$, and
$S$ be a symmetric matrix of $\Mats_n(\F)$ such that for $q_S : X \mapsto X^TSX$, the quadratic form $q_S \bot (-\alpha q_S)$ is nonisotropic.
Then
$$\calU_{S,\alpha}:=\left\{\begin{bmatrix}
A  & S+B \\
-(S+B)^T  & -\alpha A^T
\end{bmatrix} \mid (A,B) \in \Mata_n(\F)^2\right\}$$
is an irreducible optimal symplectic affine subspace of $\Mata_{2n}(\F)$.
\item If $|\F| > 2n-2$ then
for every irreducible optimal symplectic affine subspace $\calM$ of $\Mata_{2n}(\F)$, there exists
a non-square $\alpha$ in $\F$ and a symmetric matrix $S \in \Mats_n(\F)$ such that $q_S \bot (-\alpha q_S)$ is nonisotropic
and $\calM \cong \calU_{S,\alpha}$.

\item Let $\alpha'$ be a non-square in $\F$ and $S' \in \Mats_n(\F)$ be such that $q_{S'} \bot (-\alpha' q_{S'})$ is nonisotropic.
For $\calU_{S,\alpha}$ to be congruent to $\calU_{S',\alpha'}$, it is necessary and sufficient that
the polynomials $t^2-\alpha$ and $t^2-\alpha'$ have a common splitting field $\L$ and that over such a splitting field
the matrix $S'$ be star-congruent to $\pm S$ (i.e., there exists $P \in \GL_n(\L)$ such that $S'=\pm (P^\sigma)^T S P$, where
$\sigma$ denotes the non-identity involution of $\L$ over $\F$).
\end{enumerate}
\end{theo}

Let us finish with the viewpoint of trivial spectrum subspaces of $\calA_s$, for a fixed symplectic form $s$ on $V$.
We still take a quadratic extension $\L$ of $\F$ in $\End(V)$, and write $\L=\F[a]$.
Then we consider the space $\calA_s \cap \calA_{s_a}$ where, for all $\lambda \in \L$, we set
$$s_\lambda : (x,y) \mapsto s(x,\lambda(y)).$$

\begin{prop}\label{prop:UL}
Assume that $\dim V=2n$ with $|\F| > 2n-2$.
For $\calA_s \cap \calA_{s_a}$ to be an optimal trivial spectrum subspace of $\calA_s$, it is necessary and sufficient that
$s_a$ be nonisotropic.
\end{prop}

\begin{theo}\label{theo:classoptimalirrTspectrum}
Let $(V,s)$ be a symplectic space with $\dim V=2n$ with $|\F| > 2n-2$ and $n \geq 2$.
Let $S$ be an irreducible optimal trivial spectrum subspace of $\calA_s$. Then there exists a quadratic field extension $\L=\F[a]$
of $\F$ in $\End(V)$ such that
$$S=\calA_s \cap \calA_{s_a}=\underset{\lambda \in \L}{\bigcap}\calA_{s_\lambda}$$
and $s_a$ is nonisotropic.

Moreover, $\L$ is uniquely determined by $S$.
Finally, two such spaces $S$ and $S'$ are conjugated through an element of $\Sp(s)$
if and only if the corresponding extensions $\L$ and $\L'$ are conjugated through an element of the symplectic group $\Sp(s)$.
\end{theo}

A full understanding of the irreducible optimal affine subspaces of $\calA^2(V)$ would essentially require that we understand the
pairs $(s_0,a)$ consisting of a symplectic form $s_0$ and of a quadratic endomorphism $a \in \End(V)$ with irreducible minimal polynomial,
subject to the condition that $(s_0)_a$ is nonisotropic, up to equivalence
(two such pairs $(s_0,a)$ and $(s'_0,a')$ are called equivalent when there exists an automorphism $\varphi$ of $V$ such that
$\forall (x,y)\in V^2, \; s_0(x,y)=s'_0(\varphi(x),\varphi(y))$ and $a'=\varphi a \varphi^{-1}$).
We suspect that such a classification is possible over fields with characteristic other than $2$, but it would steer us very far from the techniques
that are used in this article, and hence we prefer to leave this for future work.

\subsection{Strategy, and structure of the article}\label{section:strategy}

At this point, we have announced much and proved little. The remainder of the article is of course mainly concerned with proofs
of the results stated earlier, but also of alternative viewpoints on some results. Let us here describe the organization of the remaining sections.

Section \ref{section:bilinbasics} consists of several basic definitions related to bilinear forms
(not only the alternating ones), as well as elementary considerations on $s$-alternating endomorphisms with small rank.

Section \ref{section:reduction} is devoted to the reduction of the general case of optimal trivial spectrum subspaces of $\calA_s$
(respectively, of optimal affine subspaces of symplectic forms on $V$) to the irreducible case.
There, of course Theorems \ref{theo:separationtheorem} and \ref{theo:separationtheoremTspectrum} are proved, but we
also delve much deeper into the meaning of those results, deciphering the structure of the lattice of invariant subspaces for an
optimal trivial spectrum subspace of $\calA_s$. One of the keys is to prove that if such a subspace is reducible, then it almost always
has a greatest nontrivial totally $s$-singular invariant subspace.

Next, Section \ref{section:WL} is devoted to a thorough analysis of the $W_\L$ spaces that we have introduced in Section \ref{section:irroptimalintroduction}.
There, by giving a different viewpoint on the construction of these spaces, as related to alternating bilinear forms on the $\L$-vector space $V^\L$,
we discuss the structure of the affine subspaces with translation vector space $W_\L$ that contain only symplectic forms, and we classify these up to congruence.
In particular, this yields Theorem \ref{theo:maintheoVL}, Theorem \ref{theo:refinedseparable}, as well as points (a) and (c) of Theorem
\ref{theo:refinedmatrix}. The fact that $\L$ is determined by $W_\L$ (Proposition \ref{prop:uniquenessL}) is proved there as Proposition \ref{prop:recoverL}.
The end of Section \ref{section:WL} is devoted to the analogue classification issues for trivial spectrum subspaces.

At this point we will have only played with the examples we have considered, and it will still remain to deal with arbitrary
irreducible optimal symplectic subspaces of $\calA^2(V)$. This analysis, which is by far the most difficult part of the present article,
is performed in Sections \ref{section:froma2dimalternator} to \ref{section:lastkeyprop}.
First of all, for the proof we exclusively use the trivial spectrum spaces viewpoint, i.e., we try to classify the optimal trivial spectrum \emph{linear} subspaces $S$ of $\calA_s$ where $s$ is a fixed symplectic form on $V$. Just like in the case of spaces of plain endomorphisms, a key part is played
by the notion of an \textbf{alternator} of such an operator space:
an alternator of $S$ is simply a bilinear form $b$ on $V$ (not necessarily symmetric!) such that $S \subseteq \calA_b$, and the set of all such
forms is denoted by $\Alt(S)$. Of course $s$ is an alternator of $S$. The key is then to proceed in two steps, that are carried out independently:
The plan is to show that if $S$ is an irreducible optimal trivial spectrum subspace of $\calA_s$:
\begin{enumerate}[(a)]
\item $\Alt(S)$ has dimension $2$ and all its nonzero elements are nondegenerate;
\item Then $\Alt(S)=\{(x,y) \mapsto s(x,\lambda (y)) \mid \lambda \in \L\}$ for some quadratic extension $\L$ of $\F$ in $\End(V)$,
such that $s(x,\lambda x) \neq 0$ for all $\lambda \in \L \setminus \F$ and all $x \in V \setminus \{0\}$,
and $S=\underset{b \in \Alt(S)}{\bigcap} \calA_b$.
\end{enumerate}
From these points the first part of Theorem \ref{theo:classoptimalirrTspectrum} is derived,
Theorem \ref{theo:maintheo} ensues, and then Theorem \ref{theo:refinedmatrix} follows.

Part (b) in this plan turns out to be the easier one, and it can be essentially carried out independently of part (a).
Hence we will present part (b) first, in Section \ref{section:froma2dimalternator}.
The very difficult part is (a): This is performed by
induction on the dimension of $V$. The idea is to go back to the proof of Theorem \ref{theo:majdim} featured in \cite{dSPaffinealt}, which involves the introduction of the dual operator space $\widehat{\calS}$ of $\calS$, defined as the set of all evaluation mapping
$$\widehat{x} : u \in S \longmapsto u(x) \in V, \quad \text{with $x \in V$.}$$
It is easily proved that the greatest possible rank in $\widehat{\calS}$ is $\dim V-2$,
and a great part of the analysis is spent using this rank constraint, trying to extract as much meaningful information as possible from it,
with the trivial spectrum assumption used at some critical steps.
This analysis requires the machinery of generic matrices, a method we borrow from Atkinson \cite{AtkinsonPrim}
and which gives efficient matrix identities: we devote most of Section \ref{section:tools}
to recall the basics of this method, along with a new related technical result (the Second Factorization Lemma).
Using these methods, and in particular thanks to a dictionary between generic matrices and alternators (Section \ref{section:generic2}),
we are able, after a very intricate analysis, to complete part (a) of the plan.

In the last section of the article, we will show how the results from the present article allow for a quick proof of
the structure of optimal vector spaces of nilpotent $s$-alternating endomorphisms
featured in \cite{structuredGerstenhaber1}, with the exclusion of small finite fields but with the characteristic $2$ case taken into account.

\section{Basics on bilinear forms}\label{section:bilinbasics}

In this short section, we collect some basic facts and notation.

\subsection{Radicals, left and right maps}

Let $b : V^2 \rightarrow \F$ be a bilinear form. We denote by $V^\star$ the dual vector space of $V$, and we
introduce the left and right associated linear mappings
$$L_b : x \in V \mapsto b(x,-) \in V^\star \quad \text{and} \quad R_b :  y \in V \mapsto b(-,y) \in V^\star.$$
In general we must differentiate the \textbf{left radical}
$$\Lrad(b):=\Ker L_b=\{x \in V : \forall y \in V, \; b(x,y)=0\}$$
from the \textbf{right radical}
$$\Rrad(b):=\Ker R_b=\{y \in V : \forall x \in V, \; b(x,y)=0\}.$$
We say that the form $b$ is \textbf{straight} when $\Lrad(b)=\Rrad(b)$.
An important remark, which will be used repeatedly, is that every $u \in \calA_b$ maps the left radical $\Lrad(b)$ into the right radical $\Rrad(b)$:
indeed for every such $u$ the form $(x,y) \mapsto b(x,u(y))$ is skew-symmetric, and hence for all $x \in \Lrad(b)$ and all $y \in V$,
$b(y,u(x))=-b(x,u(y))=0$.

\subsection{Alternators}

\begin{Def}
Let $S$ be a subset of $\End(V)$. An \textbf{alternator} of $S$ is a bilinear form $b : V^2 \rightarrow \F$ such that
$S \subseteq \calA_b$, i.e., $\forall x \in V, \; \forall u \in S, \; b(x,u(x))=0$.

We define the \textbf{alternator set} of $S$ as the set of all alternators of $S$, i.e.,
$$\Alt(S)=\{b : V \times V \rightarrow \F : \; b \; \text{bilinear and}\; u \in \calA_b\}.$$
It is obviously a linear subspace of the space $\calB(V^2,\F)$ of all bilinear forms on $V$.

Conversely, given a subset $\calB$ of bilinear forms on $V$, we write
$$\calA_\calB=\underset{b \in \calB}{\bigcap} \calA_b,$$
and again it is clearly a linear subspace of $\End(V)$.
\end{Def}

Note the obvious equivalence, for a subset $\calS$ of $\End(V)$ and a subset $\calB$ of bilinear forms on $V$:
$$\calB \subseteq \Alt(\calS) \Leftrightarrow \calS \subseteq \calA_\calB,$$
so $\calB \subseteq \Alt(\calA_\calB)$ and $\calS \subseteq \calA_{\Alt(\calS)}$. In general, these inclusions are not equalities,
even when $\calB$ and $\calS$ are linear subspaces.

\subsection{Basic remarks on the elements of $\calA_s$}

Let $s$ be a \emph{symplectic} form on $V$. Then every $u \in \calA_s$ is $s$-selfadjoint, meaning that
$\forall (x,y)\in V^2, \; s(u(x),y)=s(x,u(y))$ (note that the converse holds if $\car(\F) \neq 2$, but not if $\car(\F)=2$ and this is crucial).
As a consequence every such $u$ has all the nice features of selfadjoint endomorphisms with respect to orthogonality:
its kernel is the $s$-orthogonal complement of its range, if it maps a subspace $V_1$ into a subspace $V_2$ then it maps $V_2^{\bot_s}$ into $V_1^{\bot_s}$,
and in particular if it leaves a subspace $W$ invariant then it also leaves $W^{\bot_s}$ invariant.

\subsection{Alternating $2$-tensors}

\begin{Def}
Let $b$ be a bilinear form on $V$, and $x,y$ be two vectors of $V$. We set
$$x \wedge_b y : z \mapsto b(z,y)\,x-b(z,x)\,y,$$
and call it the \textbf{$b$-alternating tensor} of $x$ and $y$.
\end{Def}

One easily checks that $x \wedge_b y$ is $b$-alternating. Its range is included in $\Vect(x,y)$.
If $x$ and $y$ are linearly dependent, then $x \wedge_b y=0$. Otherwise, and if $b$ is nondegenerate, the bilinear forms
$b(-,y)$ and $b(-,x)$ are linearly independent and one deduces that $\im(x \wedge_b y)=\Vect(x,y)$.

\begin{prop}\label{prop:alttensor}
Let $b$ be a non-degenerate bilinear form on $V$.
Let $x,y$ be linearly independent vectors in $V$. Then the elements of $\calA_b$ with range included in
$\Vect(x,y)$ are the scalar multiples of $x \wedge_b y$.
\end{prop}

\begin{proof}
We have just seen that every scalar multiple of $x \wedge_b y$ belongs to $\calA_b$ and has its range included in $\Vect(x,y)$
(and even equal to $\Vect(x,y)$ if nonzero).

Let $u \in \calA_b$ satisfy $\im u \subseteq \Vect(x,y)$. Then $u: z \mapsto \varphi(z)\,x+\psi(z)\,y$
for linear forms $\varphi$ and $\psi$ on $V$. Since $u$ is $b$-alternating we obtain
$\varphi\, b(-,x)+\psi\, b(-,y)=0$.

Hence $\psi(z)b(z,y)=0$ for all $z \in H:=\Ker b(-,x)$. Note that $H \cap \Ker b(-,y)$ is a proper subspace of $H$ because
$b(-,y)$ is not a scalar multiple of $b(-,x)$. Hence the linear form $\psi$ vanishes on $H \setminus (H \cap \Ker b(-,y))$, which is a spanning
subset of $H$. It ensues that $\psi=\lambda b(-,x)$ for some $\lambda \in \F$. Replacing $u$ with $u+\lambda x \wedge_b y$ does not change the basic assumptions,
and in that reduced situation $\psi=0$, which leads to $\varphi\, b(-,x)=0$. Thus $V$ is the union of the linear subspaces $\Ker b(-,x)$ and $\Ker \varphi$,
so one of them equals $V$, and hence $\varphi=0$. Therefore $u=0$ in the reduced situation, which proves the claimed statement.
\end{proof}

\section{The core-wing decomposition, and the invariant subspaces of an optimal trivial spectrum subspace}\label{section:reduction}

Our goal here is to establish the core-wing decompositions featured in Theorems \ref{theo:separationtheorem} and \ref{theo:separationtheoremTspectrum}.
We will go much further by giving a clear geometric viewpoint, which involves an explanation on the possible
adapted subspaces for an optimal affine space of symplectic forms, and the possible invariant subspaces for an optimal trivial spectrum linear subspace of $\calA_s$.

In order to study the invariant subspaces, it is useful to recall the first part of Theorem \ref{theo:matrixcase}:
if $S$ is an optimal trivial spectrum subspace of $\End(V)$ for some vector space $V$, then the
$S$-invariant subspaces are totally ordered by inclusion, i.e., they form a (potentially incomplete) flag of $V$.
We refer to the appendix of \cite{dSPtriangularizable} for a generalization of this observation.

For an optimal trivial spectrum subspace $S$ of $\calA_s$ (where $s$ is a symplectic form), the
$S$-invariant subspaces are not always totally ordered. First of all, there is the very obvious counter-example
where $s$ has rank $2$: in that case $\calA_s=\F \id_V$ and hence every subspace of $V$ is $S$-invariant.
There is also a generalization of this counter-example, but it is too early to explain it. Fortunately, in these special cases the description
of the $S$-invariant subspaces is not very complicated.

\subsection{A preliminary study: how to use a totally $s$-singular invariant subspace}\label{section:totallysingularinvariant}

In this section, we consider a fixed symplectic space $(V,s)$ with dimension $2n>2$, and we assume that $|\F| > 2n-2$.
First of all, let $W$ be an arbitrary totally $s$-singular subspace of $V$, with dimension denoted by $m$.
Let us consider a mixed symplectic basis $\bfB=(e_1,\dots,e_m,g_1,\dots,g_{2n-2m},f_1,\dots,f_m)$ of $(V,s)$
that is adapted to $W$. Hence $e_1,\dots,g_{2n-2m}$ span $W^{\bot_s}$.
As seen in Section \ref{section:intromatrix}, the elements of $\calA_s$ that leave $W$ invariant are those that are represented in $\bfB$
by a matrix of the form
$$\begin{bmatrix}
M & (K_{2n-2m} B)^T & D \\
[0]_{(2n-2m) \times m} & K_{2n-2m}^{-1} C & B \\
[0]_{m \times m} & [0]_{m \times (2n-2m)} & M^T
\end{bmatrix}$$
with $M \in \Mat_m(\F)$, $B \in \Mat_{2n-2m,m}(\F)$, $C \in \Mata_{2n-2m}(\F)$ and $D \in \Mata_m(\F)$.
In particular, the elements of $\calA_s$ that vanish on $W$ and map $W^{\bot_s}$ into $W$ are those that are
represented in $\bfB$
by a matrix of the form
$$\begin{bmatrix}
[0]_{m \times m} & (K_{2n-2m} B)^T & D \\
[0]_{(2n-2m) \times m} & [0]_{(2n-2m) \times (2n-2m)} & B \\
[0]_{m \times m} & [0]_{m \times (2n-2m)} & [0]_{m \times m}
\end{bmatrix}$$
with $B \in \Mat_{2n-2m,m}(\F)$ and $D \in \Mata_m(\F)$. We have obtained the following result:

\begin{lemma}\label{lemma:NW}
Let $W$ be a totally $s$-singular subspace of $V$, with dimension denoted by $m$, and write $\dim W=2n$.
Set
$$\calN_W:= \bigl\{u \in \calA_s : \; u(W)=\{0\} \; \text{and}\; u(W^{\bot_s}) \subseteq W\bigr\}.$$
Then
$$\dim \calN_W=m(2n-m)+\dbinom{m}{2}.$$
\end{lemma}

Note that in this definition the condition $u(W)=\{0\}$ can be replaced with $u(V) \subseteq W^{\bot_s}$ because we are dealing with
$s$-alternating endomorphisms.
The following lemma will be used later in our work, and it is convenient to prove it here:

\begin{lemma}\label{lemma:NWspantensor}
Let $W$ be a totally $s$-singular subspace of $V$.
Then
$$\calN_W=\Vect\{x \wedge_s y \mid (x,y) \in W \times W^{\bot_s}\}.$$
\end{lemma}

Note, for every nonzero vector $x$ of $V$, the special case
$$\calN_{\F x}=\bigl\{x \wedge_s y \mid y \in \{x\}^{\bot_s}\bigr\}.$$

\begin{proof}
To start with, we will check that the rank $2$ elements in $\calN_W$ are exactly the $s$-alternating tensors
$x \wedge_s y$ with $(x,y) \in W \times W^{\bot_s}$. First of all, consider such a tensor.
It is $s$-alternating, and hence of rank $0$ or $2$. Hence if it is nonzero its range equals $\Vect(x,y)$, which is included in $W^{\bot_s}$,
and then every $z \in W^{\bot_s}$ is orthogonal to $x$ and hence satisfies $(x \wedge_s y)(z)=s(z,y)\, x \in W$.

Conversely, let $u \in \calN_W$ have rank $2$. Denote by $P$ its range, which must be included in $W^{\bot_s}$.
Assume that $P \cap W=\{0\}$. Choose a basis $(x,y)$ of $P$ such that $u=x \wedge_s y$ (this is possible, by Proposition \ref{prop:alttensor}).
Then for all $z \in W^{\bot_s}$ we must have $(x \wedge_s y)(z) \in P \cap W=\{0\}$, leading to
$s(z,x)=s(z,y)=0$. It follows by double-orthogonality that $P \subseteq W$, which is absurd.
Hence $P$ contains at least one nonzero vector $x$ of $W$, and we can use Proposition \ref{prop:alttensor} once more to find that $u=x \wedge_s y$ for some
$y \in W^{\bot_s}$.

In order to conclude, it suffices to prove that $\calN_W$ is spanned by its rank $2$ elements. To see this, take a mixed
symplectic basis $(e_1,\dots,e_r,g_1,\dots,g_{2n-2r},f_1,\dots,f_r)$ that is adapted to $W$.
Then $\calN_W$ is represented in this basis by the set of all matrices
$$\begin{bmatrix}
[0]_{m \times m} & (K_{2n-2m} B)^T & D \\
[0]_{(2n-2m) \times m} & [0]_{(2n-2m) \times (2n-2m)} & B \\
[0]_{m \times m} & [0]_{m \times (2n-2m)} & [0]_{m \times m}
\end{bmatrix}$$
with $B \in \Mat_{2n-2m,m}(\F)$ and $D \in \Mata_m(\F)$.
When we take $B$ of rank $1$ and $D=0$, we obtain a rank $2$ matrix. When we take $D$ of rank $2$ and $B=0$, we obtain a rank $2$ matrix.
The conclusion follows by noting that $\Mat_{2n-2m,m}(\F)$ is spanned by its rank $1$ matrices, while $\Mata_m(\F)$ is spanned by its rank $2$ matrices.
\end{proof}

Now, we introduce a notation that is the geometric counterpart of the space $\calM \triangle \calA$ introduced in Notation \ref{not:bullet}.
We keep a totally $s$-singular subspace $W$ of $V$, and we denote by $\overline{s}$ the symplectic form induced by $s$ on $W^{\bot_s}/W$.
Note that every $u \in \calA_s$ that leaves $W$ invariant also leaves $W^{\bot_s}$ invariant, and hence induces an endomorphism
$u_W$ of $W$ and an $\overline{s}$-symplectic endomorphism $\overline{u}^W$ of $W^{\bot_s}/W$.

Now, take a linear subspace $\calM$ of $\End(W)$ and a linear subspace $\calA$ of $\calA_{\overline{s}}$, and set
$$\calM \triangle \calA:=\bigl\{u \in \calA_s : u(W) \subseteq W, \; u_W \in \calM \; \text{and} \; \overline{u}^W \in \calA\bigr\}.$$
In particular $\calN_W \subseteq \calM \triangle \calA$.
Now, take a mixed symplectic basis
$$\bfB=(e_1,\dots,e_m,g_1,\dots,g_{2n-2m},f_1,\dots,f_m)$$
of $(V,s)$ that is adapted to $W$,
and denote respectively by $\calM_1$ and $\calA_1$ the spaces of matrices that represent $\calM$ and $\calA$, respectively in
$(e_1,\dots,e_m)$ and the projected basis $(\overline{g_1},\dots,\overline{g_{2n-2m}})$ of $W^{\bot_s}/W$ (which is $\overline{s}$-symplectic).
Then it is clear that $\calM \triangle \calA$ is represented by the matrix space $\calM_1 \triangle \calA_1$ in $\bfB$.
In particular we derive that $\calM \triangle \calA$ has trivial spectrum if and only if both $\calM$ and $\calA$ have trivial spectrum:
note indeed, from the above matrix decomposition, that the matrix
$$\begin{bmatrix}
M & (K_{2n-2m} B)^T & D \\
0 & K_{2n-2m}^{-1} C & B \\
0 & 0 & M^T
\end{bmatrix}$$
has trivial spectrum if and only if the three diagonal cells $M,K_{2n-2m}^{-1} C,M^T$ have trivial spectrum, i.e.,
$M$ and $K_{2n-2m}^{-1} C$ have trivial spectrum. In geometric terms, this means that an endomorphism $u \in \calA_s$ that leaves $W$ invariant has trivial spectrum
if and only if both $u_W$ and $\overline{u}^W$ have trivial spectrum.
Besides, it is easily seen that
$$\dim(\calM \triangle \calA)=\dim \calM+\dim \calA+\dim \calN_W=\dim \calM+\dim \calA+m(2n-2m)+\dbinom{m}{2}.$$
An immediate consequence (the computation is straightforward) is that if $\dim \calM=\dbinom{m}{2}$ and
$\dim \calA=(n-m)(n-m-1)$ then
$$\dim(\calM \triangle \calA)=n(n-1).$$

Let us conclude:

\begin{prop}
Assume that $|\F|>\dim V-2$.
Let $W$ be a (non-zero) totally $s$-singular subspace of $V$, and denote by $\overline{s}$
the symplectic form induced by $s$ on $W^{\bot_s}/W$. Take
a linear subspace $\calM$ of $\End(W)$ and a linear subspace $\calA$ of $\calA_{\overline{s}}$.
Then $\calM \triangle \calA$ is an optimal trivial spectrum subspace of $\calA_s$ that leaves $W$ invariant.
\end{prop}

Here is a converse statement:

\begin{prop}\label{prop:decompbase}
Let $S$ be an optimal trivial spectrum linear subspace of $\calA_s$. Let $W$ be a (non-zero) totally $s$-singular subspace that is
$S$-invariant. Set $S_W:=\{u_W \mid u \in S\}$ and $\overline{S}^W:=\{\overline{u}^W \mid u \in S\}$.
Then:
\begin{enumerate}[(i)]
\item $S_W$ is an optimal trivial spectrum subspace of $\End(W)$;
\item $\overline{S}^W$ is an optimal trivial spectrum subspace of $\calA_{\overline{s}}$.
\item $S=S_W \triangle \overline{S}^W$, and as a consequence $\calN_W \subseteq S$.
\end{enumerate}
\end{prop}

\begin{proof}
We have $S \subseteq S_W \triangle \overline{S}^W$, and
$S_W$ and $\overline{S}^W$ have trivial spectrum. Hence $S_W \triangle \overline{S}^W$ has trivial spectrum, and it is included in $\calA_s$.
Since $S$ is optimal it is maximal among the trivial spectrum subspaces of $\calA_s$, and we obtain point (iii).

Now, assume that $S_W$ is not optimal as a trivial spectrum subspace of $\End(W)$. Then there exists a trivial spectrum subspace
$\calM$ of $\End(W)$ such that $\dim \calM>\dim S_W$, leading to $\dim (\calM \triangle \overline{S}^W)>
\dim (S_W \triangle \overline{S}^W)$. Yet $\calM \triangle \overline{S}^W$ is a trivial spectrum subspace of $\calA_s$,
and we thereby contradict the optimality of $S$. Hence $S_W$ is optimal.

The proof that $\overline{S}^W$ is optimal is performed in the same way.
\end{proof}

\begin{Rem}\label{remark:inducedonquotient}
In the situation of the previous proposition, an additional space
that is worth considering is the space $S_{\text{mod}\, W^{\bot_s}}$ consisting of the endomorphisms of
$V/W^{\bot_s}$ induced by the elements of $S$. Actually, this space is simply similar to the transpose of $S_W$,
defined as $\{\varphi \mapsto \varphi \circ v \mid v \in S_W\} \subseteq \End(W^\star)$.
Here, by a similarity of respective subsets $\calX_1 \subseteq \End(V_1)$ and $\calX_2 \subseteq \End(V_2)$,
we mean that there exists a vector space isomorphism $\Phi : V_1 \overset{\simeq}{\longrightarrow} V_2$
such that $\calX_2=\{\Phi \circ v \circ \Phi^{-1} \mid v \in \calX_1\}$.
This can easily be seen in block matrix formulation. Choose indeed a
mixed symplectic basis $(e_1,\dots,e_m,g_1,\dots,g_{2n-2m},f_1,\dots,f_m)$ of $(V,s)$ that is adapted to $W$.
Let $u \in S$, and consider its representing matrix in the said basis. Its first and third diagonal blocks
equal $M$ and $M^T$ for some $M \in \Mat_m(\F)$, and $M$ represents $u_W$ in $(e_1,\dots,e_m)$, whereas
$M^T$ represents the transposed endomorphism $\varphi \in W^\star \mapsto \varphi \circ u_W \in W^\star$
in the dual basis of $(e_1,\dots,e_m)$, as well as the endomorphism of $V/W^{\bot_s}$ induced by $u$.
This yields the claimed equivalence.

As a consequence, $S_{\text{mod}\, W^{\bot_s}}$ is an optimal trivial spectrum subspace of $\End(W/W^{\bot_s})$.
\end{Rem}

Note that in Proposition \ref{prop:decompbase} we did not make any assumption on the cardinality of the underlying field.

\subsection{The invariant subspaces of an optimal trivial spectrum space: preliminary work}

We keep the main setting of the previous section: $(V,s)$ denotes a symplectic space of dimension $2n>2$,
and we assume that $|\F|> 2n-2$ throughout.
We will now examine the lattice of $S$-invariant subspaces, where $S$ is an optimal trivial spectrum subspace of $\calA_s$.

\begin{lemma}\label{lemma:noregularinvariant}
Let $S$ be an optimal trivial spectrum subspace of $\calA_s$. Then no nontrivial $S$-invariant subspace is $s$-regular.
\end{lemma}

\begin{proof}
Assume on the contrary that $S$ has a nontrivial $s$-regular invariant subspace $W$.
All the elements of $S$ leave both $W$ and $W^{\bot_s}$ invariant. Denote by $s_1$ and $s_2$ the induced symplectic forms on these spaces,
and consider the restriction mapping $\Phi : u \in S \mapsto (u_W,u_{W^{\bot_s}}) \in \calA_{s_1} \times \calA_{s_2}$.
Clearly $S_1:=\{u_W \mid u \in S\}$ and $S_2:=\{u_{W^{\bot_s}} \mid u \in S\}$ are trivial spectrum subspaces. Setting $2a=\dim W$ and $2b=\dim W^{\bot_s}$, we
deduce from Theorem \ref{theo:majodimTS} that
$$\dim S \leq a(a-1)+b(b-1)<(a+b)(a+b-1)=n(n-1)$$
because $a>0$ and $b>0$. This contradicts the optimality of $S$.
\end{proof}

\begin{lemma}\label{lemma:existtotallysingularinvariant}
Let $S$ be a reducible optimal trivial spectrum subspace of $\calA_s$. Then $S$ has a nontrivial totally $s$-singular
invariant subspace.
\end{lemma}

\begin{proof}
By assumption, there is a nontrivial $S$-invariant subspace $W$. Note that $W^{\bot_s}$ is also $S$-invariant because
$S \subseteq \calA_s$. Hence $W \cap W^{\bot_s}$ is a totally $s$-singular $S$-invariant subspace.
By Lemma \ref{lemma:noregularinvariant}, $W$ is not $s$-regular and hence $W \cap W^{\bot_s}\neq \{0\}$.
\end{proof}

\begin{lemma}\label{lemma:irrvsmaximal}
Let $S$ be an optimal trivial spectrum subspace of $\calA_s$.
Let $W$ be a totally $s$-singular $S$-invariant subspace.
Then $\overline{S}^W$ is irreducible if and only if $W$ is maximal among the totally $s$-singular $S$-invariant subspaces.
\end{lemma}

\begin{proof}
Denote by $\overline{s}$ the symplectic form induced by $s$ on $W^{\bot_s}/W$.
By Proposition \ref{prop:decompbase} we have $S=S_W \triangle \overline{S}^{W}$, and
$\overline{S}^{W}$ is an optimal trivial spectrum subspace of $\calA_{\overline{s}}$.
Assume that $\overline{S}^{W}$ is reducible. Then by Lemma \ref{lemma:existtotallysingularinvariant} it has a non-trivial totally $\overline{s}$-singular subspace $\overline{W_1}$, which
is then the projection of a totally $s$-singular subspace $W_1$ that includes $W$, with $W \neq W_1$.
But this would contradict the maximality of $W_1$. Hence $\overline{S}^{W}$ is irreducible.

Conversely, assume that $\overline{S}^{W}$ is irreducible.
Let $W'$ be a totally $s$-singular $S$-invariant subspace such that $W \subseteq W'$.
Since $W'$ is totally $s$-singular we have $W' \subseteq (W')^{\bot_s} \subseteq W^{\bot_s}$. Hence
$W'/W$ is a totally $\overline{s}$-singular subspace that is $\overline{S}^W$-invariant. Since $\overline{S}^W$
is irreducible this yields $W'/W=\{0\}$ or $W'/W=W^{\bot_s}/W$, that is $W'=W$ or $W'=W^{\bot_s}$.
Assume that the latter holds. Then $W^{\bot_s}$ is totally $s$-singular, which in turn shows that it equals $W$,
and hence $W'=W$. This proves the claimed statement.
\end{proof}

\subsection{The invariant subspaces of an optimal trivial spectrum space: the full description}\label{section:invariantsubspaces}

Now we can describe the invariant subspaces of an arbitrary optimal trivial spectrum subspace of $\calA_s$.
Remember that if $n=2$ then $\calA_s=\F \id_V$ and hence every linear subspace of $V$ is $\calA_s$-invariant.
It is now time to generalize this situation:
let $S$ be an optimal trivial spectrum subspace of $\calA_s$, and assume that there exists a totally $s$-singular
subspace $W$ that is $S$-invariant and such that $\dim W=n-1$. Hence $\dim W^{\bot_s}=n+1$,
and the space $\overline{S}^W$ is a trivial spectrum subspace of $\calA_{\overline{s}}$.
By the dimension $2$ case, every linear subspace of $W^{\bot_s}/W$ is invariant under every element of $\overline{S}^W$,
and hence by lifting those subspaces to linear subspaces of $V$ we deduce the following:
every linear subspace $H$ of $V$ such that $W \subset H \subset W^{\bot_s}$ is $S$-invariant.
And clearly every such subspace is totally $s$-singular (and it is a Lagrangian of $(V,s)$).

We shall see shortly that this is the sole exception to the general rule that says that ``invariant subspaces are totally ordered"
(a general rule that applies to several similar problems, see the appendix of \cite{dSPtriangularizable}).

\begin{lemma}\label{lemma:keylemmairr}
Let $S$ be an optimal trivial spectrum subspace of $\calA_s$.
Let $W$ be a totally $s$-singular $S$-invariant subspace such that $\overline{S}^W$ is irreducible.
For every $S$-invariant subspace $H$, one of the following three properties holds:
\begin{enumerate}[(i)]
\item $H \subseteq W$;
\item $W^{\bot_s} \subseteq H$;
\item $H$ and $W$ are two distinct $s$-Lagrangians and $H \cap W$ is a linear hyperplane in both of them.
\end{enumerate}
\end{lemma}

\begin{proof}
By point (iii) in Proposition \ref{prop:decompbase}, we have $\calN_W \subseteq S$, and this observation will be sufficient to prove the statement.

It will simplify the discourse to perform a \emph{reductio ad absurdum}, by assuming that none of possibilities (i) and (ii) holds.
Then, we will prove that (iii) holds.

Assume first that $H$ contains an element $z$ of $W^{\bot_s} \setminus W$.
Then there exists $y \in W^{\bot_s}$ such that $s(z,y) \neq 0$, and hence
for all $x \in W$ we have $(x \wedge_s y)(z)=s(z,y)x-s(z,x)y=s(z,y)\,x$, which shows that $x \in H$.
Hence $W \subseteq H$.
Next, the space $H \cap W^{\bot_s}$ is $S$-invariant, includes $W$ and is included in $W^{\bot_s}$.
Its projection modulo $W$ is then a nonzero $\overline{S}^W$-invariant subspace (because of our starting assumption that $z \in H$).
Since $\overline{S}^W$ is irreducible this projection must be $W^{\bot_s}/W$, yielding $W^{\bot_s} \subseteq H$ and contradicting the assumption
that (ii) does not hold.

We conclude that $H \cap (W^{\bot_s} \setminus W) \neq \emptyset$. Hence, as $H$ is not included in $W$
it must contain at least one vector $z \in V \setminus W^{\bot_s}$.
Then, there is a vector $x \in W$ such that $s(x,z) \neq 0$. For all $y \in W^{\bot_s}$ such that $s(y,z)=0$, we find that
$(x \wedge_s y)(z)=-s(x,z)\,y$, and hence $y \in H$. Therefore $W^{\bot_s} \cap \{z\}^{\bot_s} \subseteq H$.
Remembering that $W^{\bot_s} \cap H \subseteq W$, we also have
\begin{equation}\label{eq:inclu}
W^{\bot_s} \cap \{z\}^{\bot_s} \subseteq H \cap W.
\end{equation}
Besides $W^{\bot_s} \cap \{z\}^{\bot_s} \subsetneq W^{\bot_s}$ because $z \not\in W$.

Hence $\dim (W^{\bot_s} \cap \{z\}^{\bot_s}) = \dim (W^{\bot_s})-1$, to the effect that
$\dim W \geq \dim (W^{\bot_s})-1$. Hence $W$ is a Lagrangian, i.e., $W^{\bot_s}=W$.
The invalidity of (ii) now means that $W \not\subseteq H$ and hence $\dim (H \cap W) <\dim W$.
Using \eqref{eq:inclu} once more, we deduce that $W \cap \{z\}^{\bot_s}=H \cap W$ and that $H \cap W$ is a linear hyperplane of $W$.

It remains to prove that $H$ is also a Lagrangian. To see this, we consider the projection of $H$ on $V/W$.
Varying the choice of $z$ and remembering that $W=W^{\bot_s}$ we obtain that
$W \cap \{z\}^{\bot_s}=W \cap \{z'\}^{\bot_s}$ for all $z' \in H \setminus W$.
Because $W$ is a Lagrangian, taking orthogonal complements yields $\F z+W=\F z'+W$ for all $z' \in H \setminus W$,
and it follows that the projection of $H$ on $V/W$ has dimension $1$. Hence by the rank theorem $\dim H=1+\dim(H \cap W)=\dim W$ and
$H=\F z\oplus (W \cap \{z\}^{\bot_s})$. Clearly $\F z$ and $W \cap \{z\}^{\bot_s}$ are $s$-orthogonal, and since each one is totally $s$-singular
we deduce that $H$ is totally $s$-singular. Hence $H$ is a Lagrangian. This completes the proof.
\end{proof}

\begin{theo}\label{theo:invariantsubspaces1}
Let $S$ be an optimal trivial spectrum subspace of $\calA_s$, where $(V,s)$ is a symplectic space of dimension $2n>0$, with $|\F|>2n-2$.
Assume that there is no totally $s$-singular $S$-invariant subspace of dimension $n-1$.
Then the $S$-invariant subspaces are totally ordered. Consequently, they form a flag
$(V_0,\dots,V_p)$ in which $V_k^{\bot_s}=V_{p-k}$ for all $k \in \lcro 0,p\rcro$.
\end{theo}

\begin{proof}
By combining Lemmas \ref{lemma:existtotallysingularinvariant} and \ref{lemma:irrvsmaximal}, we find a totally $s$-singular subspace $W$ that is $S$-invariant
and such that $\overline{S}^W$ is irreducible.
Note that $S=S_W \triangle \overline{S}^W$ and remember that $S_W$ is an optimal trivial spectrum subspace of
$\End(W)$. By Theorem \ref{theo:matrixcase} its invariant subspaces form a flag $(W_0,\dots,W_q)$ of $W$, and they are of course $S$-invariant subspaces.
Hence by orthogonality $W^{\bot_s} = W_q^{\bot_s} \subsetneq W_{q-1}^{\bot_s} \subsetneq \cdots \subsetneq W_0^{\bot_s}=V$
are $S$-invariant subspaces.

Now, let $H$ be an $S$-invariant subspace. In Lemma \ref{lemma:keylemmairr}, the third outcome is ruled out by the assumption that
there is no totally $s$-singular $S$-invariant subspace of dimension $n-1$. Hence
$H \subseteq W$ or $W^{\bot_s} \subseteq H$. In the first case $H=W_i$ for some $i \in \lcro 0,q\rcro$,
and in the second one $H^{\bot_s} \subseteq W$ and hence $H^{\bot_s}=W_i$ for some $i \in \lcro 0,q\rcro$.
Hence the $S$-invariant subspaces are the elements of the non-decreasing sequence
$(W_0,W_1,\dots,W_q,W_q^{\bot_s},W_{q-1}^{\bot_s},\dots,W_1^{\bot_s},W_0^{\bot_s})$ (note that it is increasing if and only if $W$ is not a Lagrangian).

Hence the first conclusion in the theorem is proved. The second one is an obvious consequence of it by taking orthogonal subspaces.
\end{proof}

\begin{theo}\label{theo:invariantsubspacesbis}
Let $S$ be an optimal trivial spectrum subspace of $\calA_s$, where $(V,s)$ is a symplectic space of dimension $2n>0$ with $|\F|>2n-2$.
Assume that there is a totally $s$-singular $S$-invariant subspace $W$ of dimension $n-1$.
Then there is a unique such space, and the $S$-invariant subspaces fall into two categories:
\begin{enumerate}[(i)]
\item The $S$-invariant subspaces that are Lagrangians are exactly the linear subspaces $H$ such that
$W \subsetneq H \subsetneq W^{\bot_s}$;
\item The $S$-invariant subspaces that are not Lagrangians are totally ordered, and hence they form a flag
$(V_0,\dots,V_{2p-1})$ in which $V_k^{\bot_s}=V_{2p-k}$ for all $k \in \lcro 0,2p-1\rcro$
(there is an even number of such spaces).
\end{enumerate}
\end{theo}

\begin{proof}
First of all, like in the proof of Theorem \ref{theo:invariantsubspaces1} we can take the flag $(W_0,\dots,W_q)$ of the $S_W$-invariant subspaces,
and note that they are the $S$-invariant subspaces that are included in $W$, and by orthogonality
$W_q^{\bot_s},\dots,W_0^{\bot_s}$ are the $S$-invariant subspaces that include $W^{\bot_s}$.

Next, as seen earlier it is clear that every linear subspace $H$ of $V$ such that
$W \subsetneq H \subsetneq W^{\bot_s}$ is $S$-invariant, and all such spaces are Lagrangians.

We will prove that there are no other $S$-invariant subspace than those we have exhibited, which will conclude the proof.

So, let us choose two distinct Lagrangians $H_1,H_2$ that lay strictly between $W$ and $W^{\bot_s}$. Note that $H_1 \cap H_2=W$ and
$H_1+H_2=W^{\bot_s}$.
Let $G$ be an $S$-invariant subspace that is not a Lagrangian.
Let us apply Lemma \ref{lemma:keylemmairr} to $G$ and $H_1$. Here, outcome (iii) is ruled out. Assume that $G \subseteq H_1$, so that $G \subsetneq H_1$.
Then we cannot have $H_2 \subseteq G$ for obvious dimension reasons, and hence Lemma \ref{lemma:keylemmairr} applied to $G$ and $H_2$ yields $G \subseteq H_2$.
Hence $G \subseteq H_1 \cap H_2=W$. By orthogonality, if $H_1 \subseteq G$ we apply this to $G^{\bot_s}$ and obtain
$W^{\bot_s} \subseteq G$. Hence, $G$ must be one of $W_0,\dots,W_q,W_q^{\bot_s},\dots,W_0^{\bot_s}$.
In particular, we have proved that $W$ is the sole $S$-invariant subspace of dimension $n-1$.

Let finally $G$ be an $S$-invariant subspace that is a Lagrangian. Assume that $G\neq H_1$.
Again, let us apply Lemma \ref{lemma:keylemmairr}: none of outcomes (i) and (ii) is possible because $G \neq H_1$.
Therefore $G \cap H_1$ is a linear hyperplane of $G$, and hence it has dimension $n-1$. As it is an $S$-invariant subspace
we deduce from the first part of the proof that it equals $W$. Hence $W \subseteq G$, and by orthogonality $G=G^{\bot_s} \subseteq W^{\bot_s}$.
This completes the proof.
\end{proof}

\subsection{The Core-Wing Decomposition Theorem for trivial spectrum subspaces}\label{section:decompositionTspectrum}

Everything is now in place to prove Theorem \ref{theo:separationtheoremTspectrum}, and
we will actually prove the geometric version of it.
So, let $S$ be an optimal trivial spectrum subspace of $\calA_s$, where $(V,s)$ is a symplectic space of dimension $2n$ with $|\F|>2n-2$.
Let us take a maximal totally $s$-singular subspace $W$ of $(V,s)$
(note that it may not be unique, see the second type of invariant subspaces in Theorem \ref{theo:invariantsubspacesbis}).
Denote by $r$ its dimension, and take an adapted mixed symplectic basis $(e_1,\dots,e_r,g_1,\dots,g_{2n-2r},f_1,\dots,f_r)=\bfB$.
Then $S=S_W \triangle \overline{S}^{W}$, and $S_W$ is an optimal trivial spectrum subspace of $\End(W)$ while
$\overline{S}^{W}$ is an irreducible optimal trivial spectrum subspace of $\calA_{\overline{s}}$.
By taking the space $\calM$ of matrices that represents
$S_W$ in $(e_1,\dots,e_r)$ and the space $\calA$ of matrices that represents $\overline{S}^W$ in the symplectic basis
$(\overline{g_1},\dots,\overline{g_{2n-2r}})$, it is then clear that
the matrix space that represents $S$ in $\bfB$ equals $\calM \triangle \calA$.
And because $W$ is maximal, $\calM$ is an optimal trivial spectrum subspace of $\Mat_r(\F)$, and $\calA$ is an optimal trivial spectrum subspace of
$\calA_{s_{2n-2r}}$. Hence, the existence statement is proved.

Now, we prove the uniqueness statement. So, say that we have
a mixed symplectic basis $(e'_1,\dots,e'_t,g'_1,\dots,g'_{2n-2t},f'_1,\dots,f'_t)=\bfB'$
of $(V,s)$, an optimal trivial spectrum subspace $\calM'$ of $\Mat_t(\F)$, and an irreducible optimal trivial spectrum subspace $\calA'$ of
$\calA_{s_{2n-2t}}$ such that the matrix space that represents $S$ in $\bfB'$ is $\calM' \triangle \calA'$.
Set $W':=\Vect(e'_1,\dots,e'_t)$ and note that it is $S$-invariant and totally $s$-singular.
Then the assumptions show that $S_{W'}$ and $\overline{S}^{W'}$ are represented respectively by
$\calM'$ and $\calA'$ in the bases $(e'_1,\dots,e'_t)$ and $(\overline{g'_1},\dots,\overline{g'_{2n-2t}})$.
In particular, the irreducibility of $\calA'$ yields that $\overline{S}^{W'}$ is irreducible, and hence
$W'$ is a maximal totally $s$-singular $S$-invariant subspace.

Yet, by Theorems \ref{theo:invariantsubspaces1} and \ref{theo:invariantsubspacesbis} all the maximal totally $s$-singular $S$-invariant subspaces have the same dimension.
It follows that $r=t$.

Now we must distinguish two cases, whether $S$ has an invariant totally $s$-singular subspace of dimension $n-1$ or not.

Let us first consider the easy case where $S$ does not have such an invariant subspace. Then Theorem \ref{theo:invariantsubspaces1} shows that it has a unique maximal
totally $s$-singular invariant subspace, which is $W$. Hence $W=W'$; it follows that $\calM$ and $\calM'$ represent the same set of endomorphisms of $W$
in a different choice of bases, and hence they are similar; likewise $\calA$ and $\calA'$ represent the same set of $\overline{s}$-alternating endomorphisms in a different choice of symplectic bases of $W^{\bot_s}/W$, hence they are conjugated through an element of the symplectic group $\Sp_{2n-2r}(\F)$.

Assume finally that $S$ has an invariant totally $s$-singular subspace $H$ of dimension $n-1$.
Hence by Theorem \ref{theo:invariantsubspacesbis} we must have $r=n$ (the maximal totally $s$-singular subspaces that are $S$-invariant are all Lagrangians)
and so $\calA=\calA'=\{0\}$. Therefore, all we need is to prove that $\calM \simeq \calM'$.
Now, $\calM$ represents the set of endomorphisms $S_W$, and $\calM'$ represents $S_{W'}$.
We will simply prove that the same matrix space can be used to represent $S_W$ and $S_{W'}$, which will yield that
$\calM$ and $\calM'$ are similar. If $W =W'$ this is trivial, so we assume $W \neq W'$.
Then, we recall from point (i) of Theorem \ref{theo:invariantsubspacesbis} that $H=W \cap W'$ (beware that the notation is different here, though).
Since $H$ is an invariant hyperplane for $S_W$, each operator $v \in S_W$
induces an endomorphism of the $1$-dimensional space $W/H$, and this endomorphism must be zero because $v$ has trivial spectrum.
Hence each operator in $S_W$ maps into $H$. Now, consider an arbitrary basis $(z_1,\dots,z_{n-1})$ of $H$, denote by $\calM_0$
the matrix space that represents $S_H$ in that basis, and extend arbitrarily $(z_1,\dots,z_{n-1})$ into a basis $(z_1,\dots,z_n)$ of $W$.
By what precedes, the matrix space $\calM_1$ that represents $S_W$ in this extended basis satisfies
$$\calM_1 \subseteq \widetilde{\calM_0}:=\left\{\begin{bmatrix}
A & C \\
[0]_{1 \times (n-1)} & 0
\end{bmatrix} \mid A \in \calM_0 \; \text{and} \; C \in \F^{n-1}\right\}.$$
Yet $\widetilde{\calM_0}$ has dimension $\dim \calM_0+n-1=\dbinom{n}{2}=\dim \calM_1$, to the effect that
$\calM_1=\widetilde{\calM_0}$. Hence $\widetilde{\calM_0}$ represents $S_W$ in some basis of $W$. Likewise, we can extend
$(z_1,\dots,z_{n-1})$ into a basis of $W'$ and find that $\widetilde{\calM_0}$ also represents $S_{W'}$ in that basis.
Hence $\calM \simeq \widetilde{\calM_0} \simeq \calM'$, which completes the proof of Theorem \ref{theo:separationtheoremTspectrum}.

\subsection{The Core-Wing Decomposition Theorem for symplectic affine spaces}\label{section:canonicalForforms}

We now turn to the proof of Theorem \ref{theo:separationtheorem}. It will mostly be obtained thanks to Theorem
\ref{theo:separationtheoremTspectrum}, but the uniqueness statements need to be handled carefully.

So, let $\calS$ be an optimal symplectic affine subspace of $\calA^2(V)$. A linear subspace $W$ of $V$ will be called \textbf{$\calS$-adapted} when the following two properties are satisfied:
\begin{enumerate}[(i)]
\item $W$ is totally $s$-singular for all $s \in \calS$;
\item $W^{\bot_s}$ does not depend on the choice of $s \in \calS$ (in which case we simply denote this space by $W^{\bot_\calS}$).
\end{enumerate}

Now, denote by $S$ the translation vector space of $\calS$, and choose $s_0 \in \calS$. Consider the trivial spectrum \emph{linear} subspace
$$\calT:=\{R_{s_0}^{-1} R_s \mid s \in S\} \subseteq \calA_{s_0},$$
where we recall that $R_s : V \rightarrow V^\star$ has been defined at the start of Section \ref{section:bilinbasics}.

\begin{lemma}\label{lemma:Sadaptedsubspaces}
With the above notation, the $\calS$-adapted subspaces are the $\calT$-invariant subspaces that are totally $s_0$-singular.
\end{lemma}

\begin{proof}
Let $W$ be an $\calS$-adapted subspace of $V$. First of all $W$ is totally $s_0$-singular. Let $s \in \calS$.
Note that $R_s$ maps $W$ into the dual orthogonal of $W^{\bot_s}$, and hence onto it because $R_s$ is an isomorphism;
because of condition (ii) we deduce that $R_{s_0}^{-1} R_s$ leaves $W$ invariant. By linearity it follows that
$W$ is $\calT$-invariant.

Conversely, let $W$ be a $\calT$-invariant subspace that is totally $s_0$-singular.
For all $s \in S$, we deduce that $R_s=R_{s_0}(R_{s_0}^{-1} R_s)$ maps $W$ into the dual orthogonal of $W^{\bot_{s_0}}$.
Since this also holds for $s_0$ we deduce that every $R_s$, with $s \in \calS$, maps $W$ into the dual orthogonal of $W^{\bot_{s_0}}$, and hence onto it
because it is an isomorphism. Hence $W^{\bot_s}=W^{\bot_{s_0}}\supseteq W$ for all $s \in \calS$, and we conclude that $W$ is $\calS$-adapted.
\end{proof}

As a consequence, we have a structure theorem for the $\calS$-adapted subspaces. It is immediately obtained by combining Lemma \ref{lemma:Sadaptedsubspaces} and
Theorems \ref{theo:invariantsubspaces1} and \ref{theo:invariantsubspacesbis}:

\begin{theo}
Let $\calS$ be an optimal symplectic affine subspace of $\calA^2(V)$, where $\dim V=2n>0$ and $|\F|>2n-2$.
\begin{enumerate}[(i)]
\item If $\calS$ has no adapted subspace of dimension $n-1$, then its adapted subspaces are totally ordered.
\item If $\calS$ has an adapted subspace $W$ of dimension $n-1$, then it is the sole such subspace, and the $\calS$-adapted
subspaces fall into two categories: the ones that are not $n$-dimensional are all included in $W$, and they are totally ordered;
the remaining ones are the $n$-dimensional spaces that lay between $W$ and $W^{\bot_{\calS}}$.
\end{enumerate}
\end{theo}

Now we can finally prove Theorem \ref{theo:separationtheorem}. Let $\calS$ be an optimal symplectic affine subspace of $\calA^2(V)$,
where $V$ is a vector space with dimension $2n \geq 2$ and $|\F|>2n-2$.

Choose $s_0 \in \calS$, denote by $S$ the translation vector space of $\calS$ and set
$\calT:=\{R_{s_0}^{-1} R_s \mid s \in S\}$ like earlier.
Choose a maximal $\calS$-adapted subspace $W$. Then it is a maximal $\calT$-invariant subspace that is totally $s_0$-singular.
By the study from Section \ref{section:decompositionTspectrum}, we have $\calT=\calT_W \triangle \overline{\calT}^W$, $\calT_W$ is an optimal trivial spectrum subspace of $\End(W)$, and
$\overline{\calT}^W$ is an irreducible optimal trivial spectrum subspace of $\calA_{\overline{s_0}}$.
Let us take a mixed $s_0$-symplectic basis $\bfB=(e_1,\dots,e_r,g_1,\dots,g_{2n-2r},f_1,\dots,f_r)$ of $V$ that is adapted to $W$.
In that basis $W^{\bot_\calS}=W^{\bot_{s_0}}=\Vect(e_1,\dots,e_r,g_1,\dots,g_{2n-2r})$.
Hence for every $s \in \calS$, its Gram matrix $G(s)$ in that basis takes the simplified form
$$\begin{bmatrix}
[0]_{r \times r} & [0]_{r \times (2n-2r)} & M(s) \\
[0]_{(2n-2r) \times r} & A(s) & [?]_{(2n-2r) \times r} \\
-M(s)^T & [?]_{r \times (2n-2r)} & [?]_{r \times r}
\end{bmatrix}.$$
Hence $G(\calS) \subseteq M(\calS) \bullet A(\calS)$, and $M(\calS)$ is an affine subspace of invertible matrices of $\Mat_r(\F)$,
whereas $A(\calS)$ is an affine subspace of invertible matrices of $\Mata_{2n-2r}(\F)$.
Note that every element $s$ of $S$ induces an alternating form $\overline{s}$ on $W^{\bot_\calS}/W$, and that
$A(\calS)$ represents the elements of $\overline{s_0}+\{\overline{s} \mid s \in S\}$ in the basis $(g_1,\dots,g_{2n-2r})$.
Yet the latter has no nonzero adapted subspace because $\overline{\calT}^W$ is irreducible.
Hence $A(\calS)$ is irreducible (in the meaning of Section \ref{section:intromatrix}).

Now, with the same line of reasoning as in Proposition \ref{prop:decompbase}, we deduce that
$M(\calS)$ and $A(\calS)$ are optimal and that $G(\calS)=M(\calS) \bullet A(\calS)$.
This proves the existence statement.

Now on to the uniqueness statement. So, assume that there exists a basis~$(e'_1,\dots,e'_t,g'_1,\dots,g'_{2n-2t},f'_1,\dots,f'_t)$ of $V$ in
which the space $\calS$ is represented by $\calM' \bullet \calA'$ for some optimal affine subspace $\calM'$ of invertible elements  of $\Mat_t(\F)$ and
some irreducible optimal affine subspace $\calA'$ of invertible elements of $\Mata_{2n-2t}(\F)$.
We set $W':=\Vect(e'_1,\dots,e'_t)$ and note that it is $\calS$-adapted, as obviously the subspace $\Vect(e'_1,\dots,e'_t,g'_1,\dots,g'_{2n-2t})$ is included in $(W')^{\bot_s}$ for all $s \in \calS$, and the equality comes from the nondegeneracy.
Having $(W')^{\bot_\calS}=\Vect(e'_1,\dots,e'_t,g'_1,\dots,g'_{2n-2t})$, we see that
$\calA'$ represents the affine space $\calS'$ of all symplectic forms on $(W')^{\bot_\calS}/W'$ induced by the elements of $\calS$.
Since $\calA'$ is irreducible, $\calS'$ has no nonzero adapted subspace.
Assume now that $W' \subsetneq W''$ for some $\calS$-adapted subspace $W''$.
Then $W' \subseteq W'' \subseteq (W'')^{\bot_\calS} \subseteq (W')^{\bot_\calS}$ and
it is then easily checked that $W''/W'$ is $\calS'$-adapted, which contradicts the irreducibility of $\calS'$.
Hence $W'$ is maximal among the $\calS$-adapted subspaces.

Now, we can conclude, by distinguishing once more between two cases.

Assume first that $\calS$ has no adapted subspace of dimension $n-1$. Then there is a unique maximal totally $s$-singular $\calT$-invariant subspace, and hence
$W=W'$. It follows that:
\begin{enumerate}[(i)]
\item $\calM$ and $\calM'$ represent (in the meaning of Gram matrices)
the set of all bilinear forms $(x,\overline{y}) \in W \times (V/W^{\bot_\calS}) \mapsto s(x,y)$, where $s$
ranges over $\calS$, in a different choice of bases of $W$ and $V/W^{\bot_\calS}$; hence these spaces are equivalent;
\item $\calA$ and $\calA'$ represent the set of all alternating forms $(\overline{x},\overline{y}) \in  (W^{\bot_\calS}/W)^2 \mapsto s(x,y)$, where
$s$ ranges over $\calS$, in a different choice of basis of $W$, and hence they are congruent.
\end{enumerate}
Assume finally that $\calS$ has an adapted subspace $H$ of dimension $n-1$.
Then $W$ and $W'$ have dimension $n$, and hence $r=t=n$ and $\calA=\calA'=\{0\}$ (the void matrix).
Moreover $H \subseteq W$ and $H \subseteq W'$.
It remains to prove that the matrix spaces $\calM$ and $\calM'$ are equivalent.
To see this, take a mixed $s_0$-symplectic basis $(e''_1,\dots,e''_{n-1},g''_1,g''_2,f''_1,\dots,f''_{n-1})$ that is adapted to $W$
and such that $\Vect(e''_1,\dots,e''_{n-1})=H$.
Now, the set of matrices $\calN$ that represents $\calS$ in that basis equals
$\calM_0 \bullet \calA_0$ for an optimal affine subspace $\calM_0$ of invertible elements of $\Mat_{n-1}(\F)$ and an optimal
affine subspace $\calA_0$ of invertible elements of $\Mata_2(\F)$. Obviously $\calA_0=\{\lambda K_2\}$ for a fixed scalar $\lambda \in \F^\times$.
Hence we see that the space of all bilinear forms $(x,\overline{y}) \in W \times (V/W) \mapsto s(x,y)$, where $s$ ranges over
$\calS$, is represented by the set $\widetilde{\calM_0}^{\lambda}$ of all matrices
$\begin{bmatrix}
[0]_{(n-1) \times 1} & M_0 \\
\lambda & R
\end{bmatrix}$ with $M_0 \in \calM_0$ and $R \in \Mat_{1,n-1}(\F)$.
Hence $\calM \sim \widetilde{\calM_0}^{\lambda}$. Likewise $\calM' \sim \widetilde{\calM_0}^{\mu}$ for some $\mu \in \F^\times$.
Finally, multiplying the first column with a nonzero constant yields $\widetilde{\calM_0}^{\mu} \sim \widetilde{\calM_0}^{\lambda}$,
and hence $\calM \sim \calM'$.

This completes the proof of Theorem \ref{theo:separationtheorem}.

\subsection{The Invariant Subspace Lemma}

This whole section revolved around invariant subspaces, and we will seize the opportunity to prove a lemma that is related to them
and will be very useful in the proof of Theorem \ref{theo:maintheo}.

\begin{lemma}[Invariant Subspace Lemma]\label{lemma:invariantsubspacelemma}
Let $(V,s)$ be a symplectic space, and $V_0$ be a totally $s$-singular subspace of $V$.
Let $S$ be a trivial spectrum linear subspace of $\calA_s$. Assume that $S$ includes
$\calN_{V_0}$, i.e., that it contains every operator $u \in \calA_s$ that vanishes on $V_0$ and maps $V_0^{\bot_s}$ into $V_0$.
Then $V_0$ is $S$-invariant.
\end{lemma}

\begin{proof}
Assume on the contrary that there exist $u \in S$ and $x \in V_0$ such that $u(x)\not\in V_0$.
Note that $u(x)-x \not\in V_0$.

We start by constructing a splitting $V=V_0 \oplus V_1 \oplus V_2$ in which $V_0 \oplus V_1=(V_0)^{\bot_s}$,
$V_2 \bot_s V_1$, $V_2$ is totally $s$-singular and $u(x)-x \in V_1 \oplus V_2$:
\begin{itemize}
\item If $u(x)-x \in (V_0)^{\bot_s}$, we simply take an arbitrary direct summand $V_1$ of $V_0$ in $(V_0)^{\bot_s}$ that contains
$u(x)-x$, and then we take a direct summand $V_2$ of $V_0$ in $(V_1)^{\bot_s}$ that is totally $s$-singular
(i.e., we take a Lagrangian in $(V_1)^{\bot_s}$, transverse to $V_0$).
\item If $u(x)-x \notin (V_0)^{\bot_s}$, we can take a totally $s$-singular direct summand $V_2$ of $(V_0)^{\bot_s}$ in $V$ that contains
$u(x)-x$, and then we take $V_1:=(V_0+V_2)^{\bot_s}$.
\end{itemize}
Now, we take a mixed symplectic basis $(e_1,\dots,e_r,f_1,\dots,f_{2m},g_1,\dots,g_r)$ of $V$ that is adapted to the decomposition
$V=V_0 \oplus V_1 \oplus V_2$ and in which $(e_1,\dots,e_r,g_1,\dots,g_r)$ is a symplectic basis of $V_0 \oplus V_2$.
Hence the Gram matrix of $s$ in that basis equals
$$\begin{bmatrix}
[0]_{r \times r} & [0]_{r \times 2m} & I_r \\
[0]_{2m \times r} & K_{2m} & [0]_{2m \times r} \\
-I_r & [0]_{r \times 2m} & [0]_{r \times r}
\end{bmatrix}.$$
The $s$-alternating endomorphisms are the matrices with matrix in the previous basis of the form
$$\begin{bmatrix}
A & (K_{2m} C)^T & G \\
B & (K_{2m})^{-1}E & C \\
H & -(K_{2m}B)^T & A^T
\end{bmatrix}$$
where $A,B,C$ are arbitrary, and $G,H,E$ are alternating.
Now, our assumptions yield that the matrix $M_0$ of $u$ in that basis has its upper-left block $A(M_0)$ with first column $\begin{bmatrix}
1 \\
[0]_{r-1}
\end{bmatrix}$. Moreover, the assumptions show that the matrix space $\calM$ that represents $S$ in the said basis
contains all the matrices of the form
$$\begin{bmatrix}
[0]_{r \times r} & (K_{2m} C)^T & G \\
[0]_{2m \times r} & [0]_{2m \times 2m} & C \\
[0]_{r \times r} & [0]_{r \times 2m} & [0]_{r \times r}
\end{bmatrix}$$
with $C \in \Mat_{2m,r}(\F)$ and $G \in \Mata_r(\F)$. Adding one of those matrices to $M_0$ yields that $\calM$ contains a matrix of the form
$$\begin{bmatrix}
A(M_0) & [0]_{r \times 2m} & [0]_{r \times r} \\
[?]_{2m \times r} & [?]_{2m \times 2m} & [0]_{2m \times r} \\
[?]_{r \times r} & [?]_{r \times 2m} & [?]_{r \times r}
\end{bmatrix}$$
and it follows that $A(M_0)$ must have trivial spectrum. Yet it is obvious from the first column that $1$ is an eigenvalue of $A(M_0)$.
This contradiction completes the proof.
\end{proof}

\section{Analysis of the $W_\L$ spaces}\label{section:WL}

This section is devoted to the analysis of the $W_\L$ spaces introduced in Section \ref{section:irroptimalintroduction}.
Note here that we will not make any cardinality assumption on $\F$, but if $\F$ is finite there will be no interesting irreducible symplectic space
of forms that arises from this construction, as was already pointed out in Section \ref{section:irroptimalintroduction}.

We have to prove the results that were announced earlier. So, we have to compute the dimension of $W_\L$,
characterize the affine subspaces with translation vector space $W_\L$ that consist of symplectic forms only, and
then classify them up to congruence. And we must also deal with the trivial spectrum spaces counterpart of this construction, so as to obtain Theorem \ref{theo:classoptimalirrTspectrum}, but this will be easy when the analysis is completed in $\calA^2(V)$.

The key, and we have avoided it in the introduction, is to give a more structural view on $W_\L$, by considering the
$\L$-vector space $V^\L$. This new viewpoint is essential, and we tackle it right away.
Throughout this section, we fix a vector space $V$ of dimension $2n \geq 4$.

\subsection{A new viewpoint on the $W_\L$ spaces}

Let $\L=\F[a]$ be a quadratic extension of $\F$ inside $\End_\F(V)$.
Thus $V$ is naturally seen as an $\L$-vector space, and we will use the notation $V^\L$ when we see $V$ with this added structure
and not simply as an $\F$-vector space.

We take an arbitrary nonzero $\F$-linear form $e$ on $\L$. The choice $e=\tr_{\L/\F}$ is not possible for an inseparable extension, and as we want to keep the treatment as general as possible we will avoid specifying $e$ whenever possible.

For every bilinear form $b : V^2 \rightarrow \F$, there is a unique left-$\F$-linear and right-$\L$-linear form
$b^{\L} : (x,y)\in (V^\L)^2 \rightarrow \L$ such that
\begin{equation}\label{extension}
\forall (x,y)\in \bigl(V^\L\bigr)^2, \; b(x,y)=e\bigl(b^\L(x,y)\bigr).
\end{equation}
This comes from the following representation lemma: the mapping $\lambda \in \L \mapsto [a \mapsto e(\lambda a)] \in \Hom_\F(\L,\F)$
is a vector space isomorphism. Let $x$ and $y$ in $V$. Then $\lambda \in \L \mapsto b(x,\lambda y)$ is an $\F$-linear form, which then equals
$\lambda \mapsto e(b^\L(x,y)\lambda)$ for a unique $b^\L(x,y) \in \L$. It is then easily checked that
$b^\L$ is left-$\F$-linear and right-$\L$-linear. And from the right-$\L$-linearity one shows that \eqref{extension} defines
$b^\L$ in a unique way.

Now, beware that in general $b^\L$ is not $\L$-bilinear: indeed this would require that $b(x,\lambda y)=b(\lambda x,y)$ for all $x,y$ in $V$ and $\lambda \in \L$,
i.e., that the elements of $\L$ are all $b$-selfadjoint.
But if $b$ belongs to $W_\L$ then $b^\L$ turns out to be alternating, as for
all $x \in V$ we have $\forall \lambda \in \L, \; e(b^\L(x,x)\lambda)=e(b^\L(x,\lambda x))=b(x,\lambda x)=0$;
Then $b^\L$ is skew-symmetric (remember that $b^\L$ is left-additive and right-additive), therefore it is left-$\L$-linear.
Hence for all $b \in W_\L$ the form $b^\L$ is an alternating bilinear form on the $\L$-vector space $V^\L$.
Conversely, for every alternating bilinear form $B \in \calA^2(V^\L)$ on the $\L$-vector space $V^\L$, the form $e \circ B$ belongs to $W_\L$. Let us conclude:

\begin{prop}\label{prop:isomcomposlinform}
Let $e$ be a nonzero $\F$-linear form on $\L$. Then
$$W_\L=\bigl\{e \circ B \mid B \in \calA^2(V^\L)\bigr\},$$
and more precisely $B \in \calA^2(V^\L) \mapsto e \circ B \in W_\L$ is an isomorphism of $\F$-vector spaces.
\end{prop}

\begin{cor}\label{cor:dimWL}
The space $W_\L$ has dimension $n(n-1)$.
\end{cor}

\begin{proof}
Indeed, $\calA^2(V^\L)$ has dimension $\frac{n(n-1)}{2}$ over $\L$, and hence dimension $n(n-1)$ over $\F$.
\end{proof}

As an aside, but it will be used at times, we have the following result:

\begin{prop}\label{prop:radicaleB}
Let $B \in \calA^2(V^\L)$. Then $B$ has the same radical as $e \circ B$.
\end{prop}

\begin{proof}
Let indeed $x \in V$. Obviously if $x$ is in the radical of $B$ then it is in the one of $e \circ B$.
Conversely, assume that $x$ is in the radical of $e \circ B$. Let $y \in V$. Then
$\forall \lambda \in \L, \; e(B(x,y)\lambda)=e(B(x,\lambda y))=b(x,\lambda y)=0$, which leads to $B(x,y)=0$.
Hence $x$ is in the radical of $B$.
\end{proof}

\subsection{Main structural results on the $W_\L$ spaces}

Next, we turn to the characterisation of the symplectic affine spaces with translation vector space $W_\L$:

\begin{prop}\label{prop:dirWL}
Let $s_0 \in \calA^2(V)$. The following conditions are equivalent:
\begin{enumerate}[(i)]
\item All the forms in $s_0+W_\L$ are symplectic;
\item For all $\lambda \in \L \setminus \F$ and all $x \in V \setminus \{0\}$, one has $s_0(x,\lambda x) \neq 0$.
\item There exists $\lambda \in \L \setminus \F$ such that $x \mapsto s_0(x,\lambda x)$ is nonisotropic.
\end{enumerate}
\end{prop}

\begin{proof}
It is clear that condition (iii) and (ii) are equivalent because $s_0$ is alternating and $\L=\Vect_\F(1,\lambda)$ whatever the choice of $\lambda \in \L \setminus \F$.
Note that for all $s \in W_\L$ we have
$$\forall \lambda \in \L, \; \forall x \in V, \; (s_0+s)(x,\lambda x)=s_0(x,\lambda x).$$
Hence if (iii) holds it is clear that all the forms in $s_0+W_\L$ are nondegenerate, i.e., they are symplectic.

Assume finally that (ii) fails. Then there exists $x \in V \setminus \{0\}$ such that $\forall \lambda \in \F, \; s_0(x,\lambda x)=0$.
Then $s_0^\L(x,x)=0$. Hence $s_0^\L(x,-)$ is an $\L$-linear form on $V^\L$ that vanishes at $x$.
Classically, because $\dim_\L V \geq 2$ there exists an alternating bilinear form $B$ on $V^\L$ such that $B(x,-)=s_0^\L(x,-)$.
Then $e \circ B$ belongs to $W_\L$ and $x$ is in the left radical of $s_0-e \circ B$, which belongs to $s_0+W_\L$.
Therefore (i) does not hold.
\end{proof}

\subsection{Transitivity, and irreducibility}

Our aim here is to prove the irreducibility of the symplectic affine subspaces of $\calA^2(V)$ with translation vector space $W_\L$,
and to this end we start with a lemma that is interesting in itself.

\begin{lemma}\label{lemma:transitivity}
For all $x \in V \setminus \{0\}$, one has
$$\dim \{s(x,-) \mid s \in W_\L\}=2n-2 \quad \text{and} \quad \L x=\underset{s \in W_\L}{\bigcap} \Ker s(x,-).$$
\end{lemma}

\begin{proof}
Let $x \in V \setminus \{0\}$.
Consider the linear mapping
$$\Phi : s \in W_\L \mapsto s(x,-)\in V^\star.$$
Its kernel is the space of all $s \in W_\L$ that have $x$ in their radical, and hence by Propositions \ref{prop:radicaleB} and \ref{prop:isomcomposlinform}
it is isomorphic to the $\F$-vector space $\calT$ of all $B \in \calA^2(V_\L)$ that have $x$ in their radical.
The latter is naturally isomorphic to $\calA^2(V_\L/\L x)$, which has dimension $\dbinom{n-1}{2}$ over $\L$, and hence dimension $(n-1)(n-2)$ over $\F$.
Therefore $\Ker \Phi$ has dimension $(n-1)(n-2)$ and we deduce from the rank theorem that
$$\dim \{s(x,-) \mid s \in W_\L\}=n(n-1)-(n-1)(n-2)=2(n-1).$$
Next, by the definition of $W_\L$ we have $\L x \subseteq \underset{s \in W_\L}{\bigcap} \Ker s(x,-)$.
By duality, the latter has dimension $2n-\dim \{s(x,-) \mid s \in W_\L\}=2$, whereas $\dim_\F \L x=2$.
This yields the second claimed equality.
\end{proof}

Here is a first consequence:

\begin{prop}\label{prop:WLirreducible}
Let $\calS$ be a symplectic affine subspace of $\calA^2(V)$ with translation vector space $W_\L$. Assume that $\dim V \geq 4$.
Then $\calS$ is irreducible.
\end{prop}

\begin{proof}
Denote by $S$ the translation vector space of $\calS$.
Let us consider an $\calS$-adapted subspace $W$ (as defined at the start of Section \ref{section:canonicalForforms}). Note that $\dim W \leq n$. Assume that $W \neq \{0\}$ and choose $x \in W \setminus \{0\}$. Choose $s_0 \in \calS$. The form $s_0+s$ is nondegenerate for all $s \in S$, and hence $s_0(x,-)\not\in \{s(x,-) \mid s \in S\}$.
Hence $\dim \{s(x,-) \mid s \in \F s_0+S\}=2n-1$ by Lemma \ref{lemma:transitivity}.

Yet for all $s \in \calS$ the linear form $s(x,-)$ belongs to the dual orthogonal of $W^{\bot_\calS}$, hence by linearity
this is also true for all $s$ in $\F s_0+S$. This yields $2n-1 \leq  \dim W \leq n$, which contradicts our assumption that $n \geq 2$.
Hence $\calS$ has no nonzero $\calS$-adapted subspace, i.e., it is irreducible.
\end{proof}

Let us conclude on the duality between the alternator space and the translation vector space.

\begin{prop}\label{prop:recoverL}
Assume that $\dim V \geq 4$. Then $\L=\calA_{W_\L}$.
\end{prop}

\begin{proof}
The inclusion $\L \subseteq \calA_{W_\L}$ is obvious.

Conversely, let $u \in \calA_{W_\L}$.
Let $x \in V \setminus \{0\}$. By Lemma \ref{lemma:transitivity} we know that $\L x =P:=\underset{s \in W_\L}{\bigcap} \Ker s(x,-)$.
Since  $u \in \calA_{W_\L}$ we find $u(x) \in P$, and hence $u(x) \in \L x$.

Hence $u$ is an endomorphism of the group $(V,+)$ that satisfies $u(x)\in \Vect_\L(x)$ for all $x \in V$.
Classically (see e.g. lemma 2.1 of \cite{dSPRC}),
because $\dim_\L V=n \geq 2$ this yields an element $\alpha \in \L$ such that $u(x)=\alpha x$ for all $x \in V$.
Hence $u \in \L$. We conclude that $\L=\calA_{W_\L}$.
\end{proof}

This yields Proposition \ref{prop:uniquenessL}.

\subsection{The fundamental form (general case)}

Now, we define a fundamental invariant of an affine subspace of $\calA^2(V)$ with translation vector space $W_\L$.
Remember that we have fixed a nonzero linear form $e$ on the $\F$-vector space $\L$, which we write $e^\L$ if we need to stress the dependency
with respect to $\L$. Now, we will need to make this choice \textbf{coherent}, meaning that whenever we have an isomorphic $\F$-algebra
$\L'$ and an isomorphism $\sigma : \L \overset{\simeq}{\longrightarrow} \L'$ of $\F$-algebras, we have $\forall x \in \L, \; e^{\L}(x)=e^{\L'}(\sigma(x))$.
When $\L$ is separable over $\F$ we systematically take $e^{\L}=\tr_{\L/\F}$.
Otherwise we can fix a nonzero linear form $e^{\L}$, and then force $e^{\L'}$ to be $x \mapsto e^{\L}(\sigma^{-1}(x))$
for every isomorphic $\F$-algebra $\L'$ with an isomorphism $\sigma : \L \overset{\simeq}{\longrightarrow} \L'$ of $\F$-algebras
(which is then unique, because $\Gal(\L/\F)=\{\id\}$). So, even in the inseparable case we have a coherent choice, although not a truly canonical one.

\begin{Def}
Let $\calS$ be an affine subspace of $\calA^2(V)$ with translation vector space $W_\L$.
Then, for each $x \in V$, the scalar $s^{\L}(x,x) \in \L$ does not depend on the choice of $s \in \calS$, and we denote it by $Q_\calS(x)$. The mapping
$$x \in V \mapsto Q_\calS(x) \in \L$$
is the \textbf{fundamental form} of $\calS$ (with respect to $e$).
\end{Def}

In this definition, the uniqueness statement is obvious because $b^\L(x,x)=0$ for all $b \in W_\L$.

Note that the definition of the fundamental form is modified if we change the linear form $e$ for another one $e'$,
but the new fundamental form will then simply be a scalar multiple (over $\L$) of the former one.
In the separable case, we always choose $e=\tr_{\L/\F}$ when defining fundamental forms.

Now, the reader must beware that the form $s^{\L}$ itself does depend on the choice of $s \in \calS$ (worse still, it is an injective function of $s$!).
Only the ``diagonal" form $x \in V \mapsto s^{\L}(x,x)$ is independent of this choice.
If we fix $s \in \calS$ then $s^\L$ is right-$\L$-linear and left-$\F$-linear, so $Q_\calS$ appears as some sort of mysterious ``semi-quadratic" form.
In the separable case, we will considerably clear up that mystery in Section \ref{section:fundamentalformseparable}.

Before we focus on the separable case, we can restate Proposition \ref{prop:dirWL} differently:

\begin{prop}
Let $\calS$ be an affine subspace of $\calA^2(V)$ with translation vector space $W_\L$.
Then $\calS$ is symplectic if and only if its fundamental form is nonisotropic.
\end{prop}

\begin{proof}
Write $\L=\F[a]$, and let $s_0 \in \calS$.
Let $x \in V \setminus \{0\}$. Then $s_0(x,ax)=e(Q_\calS(x)a)$ and $s_0(x,x)=0$. So as to have $s_0(x,\lambda x)=0$ for all $\lambda \in \L$,
it is necessary and sufficient that $Q_\calS(x)=0$, which amounts to $s_0(x,ax)= 0$.
Hence by Proposition \ref{prop:dirWL} we see that $\calS$ is symplectic if and only if $\forall x \in V \setminus \{0\}, \; Q_\calS(x) \neq 0$.
\end{proof}

Next, we show that the fundamental form suffices to recover the affine subspace it is attached to.

\begin{prop}\label{prop:recoveraffinespacefromFform}
Let $s_0$ and $s_1$ belong to $\calA^2(V)$. Assume that $s_0+W_\L$ and $s_1+W_\L$ have the same fundamental form.
Then they are equal.
\end{prop}

\begin{proof}
Denote by $Q$ the common fundamental form of $s_0+W_\L$ and $s_1+W_\L$.
Let $x \in V$. Then for all $\lambda \in \L$ we have $s_0(x,\lambda x)=e(Q(x,x)\lambda)=s_1(x,\lambda x)$.
Hence $s_0-s_1 \in W_\L$, and the result ensues.
\end{proof}

We now turn to the study of the orbits, under the action of $\GL(V)$ by congruence, of the affine subspaces with translation vector space of the form $W_\L$.

Say that we have two vector spaces $U$ and $U'$ over potentially different fields $\K$ and $\K'$ that are extensions of $\F$, and that we have forms
$h : U \rightarrow \K$ and $h' : U' \rightarrow \K'$. We shall say that $h$ and $h'$ are \textbf{quasi-equivalent} when there is
an isomorphism $\sigma : \K \overset{\simeq}{\longrightarrow} \K'$ of $\F$-algebras and a \textbf{$\sigma$-quasilinear isomorphism}
$\varphi : U \overset{\simeq}{\longrightarrow} U'$ such that $h'(\varphi(x))=\sigma(h(x))$ for all $x \in U$.
We recall that a $\sigma$-quasilinear mapping from $U$ to $U'$ is a mapping $f$ that satisfies
$$\forall (\lambda,x,y)\in \K \times U^2, \; f(\lambda x+y)=\sigma(\lambda) f(x)+f(y).$$

\begin{theo}\label{theo:classifycongruencebyfundamental}
Let $V$ be a vector space of dimension $2n \geq 4$.
Let $\L$ be a separable quadratic extension of $\F$ in $\End(V)$, and $\L'$ be a quadratic extension of $\F$ in $\End(V)$.
Let $\calS$ and $\calS'$ be affine subspaces of $\calA^2(V)$ with respective translation vector spaces $W_\L$ and $W_{\L'}$.
\begin{enumerate}[(a)]
\item If $\calS$ and $\calS'$ are congruent, then the $\F$-algebras $\L$ and $\L'$ are isomorphic.
\item Assume that $\L$ and $\L'$ are isomorphic $\F$-algebras. Then
$\calS$ and $\calS'$ are congruent if and only the fundamental forms $Q_\calS$ and $Q_{\calS'}$ are quasi-equivalent.
\end{enumerate}
\end{theo}

\begin{proof}
Assume that there exists $\varphi \in \GL(V)$ such that
$$\calS'=\{(x,y) \mapsto s(\varphi^{-1}(x),\varphi^{-1}(y)) \mid s \in \calS\}.$$
Hence by linearity $W_{\L'}=\{(x,y) \mapsto s(\varphi^{-1}(x),\varphi^{-1}(y)) \mid s \in W_\L\}$. \\
Let $u \in \L$. Let $s' \in W_{\L'}$. Hence $s' : (x,y) \mapsto s(\varphi^{-1}(x),\varphi^{-1}(y))$
for some $s \in W_\L$.
Then for all $x \in V$ we have
$$s'(x,(\varphi u \varphi^{-1})(x))=s(\varphi^{-1}(x), u(\varphi^{-1}(x)))=0.$$
Therefore $\varphi u \varphi^{-1} \in \L'$ (by Proposition \ref{prop:recoverL}).
Hence $\theta : u \mapsto \varphi u \varphi^{-1}$ induces an isomorphism of $\F$-algebras from $\L$ to $\L'$.
Moreover $\varphi$ is obviously $\theta$-quasilinear from $V^{\L}$ to $V^{\L'}$.

Furthermore, let $s_0 \in \calS$ and set $s'_0 : (x,y) \mapsto s_0(\varphi^{-1}(x),\varphi^{-1}(y)) \in \calS'$. Let $x \in V$.
Then, for all $\lambda \in \L'$,
\begin{align*}
s'_0(x,\lambda x) & =s_0(\varphi^{-1}(x),\varphi^{-1}(\lambda x)) \\
& =s_0(\varphi^{-1}(x),\theta^{-1}(\lambda) \varphi^{-1}(x)) \\
& =e^{\L}\bigl(Q_\calS(\varphi^{-1}(x)) \,\theta^{-1}(\lambda)\bigr) \\
& =e^{\L'}\bigl(\theta\bigl(Q_\calS(\varphi^{-1}(x))\, \theta^{-1}(\lambda)\bigr)\bigr) \\
& =e^{\L'}\bigl(\theta\bigl(Q_\calS(\varphi^{-1}(x))\bigr)\, \lambda\bigr).
\end{align*}
Hence $\calQ_{\calS'}(x)=\theta\bigl(Q_\calS(\varphi^{-1}(x))\bigr)$.
It ensues that $\calQ_{\calS'}$ and $\calQ_{\calS}$ are quasi-equivalent, as claimed.

It remains to prove the converse statement. So, assume that $\L$ and $\L'$ are isomorphic and that
$\calQ_{\calS'}$ and $\calQ_{\calS}$ are quasi-equivalent. Choose corresponding isomorphisms $\theta : \L \overset{\simeq}{\rightarrow} \L'$
and $\varphi \in \GL(V)$.
Consider $\calS'':=\{(x,y) \mapsto s(\varphi^{-1}(x),\varphi^{-1}(y)) \mid s \in \calS\}$.
Then the same computation we have just performed shows that
$$\forall s'' \in \calS'', \; \forall x \in V, \; \forall \lambda \in \L', \; s''(x,\lambda x)=e^{\L'}\bigl(\theta(Q_\calS(\varphi^{-1}(x))) \lambda\bigr),$$
and since the right-hand side of the equality does not depend on the choice of $s''$, we find that the translation vector space of $\calS''$ is included in
$W_{\L'}$ and because the dimensions are equal we recover that it equals $W_{\L'}$. And then we find that the
fundamental form of $\calS''$ is
$x \mapsto \theta\bigl(Q_\calS(\varphi^{-1}(x))\bigr)$, which is no other than $Q_{\calS'}$.
Using Proposition \ref{prop:recoveraffinespacefromFform}, we deduce that $\calS''=\calS'$, and we conclude that $\calS' \cong \calS$.
\end{proof}

\subsection{The fundamental form (separable case)}\label{section:fundamentalformseparable}

Now, we consider only the case where $\L$ is a separable extension of $\F$ (but we might still have $\car(\F) =2$).
In that situation we take $e:=\tr_{\L/\F}$ in the definition of the fundamental form, and we denote by $\sigma$
the nonidentity automorphism of $\L$ over $\F$ (so that $e=\id_\L+\sigma$).

Our first result will dissipate the shroud of mystery that surrounds the fundamental form. Indeed, it appears as the ``diagonal" form
of a unique skew-Hermitian form on the $\L$-vector space $V^\L$.

\begin{prop}
Let $s$ be an alternating $\F$-bilinear form on $V$.
Then there is a unique skew-Hermitian form $h$ on $V^\L$ such that
$$\forall x \in V, \; \forall \lambda \in \L, \;  s(x,\lambda x)=\tr_{\L/\F} (h(x,x)\lambda).$$
\end{prop}

\begin{proof}
We will give two proofs: one by dimension argument, and also a more direct one.

Assign to every $s \in \calA^2(V)$ the function
$$K(s) : (x,\lambda)\in V\times \L \longmapsto s(x,\lambda x).$$
This defines an $\F$-linear mapping $K : s \mapsto K(s)$ with kernel $W_\L$, and hence with rank $\frac{1}{2}\,(2n)(2n-1)-\dim W_\L=n(2n-1)-n(n-1)=n^2$.
Assign to each skew-Hermitian form $h$ on $V^\L$ the alternating form
$$S(h) : (x,y) \mapsto  \tr_{\L/\F} (h(x,y)).$$
This defines an $\F$-linear mapping $S : h \mapsto S(h)$. Consider the composite mapping $K \circ S : \calSH(V^\L) \rightarrow \im K$,
where $\calSH(V^\L)$ denotes the space of all skew-Hermitian forms on $V^\L$.
It suffices to prove that $K \circ S$ is bijective. Yet both its source and target spaces have dimension $n^2$, so it suffices to prove that $K$ is injective,
which is easy: indeed if for a skew-Hermitian form $h$ on $V^\L$ we have $\forall (x,\lambda)\in V\times \L,\; \tr_{\L/\F} (h(x,\lambda x))=0$,
then with $x$ fixed, by varying $\lambda$ we find $h(x,x)=0$, and we conclude that $h=0$ because skew-Hermitian forms are controlled by their ``diagonal" part.

Now, we give a more constructive proof.
Fix $s \in \calA^2(V)$, and consider the form $s^\L : (V^\L)^2 \rightarrow \L$
such that $\forall (x,y)\in V^2, \; \forall \lambda \in \L, \; s(x,\lambda y)=\tr_{\L/\F}(s^\L(x,y)\lambda)$.
It is right-$\L$-linear but only left-$\F$-linear.
Let $y \in V$. Then the $\F$-linear form $s^\L(-,y)$ splits uniquely into $s_+(-,y)+s_-(-,y)$ where
$s_+$ is an $\L$-linear form and $s_-$ is an $\L$-antilinear form (so that $s_-(\mu x,y)=\sigma(\mu) s_-(x,y)$ for all $x \in V$ and $\mu \in \L$).
Clearly $(x,y) \mapsto s_+(x,y)$ is $\L$-bilinear and $h : (x,y) \mapsto s_-(x,y)$ is $\L$-sesquilinear.
Next, we prove that $Q : x \mapsto s_+(x,x)$ vanishes.
Let indeed $x \in V$. Then $\tr_{\L/\F}(Q(\lambda x)+h(\lambda x,\lambda x))=s(\lambda x,\lambda x)=0$ for all $\lambda \in \L$,
yielding $\tr_{\L/\F}(\lambda^2 Q(x))+\sigma(\lambda)\lambda \tr_{\L/\F}(h(x,x))=0$.
The case $\lambda=1$ yields $\tr_{\L/\F}(h(x,x))=-\tr_{\L/\F}(Q(x))$, and hence
$\forall \lambda \in \L, \; \tr_{\L/\F}\bigl((\lambda^2-\sigma(\lambda)\lambda) Q(x)\bigr)=0$.
Now, assuming that $Q(x) \neq 0$ this yields a $1$-dimensional subspace $D$ of the $\F$-vector space $\L$ that contains
$\lambda^2-\sigma(\lambda)\lambda$ for all $\lambda \in \L$. Yet
$\lambda^2-\sigma(\lambda)\lambda=\tr_{\L/\F}(\lambda)\lambda-2 N_{\L/\F}(\lambda)$ for all $\lambda \in \F$.
\begin{itemize}
\item If $\car(\F) \neq 2$ we choose $\lambda \in \L \setminus \{0\}$ with $\tr_{\L/\F}(\lambda)=0$ to obtain that $D=\F$,
and then we choose $\lambda \in \L \setminus \F$ with $\tr_{\L/\F}(\lambda) \neq 0$ to obtain that $\lambda \in D$, a contradiction.
\item If $\car(\F)=2$ we directly find that $D$ contains $\tr_{\L/\F}(\lambda)\lambda$ for all $\lambda \in \F$, and as the kernel of the trace is $\F$
we find that $D$ contains every vector of $\L \setminus \F$, again a contradiction.
\end{itemize}
Hence $Q(x)=0$ for all $x \in V$. Then $\tr_{\L/\F}(h(x,x))=s(x,x)=0$ for all $x \in V$, yielding that the sesquilinear form $h$
is skew-Hermitian. And finally we have $s(x,\lambda x)=\tr_{\L/\F}(h(x,x)\lambda)$ for all $x \in V$ and all $\lambda \in \L$.
The uniqueness statement for $h$ is proved exactly as in the proof that $K \circ S$ is injective.
\end{proof}

\begin{cor}
Let $\calS$ be an affine subspace of $\calA^2(V)$ with translation vector space $W_\L$.
Then there is a unique skew-Hermitian form $h$ on $V^\L$ such that
$$\forall x \in V, \; Q_\calS(x)=h(x,x).$$
\end{cor}

Conversely, for every skew-Hermitian form $h$ on $V^\L$, note that $s_0 : (x,y) \mapsto \tr_{\L/\F}(h(x,y))$
is an alternating bilinear form on $V$. It is then easily seen that the
fundamental form of the affine space $s_0+W_\L$ is $x \mapsto h(x,x)$.
Hence, affine subspaces of $\calA^2(V)$ with translation vector space $W_\L$ are in one-to-one correspondence with
skew-Hermitian forms on $V^\L$, and those that are symplectic are in one-to-one correspondence with
the nonisotropic skew-Hermitian forms on $V^\L$.

Finally, let $\calS$ and $\calS'$ be affine subspaces of $\calA^2(V)$ with translation vector space $W_\L$,
and denote by $h$ and $h'$ the corresponding skew-Hermitian forms.
Now the key is that if $\varphi : V \rightarrow V$ is a semi-linear automorphism of $V^\L$, then
$(x,y) \mapsto \sigma(h(\varphi(x),\varphi(y)))$ is sesquilinear and its diagonal form is
$x \mapsto \sigma(Q_\calS(\varphi(x)))$. Then by Theorem \ref{theo:classifycongruencebyfundamental}, $\calS$ is congruent to $\calS'$ if and only if $Q_\calS$ is quasi-equivalent to $Q_{\calS'}$, and there are two possibilities for such a quasi-equivalence:
\begin{itemize}
\item There is a linear equivalence from $Q_\calS$ to $Q_{\calS'}$, which means that the skew-Hermitian forms $h$ and $h'$ are equivalent in the traditional
meaning of the word;
\item There is a semi-linear equivalence from $Q_\calS$ to $Q_{\calS'}$, which means that the skew-Hermitian form $h'$ is equivalent, in the traditional meaning of the word, to
the sesquilinear form $(x,y) \mapsto \sigma(h(\varphi(x),\varphi(y)))$. By taking a diagonal basis of the latter, we see that
it is equivalent to $-h$ in the traditional meaning of the word.
\end{itemize}
Hence, $\calS$ is congruent to $\calS'$ if and only if $h \simeq \pm h'$.
The general case (where $\calS$ and $\calS'$ might not have equal translation vector spaces) follows immediately,
and this completes the proof of Theorem \ref{theo:refinedseparable}.

\begin{Rem}
In the inseparable case, we have found no analogue of skew-Hermitian forms as a way to represent the fundamental form of the subspace $\calS$.
In that case, the only upside is that the discussion on congruence is simplified because $\L$ has only one automorphism over $\F$.
\end{Rem}

\subsection{Trivial spectrum subspaces}\label{section:WLtrivialspectrum}

So far, in this section we have only discussed symplectic affine subspaces, and now it is time to deal with the trivial spectrum subspaces.
Due to the connection between the two problems, the discussion will be very quick.

Here, we will have to assume that $|\F|>2n-2$.
We keep a quadratic extension $\L$ of $\F \id_V$ in $\End(V)$, and we assume that we have a symplectic form $s$ on $V$
such that $(x,y) \mapsto s(x,\lambda y)$ is nonisotropic for some (i.e., for all) $\lambda \in \L \setminus \F$.
We also assume that $\dim V \geq 4$.
Choose $a \in \L \setminus \F$ and set $b : (x,y) \mapsto s(x,a(y))$.
Set also $P:=\{(x,y) \in V^2 \mapsto s(x,\lambda y) \mid \lambda \in \L\}$, and then
$$\calU_{s,\L}:=\calA_P=\calA_s \cap \calA_b.$$
First of all, we need to prove that $\calU_{s,\L}$ is an irreducible optimal trivial spectrum subspace of $\calA_s$.
Obviously it is included in $\calA_s$. Moreover, it is included in $\calA_b$, with $b$ nonisotropic, so it has trivial spectrum.
Next, note that $\calU_{s,\L}=\Phi(W_\L)$ where $\Phi : c \in \calA^2(V) \mapsto R_s^{-1} R_c \in \calA_s$.
From there, we find $\dim \calU_{s,\L}=\dim W_\L=n(n-1)$, and by Lemma \ref{lemma:Sadaptedsubspaces}
the irreducibility of $\Phi(W_\L)$ ensues from the one of $s+W_\L$ (which is proved in Proposition \ref{prop:WLirreducible}).

Next, we note that the algebra $\L$ can be recovered from the structure of $\calU_{s,\L}$ and from $s$, as follows.
By referring to Proposition \ref{prop:recoverL} and applying the correspondence between affine subspaces and trivial spectrum subspaces, one obtains indeed
\begin{equation}\label{eq:deLaP}
P=\Alt(\calU_{s,\L}),
\end{equation}
which is sufficient to recover $\L$ knowing $\calU_{s,\L}$ and $s$.

At this point, it only remains to understand when two spaces
$\calU_{s,\L}$ and $\calU_{s,\L'}$ are conjugated through an element of the symplectic group $\Sp(s)$ of $s$.
Let $\varphi \in \Sp(s)$. Note already that $\varphi \calA_s \varphi^{-1}=\calA_s$ since $\varphi$ is $s$-symplectic.
Next, set $b'' : (x,y) \mapsto s\left(x,(\varphi a \varphi^{-1}(y))\right)$. For every $u \in \End(V)$, we see that
\begin{align*}
u \in \calA_{b''} & \Leftrightarrow \forall x \in V, \; s(x,(\varphi a \varphi^{-1} u)(x))=0 \\
& \Leftrightarrow \forall x' \in V, \; s(\varphi(x'),(\varphi a (\varphi^{-1} u \varphi)(x'))=0 \\
& \Leftrightarrow \forall x' \in V, \; s(x',(a \varphi^{-1} u \varphi)(x'))=0 \\
& \Leftrightarrow \varphi^{-1} u \varphi \in \calA_b
\end{align*}
Hence $\varphi \calU_{s,\L} \varphi^{-1}=\calU_{s,\varphi \L \varphi^{-1}}$.
Next, if $\calU_{s,\L^{(1)}}=\calU_{s,\L^{(2)}}$ for quadratic extensions $\L^{(1)}$ and $\L^{(2)}$ of $\F$ inside $\End(V)$,
we deduce from \eqref{eq:deLaP} applied to such extensions that
$\{(x,y) \mapsto s(x,u_1(y)) \mid u_1 \in \L^{(1)}\}=\{(x,y) \mapsto s(x,u_2(y)) \mid u_2 \in \L^{(2)}\}$, and hence $\L^{(1)}=\L^{(2)}$.

Hence, we can conclude that $\calU_{s,\L}$ and $\calU_{s,\L'}$ are conjugated through an element of $\Sp(s)$ if and only if
$\L$ and $\L'$ are conjugated through an element of $\Sp(s)$. This requires that these extensions be isomorphic.
Now, assume that they are, and take $a \in \L \setminus \F$ and $a' \in \L' \setminus \F$ with the same minimal polynomial over $\F$
(i.e., the same norm and trace over $\F$).
\begin{enumerate}[(i)]
\item If $\L$ and $\L'$ are inseparable over $\F$ then $\calU_{s,\L}$ and $\calU_{s,\L'}$ are conjugated through an element of $\Sp(s)$
if and only if $a$ and $a'$ are conjugated through an element of $\Sp(s)$.
\item If $\L$ and $\L'$ are separable over $\F$ and we denote by $\sigma$ is the non-identity automorphism of $\L$,
then $\calU_{s,\L}$ and $\calU_{s,\L'}$ are conjugated through an element of $\Sp(s)$
if and only if $a'$ is conjugated to $a$ or to $\sigma(a)$ through an element of $\Sp(s)$.
\end{enumerate}

We suspect that this problem can be solved, in the separable case, by standard classification methods
(by ``to be solved" in this context is generally meant ``to be reduced to classification of Hermitian forms over arbitrary skew field extensions of $\L$", see,
e.g., \cite{Sergeichuk} for that matter).

We finish with some remarks on the structure of $\calU_{s,\L}$. First of all, Lemma \ref{lemma:transitivity} has a nice translation:

\begin{cor}\label{cor:transitivityTspectrum}
For all $x \in V \setminus \{0\}$, one has $\L x \oplus \calU_{s,\L} x=V$.
\end{cor}

\begin{proof}
Let $x \in V \setminus \{0\}$. Note that $\dim_\F (\calU_{s,\L} x)=2n-2$ by translating the first point in Lemma \ref{lemma:transitivity}.
Hence, to conclude it suffices to see that $\L x \cap \calU_{s,\L} x=\{0\}$. So, take $\lambda \in \L$ and assume that $\lambda x \in \calU_{s,\L} x$.
Then $\lambda x=u(x)$ for some $u \in \calU_{s,\L}$. Due to the definition of $\calU_{s,\L}$ this would lead to
$s(x,\mu \lambda x)=0$ for all $\mu \in \L$. This is possible only if $\lambda=0$ due to the non-isotropy assumption.
Hence $\L x \cap \calU_{s,\L} x=\{0\}$.
\end{proof}

We finish with an observation on the possible ranks in spaces of type $\calU_{s,\L}$, to be used in the last section of the article.

\begin{lemma}
Let $V$ be a vector space with dimension $2n \geq 4$.
Let $\L \subseteq \End(V)$ be a quadratic extension of $\F$. Then
the possible ranks for the elements of $W_\L$ are the multiples of $4$ in $\lcro 0,2n\rcro$.
\end{lemma}

\begin{proof}
We consider the $\L$-vector space $V^\L$ and a nonzero linear form $e$ on it.
Given an alternating $\L$-bilinear form $B$ on $V^\L$, remember from Proposition \ref{prop:radicaleB}
that $e\circ B : (x,y) \mapsto e(B(x,y))$ has the same radical as $B$.
As a consequence $\rk_\F (e \circ B)=2n-\dim_\F \Rad(e \circ B)=2n-\dim_\F \Rad(B)=2n-2\dim_\L \Rad(B)=2\rk_\L B$.

Since the possible ranks of the alternating $\L$-bilinear forms on $V_\L$ are the even integers in $\lcro 0,n\rcro$, the conclusion follows by varying $B$.
\end{proof}

\begin{cor}\label{cor:evenranks}
Let $(V,s)$ be a symplectic space of dimension $2n\geq 4$, and $a \in \End(V)$ be a quadratic element with irreducible minimal polynomial, such that
$b : (x,y) \mapsto s(x,a(y))$ is nonisotropic. Then the possible ranks for the elements of $\calU_{s,\F[a]}$ are the
multiples of $4$ in $\lcro 0,2n\rcro$.
\end{cor}

\section{Spaces with a $2$-dimensional non-degenerate alternator space}\label{section:froma2dimalternator}

Throughout this part, we keep the traditional assumptions of this article, i.e.,
we let $(V,s)$ be a symplectic space with dimension $2n\geq 4$, and we assume that $|\F|>2n-2$.

\begin{Def}
A subspace $S \subseteq \calA_s$ is said to have \textbf{nondegenerate alternator space} when all the
nonzero elements of $\Alt(S)$ are nondegenerate.
\end{Def}

Our goal is to prove the following partial result on the irreducible optimal trivial spectrum spaces of $s$-alternating endomorphisms.

\begin{theo}\label{theo:analysisnondegalt2}
Let $S$ be an optimal trivial spectrum linear subspace of $\calA_s$.
Assume that $S$ has nondegenerate alternator space with dimension $2$.
Then:
\begin{enumerate}[(a)]
\item There exists a quadratic element $a$ of $\End(V)$, with irreducible minimal polynomial, such that
$\Alt(S)=\Vect(s,b)$ for $b : (x,y) \mapsto s(x,a(y))$.
\item All the forms in $\Alt(S) \setminus \F s$ are nonisotropic.
\item One has $S=\calA_s \cap \calA_b$.
\item The algebra $\F[a]$ is uniquely determined by $S$.
\end{enumerate}
\end{theo}

The proof of this theorem is spread over Sections \ref{section:transitivityissue} to \ref{section:wrapup}.

In this section, given a bilinear form $b : V^2 \rightarrow \F$ and a vector $x \in V$, we denote by $\{x\}^{\bot_b}$
the \emph{right}-$b$-orthogonal of $x$, i.e., the kernel of the linear form $b(x,-)$
(beware not to confuse it with the kernel of $b(-,x)$, since we allow $b$ to be neither symmetric nor alternating).

\subsection{The transitivity problem}\label{section:transitivityissue}

\begin{lemma}\label{lemma:norank2}
Let $S$ be a linear subspace of $\calA_s$. Assume that
there exists a $2$-dimensional linear subspace $P \subseteq \Alt(S)$ whose nonzero elements are all nondegenerate.

Then $S$ contains no rank $2$ operator.
\end{lemma}

\begin{proof}
Assume on the contrary that $S$ has a rank $2$ element $u$ with range denoted by $\Vect(x,y)$.
By Proposition \ref{prop:alttensor}, it follows that $S$ contains $x \wedge_s y$. Choose $b \in P \setminus \F s$.
Then $x \wedge_s y$ is $b$-alternating, and $b$ is non-degenerate. Hence Proposition \ref{prop:alttensor} shows that
$x \wedge_b y=\lambda\, x \wedge_s y$ for some $\lambda \in \F \setminus \{0\}$. Comparing the coefficients on $x$ yields
$b(-,y)=\lambda s(-,y)$, and hence $b-\lambda s$ is degenerate. Yet $b-\lambda s \in P$, and hence $b-\lambda s=0$, which
is contradictory. Therefore $S$ has no rank $2$ element.
\end{proof}

\begin{prop}[Transitivity property]\label{lemma:transitivitypropoptimal}
Let $S$ be an optimal trivial spectrum linear subspace of $\calA_s$.

Assume that there exists a $2$-dimensional linear subspace $P \subseteq \Alt(S)$ whose nonzero elements are all nondegenerate.
Then, for all $x \in V \setminus \{0\}$,
$$\dim(Sx)=2n-2 \quad \text{and} \quad Sx=\underset{b \in P}{\bigcap} \Ker b(x,-).$$
\end{prop}

\begin{proof}
Let $x \in V \setminus \{0\}$. Let $b \in P \setminus \F s$.
The linear form $b(x,-)$ cannot vanish on $\{x\}^{\bot_s}$ because otherwise, for some $\lambda \in \F$, the form $b+\lambda s$
would have $x$ in its left radical, and hence would be degenerate.
Hence $\dim(\{x\}^{\bot_b} \cap \{x\}^{\bot_s}) \leq 2n-2$.
Since every element of $S$ is both $s$-alternating and $b$-alternating, it must map $x$ into $\{x\}^{\bot_b} \cap \{x\}^{\bot_s}$, which leads to
$$\dim (Sx) \leq 2n-2.$$
Now, set
$$S':=\{u \in S : u(x)=0\}$$
and denote by $\overline{s}$ the symplectic form on $W:=\{x\}^{\bot_s}/\F x$ induced by $s$.
Every $u \in S'$ induces an $\overline{s}$-alternating endomorphism $\overline{u}$ of $W$,
and the range of the mapping $\Phi : u \in S' \mapsto \overline{u} \in \calA_{\overline{s}}$ has trivial spectrum.
Hence $\rk \Phi \leq (n-1)(n-2)$ by Theorem \ref{theo:majodimTS}.
Moreover $\Phi$ is injective: indeed every $u \in \Ker \Phi$ has rank at most $2$ (indeed if we choose $y \in V \setminus \{x\}^{\bot_s}$
then $\im u=\F u(y)+u(\{x\}^{\bot_s}) \subseteq \F u(y)+\F x$), and since its rank is even (as it is $s$-alternating)
it equals $0$ or $2$; yet the second possibility is ruled out by Lemma \ref{lemma:norank2}, and hence $\Ker \Phi=\{0\}$.
Therefore $\dim S'=\rk \Phi  \leq (n-1)(n-2)$.
By the rank theorem we conclude that
$$\dim (Sx)=\dim S-\dim S' \geq 2n-2.$$
It follows that $\dim (Sx)=2n-2$ and
$$Sx=\{x\}^{\bot_b} \cap \{x\}^{\bot_s}=\underset{c \in P}{\bigcap} \Ker c(x,-).$$
\end{proof}

\subsection{Sorting out the alternator space}

\begin{lemma}\label{lemma:nonisotropic}
Let $S$ be an optimal trivial spectrum linear subspace of $\calA_s$ with nondegenerate alternator space of dimension $2$.
Then all the forms in $\Alt(S) \setminus \F s$ are nonisotropic.
\end{lemma}

\begin{proof}
Let $b \in \Alt(S) \setminus \F s$.
Let $x \in V \setminus \{0\}$. By Lemma \ref{lemma:transitivitypropoptimal} we have $Sx = \{x\}^{\bot_s} \cap \{x\}^{\bot_b}$.
Yet $S$ has trivial spectrum, so $x \not\in Sx$, and since $s(x,x)=0$ we must have $b(x,x) \neq 0$.
\end{proof}

\begin{lemma}\label{lemma:existquadratic}
Let $S$ be an optimal trivial spectrum linear subspace of $\calA_s$ with nondegenerate alternator space of dimension $2$.
Then there exists a quadratic element $a$ of $\End(V)$, with irreducible minimal polynomial, such that
$\Alt(S)=\Vect(s,b)$, where $b : (x,y) \mapsto s(x,a(y))$.
\end{lemma}

\begin{proof}
Choose $b \in \Alt(S) \setminus \F s$, and consider the unique $v \in \End(V)$ such that
$s(v(x),y)=b(x,y)$ for all $(x,y)\in V^2$.

We shall prove that $x,v(x),v^2(x)$ are linearly dependent for all $x \in V \setminus \{0\}$.

Let $x \in V \setminus \{0\}$. Note that $v(x) \neq 0$ since $b$ is nondegenerate.
We already know that $S x=\{x\}^{\bot_s}\cap \{x\}^{\bot_b}=\{x\}^{\bot_s}\cap \{v(x)\}^{\bot_s}
=\{x,v(x)\}^{\bot_s}$.
Let $u \in S$. Since $u$ is $s$-alternating, and hence $s$-selfadjoint, the fact that it maps $x$ into $\{x,v(x)\}^{\bot_s}$
implies that it maps $\{x,v(x)\}$ into $\{x\}^{\bot_s}$. In particular this yields
$S v(x) \subseteq \{x\}^{\bot_s}$. Yet we know that $S v(x) = \{v(x),v^2(x)\}^{\bot_s}$. By double-orthogonality, this yields
$x \in \Vect(v(x),v^2(x))$. It follows that $x,v(x),v^2(x)$ are linearly dependent.

Since this holds for all $x \in V \setminus \{0\}$, by Frobenius reduction theory we deduce that $v^2$ is a linear combination of $\id_V$ and $v$.
By taking the $s$-adjoint $a$ of $v$, we conclude that $b : (x,y) \mapsto s(x,a(y))$ and that $a$ is quadratic.

Finally, since every nonzero linear combination of $b$ and $s$ is nondegenerate, we see that $a$ has no eigenvalue in $\F$,
and hence its minimal polynomial is irreducible.
\end{proof}

\subsection{Conclusion}\label{section:wrapup}

We can now complete the proof of Theorem \ref{theo:analysisnondegalt2}. So, let $S$ be an optimal trivial spectrum linear subspace of $\calA_s$
and assume that $S$ has nondegenerate alternator of dimension $2$.

Then point (a) in Theorem \ref{theo:analysisnondegalt2} has been proved as Lemma \ref{lemma:existquadratic},
and point (b) as Lemma \ref{lemma:nonisotropic}. Choose $b \in \Alt(S) \setminus \F s$ and a corresponding quadratic endomorphism $a$ of $V$.
As a consequence of the first two points, the space $\calA_s \cap \calA_b$ equals $\calU_{s,\F[a]}$
and hence has dimension $n(n-1)$. Since $S \subseteq \calA_s \cap \calA_b$ we deduce that $S=\calA_s \cap \calA_b$.
Finally, the equality $\Alt(S)=\Vect(s,b)$ allows us to view $\F[a]$
as the set of all $v \in \End(V)$ for which the form $(x,y) \mapsto s(x,v(y))$ belongs to $\Alt(S)$,
which yields point (d).
This completes the proof of Theorem \ref{theo:analysisnondegalt2}.

At this very point, we have almost proved all the results announced in the introduction.
The canonical decomposition has been fully investigated in Section \ref{section:reduction},
the structure of the $W_\L$ spaces and the associated symplectic affine spaces have been
elucidated in Section \ref{section:WL}, and we have just shown that the counterpart of these spaces
for trivial spectrum subspaces of $s$-alternating endomorphisms are associated with
irreducible optimal trivial spectrum subspaces that have non-degenerate alternator space of dimension $2$.

The only missing key is the proof that every irreducible optimal trivial spectrum subspace
has such an alternator space: it will be elucidated in Section \ref{section:lastkeyprop}.
The next section deals with the technical machinery that is necessary for this proof.

\section{Intermezzo: A review of tools from the theory of spaces of matrices with bounded rank}\label{section:tools}

Before we can prove the last main result in the study of irreducible optimal trivial spectrum subspaces of alternating endomorphisms,
we must review much technical material. The reader acquainted with \cite{AtkinsonPrim,dSPLLD2,dSPaffineunits} will recognize the same tools,
as the only new result here is Lemma \ref{lemma:Factor2}. See also \cite{HuangLandsberg} for recent advances in these methods.

\subsection{The tangent space lemma, and the Flanders-Atkinson lemma}\label{section:basiclemmas}

We recall the following classical lemma (see \cite{AtkinsonPrim} for its first appearance in the literature, and Section 2 of \cite{dSPLLD1} for a simplified proof):

\begin{lemma}[Flanders-Atkinson Lemma]\label{lemma:FLA}
Let $n,p,r$ be integers with $0<r \leq \min(n,p)$. Assume that $|\F|>r$.
Let $J_r:=\begin{bmatrix}
I_r & 0 \\
0 & 0
\end{bmatrix}$ and $M=\begin{bmatrix}
A & C \\
B & D
\end{bmatrix}$ belong to $\Mat_{n,p}(\F)$, with $A \in \Mat_r(\F)$ and so on.
Assume that $\rk(sJ_r+t M) \leq r$ for all $(s,t) \in \F^2$. Then
$$D=0 \quad \text{and} \quad \forall k \geq 0, \; B A^kC=0.$$
\end{lemma}

The following corollary is simply obtained by exploiting the identity $D=0$ in the Flanders-Atkinson lemma:

\begin{lemma}[Tangent Space Lemma]\label{lemma:tangent}
Let $U$ and $V$ be finite-dimensional vector spaces over $\F$, and $S$ be a linear subspace of $\Hom(U,V)$.
In $S$, take an element $u_0$ of maximal rank $r$, and assume that $|\F|>r$.
Then every element of $S$ maps $\Ker u_0$ into $\im u_0$.
\end{lemma}

\subsection{Basics on operator spaces and dual operator spaces}

Here we collect some basic definitions.

\begin{Def}
Let $U,V$ be vector spaces over $\F$, and $S$ be a linear subspace of $\Hom(U,V)$. We say that $S$ is
\textbf{target-reduced} when no linear hyperplane $H$ of $V$ includes the range of every element of $S$, i.e., the sum of the ranges
of the elements of $S$ is $V$.
\end{Def}

\begin{Def}
Let $U,V$ be vector spaces over $\F$, and $S$ be a linear subspace of $\Hom(U,V)$.
The \textbf{dual operator space} of $S$ is the linear subspace $\widehat{S}$ of $\Hom(S,V)$ whose elements are the evaluation mappings
$$\widehat{x} : u \in S \mapsto u(x) \in V, \quad \text{with $x \in U$.}$$
\end{Def}

It is straightforward that $\widehat{S}$ is target-reduced if and only if $S$ is target-reduced.
The proof of Theorem \ref{theo:majdim} in \cite{dSPaffinealt} makes critical use of the dual operator space, and our proof of Proposition \ref{prop:keyprop}
will considerably refine the argument.

\subsection{The generic matrix machinery (I)}\label{section:genericmachinery}

Here we review the generic matrix techniques that were introduced by Atkinson and Lloyd
\cite{AtkinsonPrim,AtkLloydPrim} (see also \cite{dSPLLD2}). A more detailed account of the basic points is given in section
5.5 of \cite{dSPaffineunits}, and here we will skip some details.

Let $m,p$ be positive integers.
Let $E$ be an $\F$-vector space. We consider its dual vector space $E^\star:=\Hom_\F(E,\F)$
and the symmetric algebra $R=\Sym(E^\star)$ (defined as the external direct sum of the symmetric tensor spaces $\Sym^n(E^\star)$ for $n \geq 0$)
and its field of fractions $\mathbb{L}$. We naturally view $\F$ as a subfield of $\mathbb{L}$.
We will say that $E$ is the \textbf{parameter space}.
For all $d \geq 0$, we simply set $R_d:=\Sym^d(E^\star)$, the subspace of all $d$-homogeneous elements of $R$.

In particular $R_1$ is the space $\Sym^1(E^\star)=E^\star$ of all symmetric $1$-tensors.
We have a canonical ring homomorphism $\Sym(E^\star) \rightarrow \mathrm{Pol}(E,\F)$ to the ring of polynomial functions on $E$
(it is an isomorphism if $\F$ is infinite). Hence for every $\mathbf{p} \in S(E^\star)$ and every $z \in E$ we can consider the specialization
$\mathbf{p}(z) \in \F$. More generally, given a matrix $\mathbf{M} \in \Mat_{m,p}(R)$ and a vector $z \in E$, we denote by
$\mathbf{M}(z)$ the matrix obtained by specializing all the entries of $\mathbf{M}$ at $z$.

Let $\calM$ be an $\F$-linear subspace of $\Mat_{m,p}(\F)$. A \textbf{generic matrix} for $\calM$ (with \textbf{parameter space} $E$) consists of a matrix
$\mathbf{M}$ with entries in $R_1$
such that
$$\calM=\{\mathbf{M}(z) \mid z \in E\}.$$
In particular, it is easy to see that $\mathbf{M}$ then belongs to $\Vect_{\mathbb{L}}(\calM)$.

\begin{Not}
For a nonempty subset $\calM$ of $\Mat_{m,p}(\F)$, we denote the greatest possible rank in $\calM$ by
$$\maxrk(\calM):= \max \{\rk M \mid M \in \calM\}.$$
\end{Not}

A critical property of generic matrices is that they are well-suited to study spaces of matrices with bounded rank if the underlying field is large enough:

\begin{prop}\label{prop:maxrank}
Let $\calM$ be an $\F$-linear subspace of $\Mat_{m,p}(\F)$.
Assume that $|\F|>\maxrk(\calM)$.
Then
$$\maxrk(\Vect_\mathbb{L}(\calM))=\maxrk(\calM).$$
\end{prop}

\begin{cor}\label{cor:genericrank}
Let $\calM$ be an $\F$-linear subspace of $\Mat_{m,p}(\F)$, and $\mathbf{M}$ be a generic matrix of it.
Assume that $|\F|>\maxrk(\calM)$.
Then
$$\rk_\mathbb{L} \mathbf{M}=\maxrk(\calM).$$
\end{cor}
Let us now discuss row matrices over $R$.
Let $\mathbf{L} \in \Mat_{1,n}(R_1)$.
The \textbf{spanning rank} of $\mathbf{L}$ is defined as
$$\sprk \mathbf{L}:=\dim_\F \{\mathbf{L}(z) \mid z \in E \}.$$
Of course, one should not confuse this with the rank of $\mathbf{L}$, defined as its rank as a matrix of $\Mat_{n,1}(\mathbb{L})$.
The rank of $\mathbf{L}$ equals $0$ or $1$, whereas its spanning rank is an integer between $0$ and $n$.

The following lemma will be particularly important in our study ; it is critical in Atkinson's theorem on primitive spaces of bounded rank matrices
\cite{AtkinsonPrim}.

\begin{lemma}[First Factorization Lemma, see lemma 5.2 in \cite{dSPLLD2}]\label{lemma:Factor1}
Let $n \geq 1$ and $d \geq 1$, and let $\mathbf{X}$ and $\mathbf{Y}$ be two rows in $\Mat_{1,n}(R)$ and assume:
\begin{enumerate}[(i)]
\item That the entries of $\mathbf{X}$ are $1$-homogeneous and the ones of $\mathbf{Y}$ are $d$-homogeneous.
\item That $\sprk(\mathbf{X})>1$.
\item That $\mathbf{X}$ and $\mathbf{Y}$ are linearly dependent in the $\mathbb{L}$-vector space $\Mat_{1,n}(\mathbb{L})$.
\end{enumerate}
Then $\mathbf{Y}=\mathbf{p}\,\mathbf{X}$ for some $\mathbf{p} \in R_{d-1}$.
\end{lemma}

Our proof of Proposition \ref{prop:keyprop} will require the following variation of the First Factorization Lemma, which we state and prove right away. It uses a similar strategy as in the original proof of Atkinson (in short: it uses arithmetics in polynomial rings), yet with a far more difficult discussion.

\begin{lemma}[Second Factorization Lemma]\label{lemma:Factor2}
Let $\mathbf{X}$ and $\mathbf{Y}$ be two vectors of $\Mat_{1,n}(R_1)$ and assume:
\begin{enumerate}[(i)]
\item That $\sprk(\mathbf{X})=n$.
\item That either $1<\sprk(\mathbf{Y})<n-1$, or $n \geq 3$ and all the nontrivial linear combinations of $\mathbf{X}$ and $\mathbf{Y}$
with coefficients in $\F$ have spanning rank $n$.
\end{enumerate}
Let $\mathbf{Z}$ be a vector of $\Mat_{1,n}(R_d)$ for some $d \geq 1$, and assume that
$\mathbf{Z}$ is a linear combination of $\mathbf{X}$ and $\mathbf{Y}$ with coefficients in $\L$.
Then $\mathbf{Z}=\mathbf{a} \mathbf{X}+\mathbf{b} \mathbf{Y}$ for some elements $\mathbf{a}$ and $\mathbf{b}$ of $R_{d-1}$.
\end{lemma}

The proof uses the following folklore result, which is sometimes attributed to Schur:
Given two nonzero vectors spaces $U$ and $V$ and a linear subspace $S$ of $\Hom(U,V)$ in which every element has
rank at most $1$, one of the following conditions holds:
\begin{enumerate}[(i)]
\item All the elements of $S$ map into a fixed $1$-dimensional linear subspace of $V$;
\item All the elements of $S$ vanish on a fixed linear hyperplane of $U$.
\end{enumerate}

\begin{proof}
By assumption there are elements $\mathbf{a}$ and $\mathbf{b}$ of $\L$ (i.e., fractions of elements of $R$) such that
$\mathbf{Z}=\mathbf{a} \mathbf{X}+\mathbf{b} \mathbf{Y}$. The main difficulty is to prove that $\mathbf{b}$ is a $(d-1)$-homogeneous element of
$R$. To this end, we write it in irreducible fashion as $\mathbf{b}=\frac{\mathbf{p}}{\mathbf{q}}$ where
$\mathbf{p}$ and $\mathbf{q}$ are relatively prime elements of $R$ (at this point we cannot yet guarantee that they are homogeneous, but it will be seen later on
that they are). Now, we consider the compound matrices
$$\mathbf{M}=\begin{bmatrix}
\mathbf{X} \\
\mathbf{Y}
\end{bmatrix}
\quad \text{and} \quad
\mathbf{M}'=\begin{bmatrix}
\mathbf{X} \\
\mathbf{Z}
\end{bmatrix}.$$
Let us consider a $2$-by-$2$ minor of $\mathbf{M}$, which we write $\Delta_{i,j}(\mathbf{M})$ for the column indices $i$ and $j$
(the row indices are necessarily $1$ and $2$). Hence we have
$$\Delta_{i,j}(\mathbf{M}')=\mathbf{b}\, \Delta_{i,j}(\mathbf{M}),$$
and since $\mathbf{q}$ is relatively prime with $\mathbf{p}$ it follows that $\mathbf{q}$ divides $\Delta_{i,j}(\mathbf{M})$.

Now, we introduce the $\F$-linear subspace $\Minor_{2,2}(\mathbf{M})$ of $\L$ spanned by all the $2$-by-$2$ minors of $\mathbf{M}$, and we will
simply prove that the only (nonzero) common divisors of the elements of $\Minor_{2,2}(\mathbf{M})$ in $R$ are the nonzero elements of $R_0=\F$.
For the remainder of the proof, it is useful to recall the classical fact that $\Minor_{2,2}(\mathbf{M})$ is invariant under right-multiplying $\mathbf{M}$
with an element of $\GL_n(\F)$. Towards a contradiction, we assume that we have an \emph{irreducible}
common divisor $\mathbf{d} \in R$ of the elements of $\Minor_{2,2}(\mathbf{M})$.
It is here that assumptions (i) and (ii) will be crucial, and we split the discussion into two cases.

\vskip 3mm
\noindent \textbf{Case 1: $1<\sprk(\mathbf{Y})<n-1$.} \\
Set $r:=\sprk \mathbf{Y}$.
Then, by right-multiplying $\mathbf{M}$ with a well-chosen matrix $P \in \GL_n(\F)$, we lose no generality in assuming that
$$\mathbf{M}=\begin{bmatrix}
\mathbf{a}_1 & \cdots & \mathbf{a}_r & \mathbf{a}_{r+1} & \cdots & \mathbf{a}_n \\
\mathbf{b}_1 & \cdots & \mathbf{b}_r & 0 & \cdots & 0
\end{bmatrix}$$
where $\mathbf{b}_1,\dots,\mathbf{b}_r$ are $\F$-linearly independent elements of $R_1$, and
$\mathbf{a}_1,\dots,\mathbf{a}_n$ are $\F$-linearly independent elements of $R_1$ (thanks to assumption (i)).
Then by taking the minors associated with the pairs of column indices $(r,r+1)$ and $(r-1,r+1)$, we see that $\mathbf{d}$ divides both
$\mathbf{b}_r\mathbf{a}_{r+1}$ and $\mathbf{b}_{r-1}\mathbf{a}_{r+1}$. Since $\mathbf{b}_r$ and $\mathbf{b}_{r-1}$
are irreducible and they are not scalar multiples of one another, they are relatively prime, and hence the only possibility is that $\mathbf{d}$ divides $\mathbf{a}_{r+1}$. Likewise, by taking the minors associated with the pairs of column indices $(r,r+2)$ and $(r-1,r+2)$, we find that $\mathbf{d}$ divides $\mathbf{a}_{r+2}$. Again, $\mathbf{a}_{r+1}$ and $\mathbf{a}_{r+2}$ are relatively prime, and this yields a contradiction.

\vskip 3mm
\noindent \textbf{Case 2: Every nontrivial linear combination of $\mathbf{X}$ and $\mathbf{Y}$ has spanning rank $n$, and $n \geq 3$.} \\
If all the specializations of $\mathbf{M}$ have rank at most $1$, then Schur's theorem applied to the space of all specializations of
$\mathbf{M}$ yields that either some nontrivial linear combination of $\mathbf{X}$ and $\mathbf{Y}$ over $\F$ equals zero, or
$\mathbf{X}$ and $\mathbf{Y}$ have spanning rank $1$. None of these outcomes is possible here, hence
$\Minor_{2,2}(\mathbf{M})$ is nonzero. As the elements of $\Minor_{2,2}(\mathbf{M})$ are $2$-homogeneous elements of $R$,
it follows that $\mathbf{d}$ is homogeneous with degree either $1$ or $2$.

\vskip 3mm
\noindent \textbf{Subcase 2.1: $\mathbf{d}$ has degree $2$.} \\
Then all the elements of $\Minor_{2,2}(\mathbf{M})$ are products of $\mathbf{d}$
with a scalar in $\F$. In turns, this shows that if two such minors are non-zero then they vanish exactly at the same points of the parameter space $E$.
We will see that this contradicts the above assumption on $\mathbf{X}$ and $\mathbf{Y}$. To see this, remember that $\Minor_{2,2}(\mathbf{M})$ is nonzero,
and hence at least one specialization of $\mathbf{M}$ has rank $2$.
Hence, after right-multiplying $\mathbf{M}$ with a well-chosen matrix of $\GL_n(\F)$, we can assume that some specialization of $\mathbf{M}$ equals
$$\begin{bmatrix}
1 & 0 & 0 & \cdots & 0 \\
0 & 1 & 0 & \cdots & 0
\end{bmatrix}.$$
In particular, at such a parameter the specialization of the first minor $\Delta_{1,2}(\mathbf{M})$ does not vanish, but the specialization vanishes for
\emph{all} the other minors, and hence the other minors are identically zero on the parameter space. In particular, removing the first column of $\mathbf{M}$
to obtain a new matrix $\mathbf{M}'$, we find $\rk \mathbf{M}'\leq 1$.
Now, denote by $\mathbf{R_1}$ and $\mathbf{R_2}$ the rows of $\mathbf{M}'$, and note that both have spanning rank $n-1$, as
well as every nontrivial linear combination of them.
Just like in the above, this is in contradiction with Schur's classification of vector spaces of matrices with rank at most $1$, because $n-1 \geq 2$.

\vskip 3mm
\noindent \textbf{Subcase 2.2: $\mathbf{d}$ has degree $1$.} \\
Then $\mathbf{d}$ is a nonzero linear form $\varphi$ on the parameter space $E$.
For all $z \in \Ker \varphi$, all the elements of $\Minor_{2,2}(\mathbf{M})$ vanish at $z$, and hence
$\rk \mathbf{M}(z) \leq 1$. Again, Schur's theorem applied to the matrix space $\{\mathbf{M}(z) \mid z \in \Ker \varphi\}$
yields that either there exists $(\alpha,\beta) \in \F^2 \setminus \{(0,0)\}$ such that
$\alpha \mathbf{X}+\beta \mathbf{Y}$ vanishes at every parameter $z \in \Ker \varphi$, or
the contrary holds and $\dim \{\mathbf{X}(z) \mid z \in \Ker \varphi\} \leq 1$, to the effect that
$\dim \{\mathbf{X}(z) \mid z \in E\} \leq 2$. In each case, we find a nontrivial linear combination of
$\mathbf{X}$ and $\mathbf{Y}$ with spanning rank less than $n$, which contradicts our assumption for Case 2.

\vskip 3mm
This final contradiction shows that every common divisor in $R$ of the elements of $\Minor_{2,2}(\mathbf{M})$
is constant. Now the conclusion is nearby. We come back to the fraction $\mathbf{b}=\frac{\mathbf{p}}{\mathbf{q}}$, remember
from the start of the proof that $\mathbf{q}$ divides all the $2$-by-$2$ minors of $\mathbf{M}$, and deduce that
$\mathbf{q} \in \F$. Therefore $\mathbf{b} \in R$.

Assume for a moment that $\mathbf{b}$ is nonzero.
We come back to the identity $\Delta_{i,j}(\mathbf{M}')=\mathbf{b} \Delta_{i,j}(\mathbf{M})$ for all $i<j$ in $\lcro 1,n\rcro$, and we choose
$i,j$ so that $\Delta_{i,j}(\mathbf{M}) \neq 0$ (this is possible because of what we have just shown of $\Minor_{2,2}(\mathbf{M})$).
Hence $\Delta_{i,j}(\mathbf{M}')\neq 0$. Yet $\Delta_{i,j}(\mathbf{M}')$ is obviously homogeneous with degree $d+1$, while
$\Delta_{i,j}(\mathbf{M})$ is homogeneous with degree $2$. Hence $\mathbf{b}$ is homogeneous with degree $d-1$, and
this property still holds in case $\mathbf{b}=0$.

The same method applied to $\mathbf{M}'':=\begin{bmatrix}
\mathbf{Z} \\
\mathbf{Y}
\end{bmatrix}$ yields that $\mathbf{a}$ is a $(d-1)$-homogeneous element of $R$, which completes the proof.
\end{proof}

\subsection{The generic matrix machinery (II): Catchers vs Alternators}\label{section:generic2}

This last part of our introduction to generic matrices will be spent discussing left annihilators of generic matrices and their relationship with
alternators.

\begin{Def}
Let $\mathbf{M} \in \Mat_{m,p}(R)$. A \textbf{left-annihilator} of $\mathbf{M}$ is a row matrix $\mathbf{L} \in \Mat_{1,m}(\mathbb{L})$
such that that $\mathbf{L}\mathbf{M}=0$. A \textbf{catcher} of $\mathbf{M}$ is a left-annihilator of $\mathbf{M}$
whose entries are in $R_1$.
We denote by $\catch(\mathbf{M})$ the $\F$-linear subspace of $\Mat_{1,m}(R)$ that consists of all the catchers of $\mathbf{M}$,
and call it the \textbf{catcher space} of $\mathbf{M}$.
\end{Def}

Catchers are intimately connected with alternators, as was carefully laid out in section 4.3 of \cite{dSPaffineunits}, and here we briefly recall
this relationship.

Let us take a finite-dimensional vector space $U$ and a linear subspace $\calV$ of $\End(U)$,
and let $\bfB_\calV=(v_1,\dots,v_n)$ and $\bfB_U=(e_1,\dots,e_p)$ be respective bases of $\calV$ and $U$.
We consider the matrix space $\calM$ that represents the dual operator space $\widehat{\calV}$ in $\bfB_\calV$ and $\bfB_U$, and
we take an arbitrary generic matrix $\mathbf{M}$ that is associated with $\calM$, with parameter space denoted by $E$.

We have the surjective linear mapping $\Phi$ assign to every $z \in E$ the operator in $\widehat{\calV}$ whose matrix in the bases
$\bfB_\calV$ and $\bfB_U$ is $\mathbf{M}(z)$, and the mapping
$\widehat{(-)}$ assign to every $x \in U$ the evaluation mapping $\widehat{x} \in \widehat{\calV}$, and we represent them in the following diagram:
$$\xymatrix{ & E \ar@{>>}[d]^\Phi  \\
U \ar@{>>}[r]^{\widehat{(-)}} & \widehat{\calV}.
}$$
Now, \emph{we assume that $\calV$ is target-reduced}, and we briefly recall (see section 4.3 of \cite{dSPaffineunits} for a carefully detailed account)
the construction of a vector space isomorphism:
$$\Lambda : \Alt(\calV) \overset{\simeq}{\longrightarrow} \catch(\mathbf{M}).$$
Let $b \in \Alt(\calV)$. Using the fact that $\calV$ is target-reduced, one proves that the
mapping $L_b : x \in U \mapsto b(x,-) \in U^\star$ factorizes as $\Theta \circ \widehat{(-)}$ for a unique linear mapping $\Theta : \widehat{\calV} \rightarrow U^\star$,
then we set $\Psi:=\Theta \circ \Phi$ and obtain the commutative diagram
$$\xymatrix{ & E \ar@{>>}[d]^\Phi \ar@{-->}[dr]^{\Psi} & \\
U \ar@{>>}[r]^{\widehat{(-)}} \ar@/_1pc/[rr]_{L_b} & \widehat{\calV} \ar@{-->}[r]^{\Theta}  & U^\star.
}$$
One defines $\Lambda(b)$ as the matrix $\mathbf{L} \in \Mat_{1,p}(R_1)$ whose entries are the composite linear forms $z \in E \mapsto \Psi[z](e_i)$, with $i \in \lcro 1,p\rcro$.
In terms of specializations, this means that, for every parameter $z$, $\mathbf{L}(z)$ is the matrix of the linear form $\Psi(z)$ in the basis $(e_1,\dots,e_p)$.
Using the fact that $b$ is an alternator of $\calV$, one checks that $\Lambda(b)$ is a catcher of $\mathbf{M}$.
Then one also checks that the mapping $\Lambda$ is a vector space isomorphism.

We also note that the rank and radicals of $b$ are intimately connected with $\mathbf{L}$. Specifically, it is easily checked
that
\begin{equation}\label{eq:Rradequation}
\sprk (\mathbf{L})=\rk(b)
\end{equation}
and that the right radical $\Rrad(b)$ is the set of all vectors $x \in U$ whose associated matrix $X \in \F^p$ in the basis $(e_1,\dots,e_p)$
satisfies $\mathbf{L} X=0$.
Finally, we can also understand the left radical $\Lrad(b)$ in terms of $\mathbf{L}$:
letting $x \in U$, and letting $z \in E$ be such that $\Phi(z)=\widehat{x}$,
we have
\begin{equation}\label{eq:Lradequation}
x \in \Lrad(b) \Leftrightarrow L_b(x)=0 \Leftrightarrow \Theta(\widehat{x})=0 \Leftrightarrow \Psi(z)=0 \Leftrightarrow \mathbf{L}(z)=0.
\end{equation}

We finish this discussion with an incursion in the difficult case where $\calV$ is not target-reduced.
In that case the mapping $\Lambda$ is not well defined, but we can still partially make up for that by creating, for every
alternator $b \in \Alt(\calV)$, a catcher $\mathbf{L}$ of $\mathbf{M}$ whose
spanning rank is large enough with respect to the one of $b$ (this will be used in Section \ref{section:lastkeyprop}).
First of all, we introduce the defect
$$d:=\dim \Ker \widehat{(-)}.$$
Next, we let $b \in \Alt(\calV)$, and we choose an arbitrary linear section $\kappa : \widehat{\calV} \hookrightarrow U$ of $\widehat{(-)}$
and take the composite mapping $\Psi :=L_b \circ \kappa \circ \Phi \in \Hom(E,U^\star)$,
and the matrix $\mathbf{L}$ whose entries are the composite linear forms $z \in E \mapsto \Psi[z](e_i)$, with $i \in \lcro 1,p\rcro$.
One checks that $\mathbf{L}$ is a catcher of $\mathbf{M}$: indeed, since the entries of  $\mathbf{L} \mathbf{M}$ are homogeneous of degree $2$ in $R$, it will suffice
to prove that $\forall z \in E, \; \mathbf{L}(z) \mathbf{M}(z)=0$.
Letting $z \in E$ and setting $x :=\kappa(\Phi(z))$, so that $\Phi(z)=\widehat{x}$, we see from the definition that, for all $j \in \lcro 1,n\rcro$, the $j$-th entry of
$\mathbf{L}(z) \mathbf{M}(z)$ equals $\sum_{i=1}^p \Psi[z](e_i)\, e_i^\star(v_j(x))$, where $(e_1^\star,\dots,e_p^\star)$
stands for the dual basis of $(e_1,\dots,e_p)$. Hence
$$\bigl(\mathbf{L}(z) \mathbf{M}(z)\bigr)_j=\Psi[z](v_j(x))=b(x,v_j(x))=0.$$
Hence $\mathbf{L}$ is a catcher of $\mathbf{M}$.
We claim that
\begin{equation}\label{eq:minsprk}
\sprk \mathbf{L} \geq \rk b-d.
\end{equation}
To see this, note that the definition of $\mathbf{L}$ shows that a vector $X \in \F^p$
satisfies $\mathbf{L}X=0$ if and only if it represents, in the basis $(e_1,\dots,e_p)$, a vector $x \in U$ such that $b(y,x)=0$ for all
$y \in \im \kappa$. Denoting by $W$ the space of all such vectors of $U$, we see that the kernel of the linear mapping
$$x \in W \mapsto [y \mapsto b(y,x)] \in \bigl(\Ker\widehat{(-)}\bigr)^\star$$
is precisely $\Rrad(b)$, because $\Ker\widehat{(-)} \oplus \im \kappa=U$, and we deduce from the rank theorem that
$$\sprk \mathbf{L}=p-\dim W \geq p-\dim \Rrad(b)-\dim \bigl(\Ker\widehat{(-)}\bigr)^\star = \rk b-d.$$

\section{The structure of irreducible optimal trivial spectrum subspaces: the last key}\label{section:lastkeyprop}

\subsection{Main result}

The present section is devoted to the proof of the following result:

\begin{prop}\label{prop:keyprop}
Let $(V,s)$ be a symplectic space of dimension $2n \geq 4$, with $|\F| > 2n-2$,
and let $S$ be an irreducible optimal trivial spectrum subspace of $\calA_s$.
Then $\Alt(S)$ has dimension $2$ and all its nonzero elements are nondegenerate.
\end{prop}

Combining this result with Theorem \ref{theo:analysisnondegalt2} will complete the proof of Theorem \ref{theo:classoptimalirrTspectrum}:
note that in that theorem, the uniqueness of $\mathbb{L}$ as a function of $S$ follows directly from Theorem \ref{theo:analysisnondegalt2},
whereas the last point has already been discussed in Section \ref{section:WLtrivialspectrum}. As explained in Section \ref{section:wrapup}, obtaining Proposition \ref{prop:keyprop} will thereby complete our proof of all the theorems that were stated in the introductory section.

The proof of Proposition \ref{prop:keyprop} is by induction.
Along the way, we will prove the following lemma, which requires no induction:

\begin{lemma}[Quasitransitivity Lemma]\label{lemma:quasitransitivity}
Let $(V,s)$ be a symplectic space with dimension $2n+2 \geq 6$ and $|\F| > 2n$,
and let $S$ be an optimal trivial spectrum subspace of $\calA_s$.
Then:
\begin{enumerate}[(i)]
\item When $x$ ranges over $V \setminus \{0\}$, the greatest integer $\dim \calS x$ is $2n$.
\item The set of all $x \in V \setminus \{0\}$ such that $\dim \calS x=2n$ is a spanning subset of $V$.
\end{enumerate}
\end{lemma}

In Section \ref{section:dim4case}, we will start the proof of Proposition \ref{prop:keyprop} by
settling the case $n=2$ thanks to some results from Section \ref{section:WL}.
Then we will set things up for the inductive step by coming back to the proof technique that was used in \cite{dSPaffinealt} to obtain Theorem \ref{theo:majdim}.
The Quasitransitivity Lemma will be obtained easily, and then we will
restrict the study to the irreducible case to prove the inductive step in Proposition \ref{prop:keyprop}.
Generic matrix computations will be performed to prove that $\Alt(S)$ has dimension $2$, and then to obtain a clearer understanding of its
elements. And finally the ``bad" case where $\Alt(S)$ contains a nonzero degenerate form will be discarded in the end of a long and meticulous analysis
(Sections \ref{section:badcase1} to \ref{section:badcase4}).

\subsection{The $4$-dimensional case}\label{section:dim4case}

Let us consider a symplectic space $(V,s)$ of dimension $4$, and an irreducible optimal
trivial spectrum linear subspace $\calV$ of $\calA_s$. It turns out that it is more convenient to use the forms viewpoint, so we
take
$$S:=\{(x,y) \mapsto s(x,v(y)) \mid v \in \calV\}\subseteq \mathcal{A}^2(V),$$
which has dimension $2$, and consider the plane $\calS:=s+S$ of symplectic forms, which is an irreducible
optimal affine subspace of $\calA^2(V)$.
To start with, we prove that all the nonzero forms in $S$ are symplectic.
Assuming otherwise, there would be a rank $2$ form $s_1$ in $S$, and we can take a symplectic basis $(e_1,e_2,f_1,f_2)$ of $V$
in which the Gram matrix of $s_1$ equals
$$M(s_1)=\begin{bmatrix}
0_2 & 0_2 \\
0_2 & K_2
\end{bmatrix}.$$
Consider an arbitrary element $s_0 \in \calS$, and denote by $M(s_0)$ the Gram matrix of $s_0$ in that basis.
We also denote by $\Pf$ the Pfaffian of alternating matrices.
The affine mapping $t \mapsto \Pf(M(s_0)+tM(s_1))$ only takes nonzero values,
and hence its linear part is zero: therefore, the upper-left block of $M(s_0)$ is zero.
Varying $s_0$, we deduce that $\Vect(e_1,e_2)$ is $\calS$-adapted, which contradicts the irreducibility of $\calS$.
Hence all the nonzero forms in $\calS$ are symplectic.

Next, we take an arbitrary basis $(s_1,s_2)$ of $S$, and we consider the automorphism $u$ of $V$ such that $s_2(-,x)=s_1(-,u(x))$ for all $x \in V$.
Because every nontrivial linear combination of $s_1$ and $s_2$ is nondegenerate, $u$ has no eigenvalue in $\F$.
Note that $u$ is $s_1$-alternating; hence it is known that its invariant factors (as a vector space endomorphism) come in pairs
\cite{Scharlaupairs}.
As here $\dim V=4$ the only option is that these invariant vectors are $p,p$ for some irreducible monic polynomial $p$ of degree $2$ in $\F[t]$.
Therefore $\L:=\F[u]$ is a quadratic extension of $\F$, and $u^2=\lambda u+\mu$ for scalars $\lambda,\mu$ in $\F$.
Note that, for all $x \in V$,
$$s_1(x,u(x))=s_2(x,x)=0$$
and
$$s_2(x,u(x))=s_1(x,u^2(x))=\mu s_1(x,x)+\lambda s_2(x,x)=0.$$
Hence $S \subseteq W_\L$. As the dimensions are equal we deduce that $S=W_\L$, and we conclude that $\calS=s+W_\L$.

Since $\calS$ contains only symplectic forms, it follows from Proposition \ref{prop:dirWL}
that the form $b : (x,y) \mapsto s(x,u(y))$ is nonisotropic, and we deduce that
$P:=\Vect(s,b)$ has dimension $2$ and contains only nondegenerate forms. Then, as seen at the start of Section \ref{section:WLtrivialspectrum} (and with the same notation),
we derive that $\calV=\calU_{s,\L}$ and $P=\Alt(\calV)$. This completes the proof of the special case $n=2$.

\subsection{Setting the inductive step up}\label{section:startinduction}

Now, we go right back to the line of reasoning from the proof of Theorem \ref{theo:majdim} that we gave in \cite{dSPaffinealt}.
Let $n \geq 2$ and $(V,s)$ be a symplectic space of dimension $2n+2$.
Assume that $|\F| > 2n$ and let $S$ be an optimal trivial spectrum subspace of $\calA_s$.
At this point, we do not assume that $S$ is irreducible.
For every $x \in V$, we introduce the evaluation mapping
$$\widehat{x} : u \in S \longmapsto u(x) \in V$$
and the corresponding dual operator space
$$\widehat{S}=\{\widehat{x} \mid x \in V\} \subseteq \Hom(S,V).$$

Denote by $r$ the greatest possible rank for the elements of $\widehat{S}$. Choose
$x \in V$ such that $\rk \widehat{x}=r$, that is $\dim S x=r$, and set
$$S':=\Ker \widehat{x}=\{u \in S : u(x)=0\}.$$
Applying the Tangent Space Lemma (Lemma \ref{lemma:tangent}) to $\widehat{x}$ yields
$$\forall u \in S', \; \forall y \in V, \; u(y) \in \im \widehat{x},$$
i.e., every $u\in S'$ has its range included in $S x$.

Now, let us analyse $S x$. We have $\F x \oplus S x \subseteq \{x\}^{\bot_s}$.
Indeed $x \in  \{x\}^{\bot_s}$ since $s$ is alternating; the inclusion
$S \subseteq \calA_s$ yields $S x \subseteq \{x\}^{\bot_s}$; and finally $x \not\in S x$ because $S$ has trivial spectrum.
Hence $\dim S x \leq 2n$, i.e., $r \leq 2n$.

Next, we embed $S x$ inside a direct summand $V'$ of $\F x$ in $\{x\}^{\bot_s}$.
The space $V'$ is then $s$-regular, and hence $s$ induces a symplectic form $\overline{s}$ on $V'$.
Every $u \in S'$ maps into $V'$ and hence vanishes on $(V')^{\bot_s}$ (as $u$ is $s$-alternating),
which is a direct summand of $V'$ in $V$ because $V'$ is $s$-regular.
Hence we obtain an injective linear mapping
$$\Phi : u \in S' \mapsto u_{V'} \in \End(V'),$$
whose range is actually included in $\calA_{\overline{s}}$ and has trivial spectrum.
Noting that $|\F| > 2n-2$, we deduce from Theorem \ref{theo:majdim} that
$$\dim S'=\dim \Phi(S') \leq n(n-1).$$
Combining this with the rank theorem applied to $\widehat{x}$, we retrieve
$$\dim S= \dim S x+\dim S' \leq 2n+n(n-1)=(n+1)n.$$

Here we have assumed that $S$ is optimal, so all the inequalities turn out to be equalities.
Hence:
\begin{itemize}
\item $\dim S x=2n$, to the effect that $r=2n$ and $V'=S x$;
\item $\dim \Phi(S')=n(n-1)$, and hence $\Phi(S')$ is an optimal trivial spectrum subspace of $\calA_{\overline{s}}$.
\end{itemize}
At this point, we have already proved the first statement in the Quasitransitivity Lemma, and before we move forward
we immediately derive the second one.

\begin{proof}[Proof of the second statement in Lemma \ref{lemma:quasitransitivity}]
Let a linear form $\varphi$ on $V$ vanish at every $z \in V$ such that $\rk \widehat{z}=2n$.
We consider the operator space
$$\{\widehat{z} \oplus \varphi(z).\id_\F \mid z \in V\} \subseteq \Hom(S \oplus \F, V \oplus \F),$$
and a matrix space $\calM$ that represents it in respective bases of
$S \oplus \F$ and $V \oplus \F$ that are adapted to the decompositions into direct summands.

Let us take a generic matrix $\mathbf{N}$ of $\calM$. Hence this matrix takes the form
$$\mathbf{N}=\begin{bmatrix}
\mathbf{N_1} & 0 \\
0 & \mathbf{a}
\end{bmatrix}$$
where $\mathbf{N_1}$ is a generic matrix for a matrix space that represents $\widehat{S}$, and
$\mathbf{a}$ is nonzero if $\varphi \neq 0$.
By specializing the indeterminates, we only find matrices with rank at most $2n$ (because $\rk \widehat{z} \leq 2n$ in any case, with sharp inequality
if $\varphi(z)\neq 0$). By Corollary \ref{cor:genericrank}, we deduce that $\rk \mathbf{N} \leq 2n$ and
$\rk \mathbf{N_1}=2n$ (because $\rk \widehat{x}=2n$).
Hence $\mathbf{a}=0$ and we deduce that $\varphi=0$. By duality we deduce that $V$ is the linear span
of the vectors $z \in V$ such that $\rk \widehat{z}=2n$.
\end{proof}

Now that Lemma \ref{lemma:quasitransitivity} has been proved, we move forward with the analysis of $S$, and from now
on we assume that $S$ is \emph{irreducible}.
In the remainder of the proof, we will use the convenient notation
$$S^{(x)}:=\Phi(S')=\bigl\{u_{S x} \mid u \in S \; \text{such that $u(x)=0$}\bigr\}.$$
We have just shown that $S^{(x)}$ is an optimal trivial spectrum subspace of $\calA_{\overline{s}}$.
We wish to apply an induction hypothesis to $S^{(x)}$, but we must beware that there is no obvious reason
why $S^{(x)}$ should be irreducible. In fact, it is only after a very long analysis that we will conclude that $S^{(x)}$ is irreducible.
However, we can still use the induction hypothesis to retrieve some information on $S^{(x)}$.
To this end, remember that $\overline{s}$ denotes the symplectic form on $Sx$ induced by $s$.

\begin{claim}\label{claim:induction}
Denote by $W$ the greatest proper $S^{(x)}$-invariant subspace of $S x$.
\begin{enumerate}[(i)]
\item If $S^{(x)}$ is irreducible then $\Alt(S^{(x)})$ has dimension $2$ and its nonzero elements are all non-degenerate.
\item Otherwise $W \neq \{0\}$, and there is a straight alternator $b \in \Alt(S^{(x)})$ with radical $W$
and such that all the elements of $\Vect(\overline{s},b) \setminus \F b$ are nondegenerate.
\end{enumerate}
\end{claim}

\begin{proof}
The first point is readily deduced from the induction hypothesis.

Assume for the remainder of the proof that $S^{(x)}$ is reducible, to the effect that $W \neq \{0\}$.
Since $S^{(x)}$ consists of $\overline{s}$-alternating endomorphisms, the set of all $S^{(x)}$-invariant subspaces
is invariant under taking the $\overline{s}$-orthogonal complement, and we deduce that $W^{\bot_{\overline{s}}}$ is the smallest non-zero
$S^{(x)}$-invariant subspace of $S x$. By combining Theorems \ref{theo:invariantsubspaces1} and \ref{theo:invariantsubspacesbis},
we see that $W^{\bot_{\overline{s}}}$ is totally $\overline{s}$-singular.

Hence, by Proposition \ref{prop:decompbase} combined with Remark \ref{remark:inducedonquotient}, the elements of $S^{(x)}$ induce endomorphisms of $S x/W$ that
constitute an optimal trivial spectrum subspace $\calU$ of $\End(S x/W)$, and we can apply
Theorem \ref{theo:matrixcase} to understand the structure of $\calU$. Here, the situation is further simplified
by noting that $\calU$ is irreducible, owing to the maximality of $W$.
This yields a nondegenerate alternator $\overline{b}$ of $\calU$. Then it is easily checked that
$$b : (y,z) \in (S x)^2 \mapsto \overline{b}(\overline{y},\overline{z})$$
is an alternator of $S^{(x)}$, and obviously $\Lrad(b)=\Rrad(b)=W$.

In order to conclude, it suffices to check that $\overline{s}+\lambda b$ is nondegenerate for all $\lambda \in \F$.
To see this, let $\lambda \in \F$ and $y \in \Lrad(\overline{s}+\lambda b)$.
For all $z \in W^{\bot_{\overline{s}}}$ we have $\overline{s}(y,z)=\overline{s}(y,z)+\lambda b(y,z)=0$ because $W^{\bot_{\overline{s}}} \subseteq W=\Rrad(b)$.
Hence $y \in (W^{\bot_{\overline{s}}})^{\bot_{\overline{s}}}=W$, and then $y \in \Lrad(b)$. Therefore $\overline{s}(y,-)=\overline{s}(y,-)+\lambda b(y,-)=0$, and we conclude that $y=0$
because $\overline{s}$ is nondegenerate.
\end{proof}

\subsection{Setting the matrix vision up}

Now, we can finally start to unleash the power of the generic matrix methods.
We choose respective bases $(u_1,\dots,u_N)$ (with $N=(n+1)n$) and $(e_1,\dots,e_{2n+2})$ of
$S$ and $V$ in which the matrix of $\widehat{x}$ equals
$$J_{2n}=\begin{bmatrix}
I_{2n} & [0]_{2n \times (N-2n)} \\
[0]_{2 \times 2n} & [0]_{2 \times (N-2n)}
\end{bmatrix}.$$
Note that
$$S'=\Ker \widehat{x}=\Vect(u_{2n+1},\dots,u_N) \quad \text{and} \quad S x=\Vect(e_1,\dots,e_{2n}).$$
We denote by $\calM$ the matrix space that represents the elements of $\widehat{S}$ in those bases.

Next, we use the canonical construction for a generic matrix $\mathbf{M}$ of $\calM$ with parameter space $E=V$.
First of all, we write $R=\Sym(E^\star)$ and denote by $\L$ its fraction field.
For all $(i,j)\in \lcro 1,2n+2\rcro \times \lcro 1,N\rcro$ we take the linear form $\mathbf{m}_{i,j}$ on $V$
that takes each $y \in V$ to the $(i,j)$-entry of the matrix of $u \in S \mapsto u(y) \in V$ in the bases
$(u_1,\dots,u_N)$ and $(e_1,\dots,e_{2n+2})$, and finally we put $\mathbf{M}:=(\mathbf{m}_{i,j})_{i,j} \in \Mat_{2n+2,2N}(R)$.
As Lemma \ref{lemma:quasitransitivity} yields that the greatest possible rank in $\widehat{S}$ is $2n$,
we deduce from Corollary \ref{cor:genericrank} that
$$\rk \mathbf{M}=2n.$$

Next, remembering that all the elements of $\Ker \widehat{x}$ have their range included in $\im \widehat{x}$, we obtain that
$\mathbf{M}$ can be written in block form along the same format as $J_{2n}$:
$$\mathbf{M}=\begin{bmatrix}
\mathbf{A} & \mathbf{C} \\
\mathbf{B} & [0]_{2 \times (N-2n)}
\end{bmatrix} \quad \text{where $\mathbf{A} \in \Mat_{2n}(R_1)$, $\mathbf{B} \in \Mat_{2,2n}(R_1)$, $\mathbf{C} \in \Mat_{2n,N-2n}(R_1)$.}$$
It is then clear that $\mathbf{C}$ is a generic matrix of $\widehat{S^{(x)}}$ (but with parameter space $E=V$ rather than $S x$).

We recall now that $S^{(x)}$ is an optimal trivial spectrum subspace of $\calA_{\overline{s}}$,
so Lemma \ref{lemma:quasitransitivity} applies to it and yields
$$\rk \mathbf{C}=2n-2.$$
We also recall that $\mathbf{M}$ belongs to $\Vect_\L(\calM)$ because it is a generic matrix of $\calM$,
and hence $\alpha J_{2n}+\beta \mathbf{M} \in \Vect_\L(\calM)$ for all $(\alpha,\beta)\in \L^2$.
Since $\maxrk(\Vect_\L(\calM))=\maxrk(\calM)=2n$ by Proposition \ref{prop:maxrank}, we deduce that
$\forall (\alpha,\beta)\in \L^2, \; \rk(\alpha J_{2n}+\beta \mathbf{M}) \leq 2n$.
Hence the Flanders-Atkinson Lemma (Lemma \ref{lemma:FLA}) yields
\begin{equation}\label{eq:FLA}
\forall k \geq 0, \; \mathbf{B}\mathbf{A}^k\mathbf{C}=0.
\end{equation}
In particular $\mathbf{B}\mathbf{C}=0$, which yields that every row of $\mathbf{B}$ is a catcher of $\mathbf{C}$.
Now we write
$$\mathbf{B}=\begin{bmatrix}
\mathbf{B_1} \\
\mathbf{B_2}
\end{bmatrix} \quad \text{with $\mathbf{B_1} \in \Mat_{1,2n}(R_1)$ and $\mathbf{B_2} \in \Mat_{1,2n}(R_1)$,}$$
so that every $\F$-linear combination of $\mathbf{B_1}$ and $\mathbf{B_2}$ is a catcher of $\mathbf{C}$.

\begin{claim}\label{claim:2rowsindep}
The rows $\mathbf{B_1}$ and $\mathbf{B_2}$ are linearly independent over $\F$.
\end{claim}

\begin{proof}
Assume the contrary. Then we find a pair $(\alpha,\beta)\in \F^2 \setminus \{(0,0)\}$ such that $\alpha \mathbf{B_1}+\beta \mathbf{B_2}=0$.
Then all the operators in $S$ would map into $\Vect(e_1,\dots,e_{2n},\alpha e_{2n+1}+\beta e_{2n+2})$,
thereby contradicting the assumed irreducibility of $S$.
\end{proof}

It is useful now to explain where we are headed. Our goal is to extend
$\mathbf{B_1}$ and $\mathbf{B_2}$, which are catchers of $\mathbf{C}$, to catchers
$\widetilde{\mathbf{B_1}}$ and $\widetilde{\mathbf{B_2}}$ of $\mathbf{M}$. In other words we are searching for four
elements of $R_1$, named $\mathbf{a},\mathbf{b},\mathbf{c},\mathbf{d}$, such that
$\begin{bmatrix}
\mathbf{B_1} & -\mathbf{a} & -\mathbf{b}
\end{bmatrix}$
and
$\begin{bmatrix}
\mathbf{B_2} & -\mathbf{c} & -\mathbf{d}
\end{bmatrix}$
are catchers of $\mathbf{M}$, which simply amounts to having
$$\mathbf{B_1} \mathbf{A}=\mathbf{a} \mathbf{B_1}+\mathbf{b} \mathbf{B_2} \quad \text{and} \quad
\mathbf{B_2} \mathbf{A}=\mathbf{c} \mathbf{B_1}+\mathbf{d} \mathbf{B_2}.$$
To obtain this (if possible), we will use the case $k=1$ in \eqref{eq:FLA}
to find that $\mathbf{B_1} \mathbf{A}$ and $\mathbf{B_2} \mathbf{A}$ are left-annihilators of $\mathbf{C}$.
Since $\rk \mathbf{C}=2n$, and in case $\mathbf{B_1}$ and $\mathbf{B_2}$ are linearly independent over $\L$,
the matrices $\mathbf{B_1} \mathbf{A}$ and $\mathbf{B_2} \mathbf{A}$ must be linear combinations of
$\mathbf{B_1}$ and $\mathbf{B_2}$ over $\L$, and the difficulty is to prove that they are actually
linear combinations of them with coefficients that are $1$-homogeneous in $R$. This is of course connected to
the Second Factorization Lemma (Lemma \ref{lemma:Factor2}).

A final remark before we proceed: We have seen that $S x$ is $s$-regular, and $x$ belong to its $s$-orthogonal complement
$$P:=(S x)^{\bot_s}.$$

\begin{claim}\label{claim:actiononP}
All the elements of $S$ map $P$ into $S x$, and all the elements of $S'=\Ker \widehat{x}$
vanish on $P$.
\end{claim}

\begin{proof}
We start with the first point. Let $u \in S$.
Let $z \in P \setminus \F x$. Then as $u(x) \in S x$ we have $u(P)=u((S x)^{\bot_s}) \subseteq \{x\}^{\bot_s}$ because
$u$ is $s$-selfadjoint, and in particular $u(z) \in \{x\}^{\bot_s}$.
Next, $u(z) \in \{z\}^{\bot_s}$ because $u$ is $s$-alternating, and hence $u(z) \in \{x,z\}^{\bot_s}=S x$ by double-orthogonality.
Varying $z$ then yields the first result.

For the second one, remember that we have proved that all the elements of $S'$ have their range included in $S x$,
and since they are $s$-selfadjoint we deduce that their kernel includes $(Sx)^{\bot_s}$.
\end{proof}

The previous result translates as follows on the generic matrix $\mathbf{M}$:
\begin{equation}\label{eq:P}
\forall z \in P, \quad \mathbf{B}(z)=0 \quad \text{and} \quad \mathbf{C}(z)=0.
\end{equation}

In order to move forward we need to discard the problematic case where $S^{(x)}$ is not target-reduced:
this is done in the next section.

\subsection{Discarding the first degenerate case}

The aim of this section is to prove the following result:

\begin{claim}\label{claim:targetreduced}
The space $S^{(x)}$ is target-reduced.
\end{claim}

To do so, we explore the consequences of the opposite statement. So, throughout this section we assume that there exists a
linear hyperplane $H$ of $S x$ that includes the range of every operator in $S^{(x)}$, and we seek a contradiction.

To start with, we note that $H$ is $S^{(x)}$-invariant, and we use the classification of the invariant
subspaces for optimal trivial spectrum subspaces of $\calA_{\overline{s}}$
(see Section \ref{section:invariantsubspaces})
to deduce that $H$ is the sole linear hyperplane of $Sx$ that is $S^{(x)}$-invariant (this uses $2n \geq 4$).
In particular $H$ is the sum of the ranges of the elements of $S^{(x)}$, and by orthogonality this yields that
$H^{\bot_{\overline{s}}}$ is the intersection of the kernels of the elements of $S^{(x)}$.

To simplify things, we alter the basis we have taken so that $H=\Vect(e_1,\dots,e_{2n-1})$, without fundamentally changing
the starting assumptions.
Then we can rewrite
$$\mathbf{C}=\begin{bmatrix}
\mathbf{C_1} \\
[0]_{1 \times (N-2n)}
\end{bmatrix}$$
where the matrix $\mathbf{C_1} \in \Mat_{2n-1,N-2n}(R_1)$ satisfies
$$\rk \mathbf{C_1}=\rk \mathbf{C}=2n-2.$$
Since $\mathbf{C_1}$ has $2n-1$ rows, the $\L$-vector space of all its left-annihilators has dimension $1$.

Next, remember that the induced symplectic form $\overline{s}=s_{|(S x)^2}$
is an alternator of $S^{(x)}$. By the procedure described in the end of Section \ref{section:generic2},
we use $\overline{s}$ to recover a catcher $\mathbf{X}$ of $\mathbf{C}$ with spanning rank at least $2n-d$,
where $d$ stands for the dimension of the kernel of the linear mapping
$$y \in Sx \mapsto [u \in S^{(x)} \mapsto u(y)] \in \widehat{S^{(x)}}.$$
We have seen earlier that this kernel is precisely $H^{\bot_{\overline{s}}}$, so $d=1$.
Hence $\sprk \mathbf{X} \geq 2n-1$.

Then, by removing the last entry of $\mathbf{X}$ we obtain a catcher $\mathbf{X}'$ of
$\mathbf{C_1}$ with $\sprk \mathbf{X'} \geq 2n-2 \geq 2$.

Next, let us rewrite the matrix $\mathbf{B}$ as follows:
$$\mathbf{B}=\begin{bmatrix}
\mathbf{B'_1} & ? \\
\mathbf{B'_2} & ?
\end{bmatrix} \quad \text{with $\mathbf{B'_1}$ and $\mathbf{B'_2}$ in $\Mat_{1,2n-1}(R_1)$.}$$
The rows $\mathbf{B'_1}$ and $\mathbf{B'_2}$ are catchers of $\mathbf{C_1}$, and hence
they are scalar multiples of $\mathbf{X}'$ over $\L$.
Remembering that $\sprk \mathbf{X}' \geq 2$,
we use the First Factorization Lemma to obtain scalars $\alpha$ and $\beta$ in $\F$ such that
$\mathbf{B'_1}=\alpha \mathbf{X}'$ and $\mathbf{B'_2}=\beta \mathbf{X}'$.
Hence, some nontrivial linear combination of $\mathbf{B'_1}$ and $\mathbf{B'_2}$ over $\F$ is zero.
Changing the basis vectors $e_{2n+1}$ and $e_{2n+2}$ appropriately, we lose no generality in assuming that
$\mathbf{B'_1}=0$, which we do from now on.
As a consequence of Claim \ref{claim:2rowsindep}, we have $\mathbf{B_1}\neq 0$
and hence
$$\mathbf{B_1}=\begin{bmatrix}
[0]_{1 \times (2n-1)} & \mathbf{a}
\end{bmatrix} \quad \text{for some $\mathbf{a} \in R_1 \setminus \{0\}$.}$$
Now, we use identity \eqref{eq:FLA} with $k=1$:
it yields that $\mathbf{B_1}\mathbf{A}$ is a left-annihilator of $\mathbf{C}$.
Writing the last row of $\mathbf{A}$ as $\begin{bmatrix}
\mathbf{L} & ?
\end{bmatrix}$ with $\mathbf{L} \in \Mat_{1,2n-1}(R_1)$,
we obtain that $\mathbf{a}\mathbf{L}$ is a left-annihilator of $\mathbf{C}_1$, and since
$\mathbf{a} \neq 0$ it follows that $\mathbf{L}$ is a catcher of $\mathbf{C}_1$.

Applying the First Factorization Lemma once more, we obtain a scalar $\gamma \in \F$ such that $\mathbf{L}=\gamma \mathbf{X}'$.
Besides, we remember that $\mathbf{B'_2}=\beta \mathbf{X}'$ for some $\beta \in \F$.

Let us now sum up our findings in geometrical terms. Set
$$G:=\begin{cases}
H\oplus \F (\gamma e_{2n}+\beta e_{2n+2}) & \text{if $(\gamma,\beta) \neq (0,0)$,} \\
H \oplus \F e_{2n} & \text{otherwise.}
\end{cases}$$
The above generic matrix identities between $\mathbf{L},\mathbf{B'_2},\mathbf{X}'$ show that
all the elements of the subspace
$$\calT:=\Vect(u_1,\dots,u_{2n-1},u_{2n+1},\dots,u_{N})$$
have their range included in $G$. Since they are all $s$-selfadjoint,
we deduce that they all vanish on the $2$-dimensional space $G^{\bot_s}$.
It is now necessary to discuss whether $G^{\bot_s}$ is $s$-regular or not.

\vskip 3mm
\noindent \textbf{Case 1: $G^{\bot_s}$ is $s$-regular.} \\
Then $G$ is also $s$-regular. By restricting to $G$, we obtain a vector space isomorphism from
$\calT$ to a trivial spectrum subspace of $\calA_{s'}$, where $s'$ denotes the symplectic form induced by $s$ on $G$.
Then $\dim \calT \leq n(n-1)$ by Theorem \ref{theo:majdim}, and hence $n(n+1)-1 \leq n(n-1)$, which is absurd.

\vskip 3mm
\noindent \textbf{Case 2: $G^{\bot_s}$ is totally $s$-singular.} \\
Hence $G^{\bot_s} \subseteq G$.
Denote by $\overline{\overline{s}}$ the symplectic form induced by $s$ on the quotient space $G/G^{\bot_s}$,
and note that every element $u \in \calT$ induces an $\overline{\overline{s}}$-alternating endomorphism $\overline{\overline{u}}$ of $G/G^{\bot_s}$
with trivial spectrum. We denote by $\overline{\overline{\calT}}$ the space of those endomorphisms, and we note from
Theorem \ref{theo:majdim} that
$$\dim \overline{\overline{\calT}} \leq (n-1)(n-2).$$
Denoting by $K$ the kernel of the mapping $u \in \calT \mapsto \overline{\overline{u}}$, we
see from the rank theorem that $\dim \calT \leq (n-1)(n-2)+\dim K$.
Yet, in the notation of Lemma \ref{lemma:NW} we have $K \subseteq \calN_{G^{\bot_s}}$.
Hence $\dim K \leq 2(2n-2)+1$, with equality holding only if $K=\calN_{G^{\bot_s}}$.
Hence
$$\dim \calT \leq (n-1)(n-2)+1+2(2n-2)=(n+1)n-1.$$
In turn, this shows that all the previous inequalities were equalities, and in particular
$K=\calN_{G^{\bot_s}}$, to the effect that $\calN_{G^{\bot_s}} \subseteq S$.
The Invariant Subspace Lemma (Lemma \ref{lemma:invariantsubspacelemma}) then shows that $G^{\bot_s}$ is $S$-invariant, thereby contradicting the assumed irreducibility of $S$. Therefore Claim \ref{claim:targetreduced} is now proved.

\subsection{Finishing the case where $S^{(x)}$ is irreducible}\label{section:goodcase}

Here, we assume that $S^{(x)}$ is irreducible, and we will see how to conclude in that case.

By point (a) of Claim \ref{claim:induction}, the alternator space
$\Alt(S^{(x)})$ has dimension $2$ and all its nonzero elements are nondegenerate.
Since $S^{(x)}$ is target-reduced, we can use the bijective linear correspondence between
the catchers of $\mathbf{C}$ and the alternators of $S^{(x)}$, recovering that
$\catch(\mathbf{C})$ has dimension $2$ (over $\F$) and all its nonzero elements have spanning rank $2n$.

As $\catch(S^{(x)})$ contains $\mathbf{B_1}$ and $\mathbf{B_2}$, which are not $\F$-collinear (Claim \ref{claim:2rowsindep}),
we see that these rows have spanning rank $2n \geq 2$. Hence the First Factorization Lemma shows that
$\mathbf{B_1}$ and $\mathbf{B_2}$ are linearly independent over $\L$. Since $\rk(\mathbf{C})=2n-2$,
we conclude that the annihilator space of $\mathbf{C}$ is $\L$-spanned by $\mathbf{B_1}$ and $\mathbf{B_2}$.
Note finally that $\catch(S^{(x)})$ is spanned by $\mathbf{B_1}$ and $\mathbf{B_2}$, and hence
every nontrivial linear combination of $\mathbf{B_1}$ and $\mathbf{B_2}$ over $\F$ has spanning rank $2n$.

Next, we apply the case $k=1$ in \eqref{eq:FLA} to obtain that $\mathbf{B_1}\mathbf{A}$ is a left-annihilator of $\mathbf{C}$, and hence
$\mathbf{B_1}\mathbf{A}$ is a linear combination of $\mathbf{B_1}$ and $\mathbf{B_2}$ over $\L$.
Then we apply the Second Factorization Lemma to $\mathbf{X}=\mathbf{B_1}$ and $Y=\mathbf{B_2}$:
here both conditions in this lemma are satisfied (in condition (ii), we have the second option), so we obtain
elements $\mathbf{a},\mathbf{b}$ in $R_1$ such that
$$\mathbf{B_1}\mathbf{A}=\mathbf{a}\mathbf{B_1}+\mathbf{b} \mathbf{B_2}.$$
It follows that
$$\widetilde{\mathbf{B_1}}:=\begin{bmatrix}
\mathbf{B_1} & -\mathbf{a} & -\mathbf{b}
\end{bmatrix} \in \Mat_{1,2n+2}(R_1)$$
is a catcher of $\mathbf{M}$.
In the remainder of the proof we will forget how $\mathbf{a}$ and $\mathbf{b}$ have been obtained, and we will simply consider
an \emph{arbitrary} extension $\widetilde{\mathbf{B_1}}:=\begin{bmatrix}
\mathbf{B_1} & -\mathbf{a} & -\mathbf{b}
\end{bmatrix}$ of $\mathbf{B_1}$ into a catcher of $\mathbf{M}$.

\begin{claim}\label{claim:spanrank2n}
The spanning rank of $\widetilde{\mathbf{B_1}}$ equals $2n+2$.
\end{claim}

\begin{proof}
Let us consider the quadratic identity $\mathbf{B_1} \mathbf{A}=\mathbf{a} \mathbf{B_1}+\mathbf{b} \mathbf{B_2}$.
By specializing it at $x$ and then polarizing the result, we obtain
$$\mathbf{B_1}\, \mathbf{A}(x)+\mathbf{B_1}(x)\, \mathbf{A}=\mathbf{a}\, \mathbf{B_1}(x)+\mathbf{b}\, \mathbf{B_2}(x)+\mathbf{a}(x)\, \mathbf{B_1}
+\mathbf{b}(x)\, \mathbf{B_2}.$$
Yet $\mathbf{A}(x)=I_{2n}$ and $\mathbf{B_1}(x)=0=\mathbf{B_2}(x)$. Hence
$(1-\mathbf{a}(x))\mathbf{B_1}-\mathbf{b}(x) \mathbf{B_2}=0$. Since $\mathbf{B_1}$ and $\mathbf{B_2}$ are linearly independent over $\F$, this yields
$\mathbf{a}(x)=1$ and $\mathbf{b}(x)=0$.
Hence
$$\widetilde{\mathbf{B_1}}(x)=\begin{bmatrix}
[0]_{1 \times (2n)}  & -1 & 0
\end{bmatrix}.$$
Now, let us take an arbitrary $z \in P \setminus \F x$, and remember that $P=(S x)^{\bot_s}$ has dimension $2$.
Remember from identity \eqref{eq:P} on page \pageref{eq:P} that $\mathbf{B}(z)=0$. Hence the same polarization technique as before yields
$$\mathbf{B_1} \mathbf{A}(z)=\mathbf{a}(z) \mathbf{B_1}+\mathbf{b}(z) \mathbf{B_2},$$
i.e.,
$$\mathbf{B_1}(\mathbf{A}(z)-\mathbf{a}(z) I_n)=\mathbf{b}(z) \mathbf{B_2}.$$
Now, assume that $\mathbf{b}(z)=0$. Then as the spanning rank of $\mathbf{B_1}$ is $2n$ we obtain $\mathbf{A}(z)-\mathbf{a}(z) I_n=0$ by specializing.
Hence $\mathbf{M}(z)=\mathbf{a}(z) \mathbf{M}(x)$, which we can rewrite $\mathbf{M}(z-\lambda x)=0$ for $\lambda:=\mathbf{a}(z)$.
Due to the specific construction of the generic matrix $\mathbf{M}$, this yields $\widehat{z-\lambda x}=0$.
But since $z-\lambda x \neq 0$, this would yield that $\F (z-\lambda x)$ is a nontrivial subspace of $V$ that is invariant under $S$,
thereby contradicting the assumed irreducibility of $S$.

We deduce that $\mathbf{b}(z) \neq 0$.
From there, we have
$$\widetilde{\mathbf{B_1}}(z)=\begin{bmatrix}
[0]_{1 \times (2n)}  & ? & -\mathbf{b}(z)
\end{bmatrix} \quad \text{and} \quad \widetilde{\mathbf{B_1}}(x)=\begin{bmatrix}
[0]_{1 \times (2n)}  & -1 & 0
\end{bmatrix}.$$
Combining this with the fact that the spanning rank of $\mathbf{B_1}$ is $2n$, we conclude that
$\sprk(\widetilde{\mathbf{B_1}})=2n+2$.
\end{proof}

With the same line of reasoning we also obtain an extension $\widetilde{\mathbf{B_2}}:=\begin{bmatrix}
\mathbf{B_2} & -\mathbf{c} & -\mathbf{d}
\end{bmatrix}$
into a catcher of $\mathbf{M}$, and every such extension has spanning rank $2n+2$.

Now, let $(\alpha,\beta)\in (\F \setminus \{0\}) \times \F$.
Then $\alpha \widetilde{\mathbf{B_1}}+\beta \widetilde{\mathbf{B_2}}=\begin{bmatrix}
\alpha \mathbf{B_1}+\beta \mathbf{B_2} & ? & ?
\end{bmatrix}$
is a catcher of $\mathbf{M}$ that extends the catcher $\alpha \mathbf{B_1}+\beta \mathbf{B_2}$ of $\mathbf{C}$.
By generalizing Claim \ref{claim:spanrank2n}
(we simply replace the basis $(e_1,\dots,e_{2n+2})$ with $(e_1,\dots,e_{2n},\alpha e_{2n+1}+\beta e_{2n+2},e_{2n+2})$),
we obtain that $\alpha \widetilde{\mathbf{B_1}}+\beta \widetilde{\mathbf{B_2}}$
has spanning rank $2n+2$.

Hence, every nontrivial $\F$-linear combination of $\widetilde{\mathbf{B_1}}$
and $\widetilde{\mathbf{B_2}}$ has spanning rank $2n+2$,
and now the conclusion will be easy. We have just
constructed a $2$-dimensional subspace $\Vect_\F(\widetilde{\mathbf{B_1}},\widetilde{\mathbf{B_2}})$
of $\catch(\mathbf{M})$ in which all the nonzero elements have spanning rank $2n+2$.
Remembering that $\rk \mathbf{M}=2n$, we obtain that every catcher of $\mathbf{M}$
is a linear combination of $\widetilde{\mathbf{B_1}}$
and $\widetilde{\mathbf{B_2}}$ over $\L$. Applying the Second Factorization Lemma once more,
we deduce that $\catch(\mathbf{M})=\Vect_\F(\widetilde{\mathbf{B_1}},\widetilde{\mathbf{B_2}})$,
and in particular $\catch(\mathbf{M})$ has dimension $2$ and all its nonzero elements have spanning rank $2n+2$.
Finally, since $S$ is irreducible it is target-reduced, and hence we recover from the correspondence laid out in Section \ref{section:generic2}
that $\Alt(S)$ has dimension $2$ and that all its nonzero elements are nondegenerate.
Hence the inductive step is climbed in the case where $S^{(x)}$ is irreducible.

\subsection{The case where $S^{(x)}$ is reducible (1)}\label{section:badcase1}

The proof is far from finished, as we still need to consider the situation where $S^{(x)}$ is reducible.
So, from that point on we assume that $S^{(x)}$ is reducible, and we seek a contradiction (which will only come after a very long analysis of the situation).

Remember first that $W$ denotes the greatest proper $S^{(x)}$-invariant subspace of $Sx$.
We denote by $m$ its dimension and note that $m \geq n$ because $W^{\bot_{\overline{s}}}$ is totally $s$-singular
(see the explanation in the proof of Claim \ref{claim:targetreduced}). Moreover $m \leq 2n-2$ because $S^{(x)}$ is target-reduced.

Next, the second point of Claim \ref{claim:induction} yields that $\Alt(S^{(x)})$ contains an element $b$ such that
$\Lrad(b)=\Rrad(b)=W$ and every element of $\Vect(\overline{s},b) \setminus \F b$ is nondegenerate.
Remembering that $S^{(x)}$ is target-reduced, we can use the correspondence between the catchers of $\mathbf{C}$
and the alternators of $S^{(x)}$, and we recover:
\begin{itemize}
\item a catcher $\mathbf{X}$ of $\mathbf{C}$ that corresponds to an arbitrarily chosen element of $\Alt(S^{(x)}) \setminus \F b$
(for example, to $\overline{s}$, but not necessarily);
\item a catcher $\mathbf{Y}$ of $\mathbf{C}$ that corresponds to $b$.
\end{itemize}
Due to the correspondance between the spanning rank of a catcher and the rank of the associated alternator, we recover:
\begin{itemize}
\item that $\mathbf{Y}$ has spanning rank $2n-m$;
\item that $\alpha \mathbf{X}+\beta \mathbf{Y}$ has spanning rank $2n$ for all $\alpha \in \F \setminus \{0\}$ and all $\beta \in \F$.
\end{itemize}

In particular, since $n \leq m \leq 2n-2$, we have
$$2 \leq \sprk \mathbf{Y} \leq n \leq 2n-2.$$
Hence both conditions in the Second Factorization Lemma are satisfied by the pair $(\mathbf{X},\mathbf{Y})$.
The rows $\mathbf{X}$ and $\mathbf{Y}$ are linearly independent over $\F$ (as they have different spanning ranks)
and hence the First Factorization Lemma shows that they are also linearly independent over $\L$.
Since $\rk \mathbf{C}=2n-2$, every catcher of $\mathbf{C}$ is a linear combination of $\mathbf{X}$ and $\mathbf{Y}$
with coefficients in $\L$, and then the Second Factorization Lemma yields $\catch(\mathbf{C})=\Vect_\F(\mathbf{X},\mathbf{Y})$.
In particular, since $\mathbf{B_1}$ and $\mathbf{B_2}$ are linearly independent over $\F$
(see Claim \ref{claim:2rowsindep}), they constitute a basis of $\catch(\mathbf{C})$.
Therefore, by modifying the last two vectors $e_{2n+1}$ and $e_{2n+2}$ of the basis of $V$ we started from, and by replacing them with appropriate
vectors of a basis of $\Vect(e_{2n+1},e_{2n+2})$ (which does not affect the $\mathbf{C}$ block),
we see that no generality is lost in assuming that
$$\mathbf{B_1}=\mathbf{X} \quad \text{and} \quad \mathbf{B_2}=\mathbf{Y},$$
an assumption we will make from that point on.

Now, we will perform a similar extension method as in the previous section.
Applying \eqref{eq:FLA} to $k=1$, we find that $\mathbf{B_1} \mathbf{A}$ and
$\mathbf{B_2} \mathbf{B}$ are left annihilators of $\mathbf{C}$, and hence are linear combinations of
$\mathbf{B_1}$ and $\mathbf{B_2}$ with coefficients in $\mathbb{L}$. Applying the Second Factorization Lemma once more,
we find elements $\mathbf{a},\mathbf{b},\mathbf{c},\mathbf{d}$ of $R_1$ such that
$$\mathbf{X}\mathbf{A}=\mathbf{a}\mathbf{X}+\mathbf{b} \mathbf{Y} \quad \text{and} \quad
\mathbf{Y}\mathbf{A}=\mathbf{c}\mathbf{X}+\mathbf{d} \mathbf{Y},$$
and we recover that the completions
$$\widetilde{\mathbf{X}}:=\begin{bmatrix}
\mathbf{X} & -\mathbf{a} & -\mathbf{b}
\end{bmatrix} \quad \text{and} \quad \widetilde{\mathbf{Y}}:=\begin{bmatrix}
\mathbf{Y} & -\mathbf{c} & -\mathbf{d}
\end{bmatrix}$$
are catchers of $\mathbf{M}$.

\begin{claim}
Every completion of $\mathbf{X}$ as a catcher of $\mathbf{M}$ has spanning rank $2n+2$.
\end{claim}

\begin{proof}
This is proved exactly as for Claim \ref{claim:spanrank2n} by using the fact that $\sprk \mathbf{X}=2n$
and that $\forall z \in P, \; \mathbf{B}(z)=0$, where we recall that $P=(S x)^{\bot_s}$.
\end{proof}

\begin{claim}
Let $\alpha \in \F \setminus \{0\}$ and $\beta \in \F$. Then the spanning rank of
$\alpha \widetilde{\mathbf{X}}+\beta \widetilde{\mathbf{Y}}$ equals $2n+2$.
\end{claim}

\begin{proof}
Indeed, $\alpha \widetilde{\mathbf{X}}+\beta \widetilde{\mathbf{Y}}$ is a catcher of $\mathbf{M}$, and it is a completion of $\alpha \mathbf{X}+\beta \mathbf{Y}$, which has spanning rank $2n$. Then it suffices to work just like in the previous claim.
\end{proof}

Now, at this point we can already state a partial conclusion on the structure of $\Alt(S)$.
First of all, we note that
$$2 \leq \sprk(\mathbf{Y}) \leq \sprk(\widetilde{\mathbf{Y}}) \leq \sprk(\mathbf{Y})+2 \leq 2n.$$
Combining this with $\sprk(\widetilde{\mathbf{X}})=2n+2$ and $\rk \mathbf{M}=2n$,
we deduce from the First and Second Factorization Lemmas that
$$\catch(\mathbf{M})=\Vect_\F(\widetilde{\mathbf{X}},\widetilde{\mathbf{Y}}).$$
Since $S$ is target-reduced, the correspondence between alternators and catchers yields an alternator $c$ of
$S$ that corresponds to the catcher $\widetilde{\mathbf{Y}}$, as well as the following result:

\begin{claim}\label{claim:structureAltdeg}
The space $\Alt(S)$ has dimension $2$, and every element of $\Alt(S) \setminus \F c$ is non-degenerate.
\end{claim}

\subsection{The case where $S^{(x)}$ is reducible (2)}\label{section:badcase2}

Now we analyze the alternator $c$. Note that it is degenerate and nonzero, and
$$0<\rk c=\sprk \widetilde{\mathbf{Y}} \leq 2+\sprk \mathbf{Y}=2+\rk b.$$
We also note that $c$ is, up to multiplication with a nonzero scalar, the only degenerate nonzero element in $\Alt(S)$.

Because $c$ is an alternator of $S$ we find that all the elements of $S$ map $\Lrad(c)$ into $\Rrad(c)$.
By the irreducibility of $S$, we deduce:

\begin{claim}\label{claim:cstraight}
The bilinear form $c$ is not straight.
\end{claim}

Now we will seek to draw consequences of this, until we arrive at a final contradiction.

It is natural now that we investigate the left and right radicals of $c$ more closely.

\begin{claim}\label{claim:specialY}
One has
$$\widetilde{\mathbf{Y}}(x)=\begin{bmatrix}
[0]_{1 \times 2n} & 0 & -1
\end{bmatrix},$$
whereas there is no parameter $y \in V$ for which
$\widetilde{\mathbf{Y}}(y)=\begin{bmatrix}
[0]_{1 \times 2n} & 1 & 0
\end{bmatrix}$.
\end{claim}

\begin{proof}
First of all, we use the specialization-polarization argument, just like in the proof of Claim \ref{claim:spanrank2n},
to obtain
$$(1-\mathbf{d}(x)) \mathbf{Y}=\mathbf{c}(x) \mathbf{X}.$$
Since $\mathbf{X}$ and $\mathbf{Y}$ are linearly independent over $\F$, this yields $\mathbf{c}(x)=0$ and $\mathbf{d}(x)=1$.

Next, assume towards a contradiction that there exists $y \in V$ such that
$\widetilde{\mathbf{Y}}(y)=\begin{bmatrix}
[0]_{1 \times 2n} & 1 & 0
\end{bmatrix}$. Polarizing the quadratic identity $\mathbf{Y} \mathbf{A}=\mathbf{c} \mathbf{X}+\mathbf{d} \mathbf{Y}$
at $y$ then yields
$$\mathbf{Y} \mathbf{A}(y)=\mathbf{X}+\mathbf{c}\mathbf{X}(y)$$
because $\mathbf{Y}(y)=0$, $\mathbf{c}(y)=1$ and $\mathbf{d}(y)=0$.
Hence $\mathbf{X}=\mathbf{Y} \mathbf{A}(y)-\mathbf{c}\mathbf{X}(y)$. Yet obviously
$$\sprk(\mathbf{Y} \mathbf{A}(y)-\mathbf{c}\mathbf{X}(y)) \leq \sprk \mathbf{Y}+1<2n+2=\sprk \mathbf{X},$$
so this yields a contradiction.
\end{proof}

As a consequence, we can compute the rank of $c$.

\begin{claim}
One has $\rk c=1+\rk b$. Consequently
$$\dim \Lrad c=\dim \Rrad c=\dim W+1.$$
\end{claim}

\begin{proof}
Let us consider the spanning rank $r:=\sprk(\widetilde{\mathbf{Y}})$.
By the rank theorem, $r=\sprk \mathbf{Y}+d$, where $d$ denotes the dimension of the space $K$ of all specializations of
$\widetilde{\mathbf{Y}}$ for which the corresponding specialisation of $\mathbf{Y}$ vanishes.
However, the first point of the previous claim shows that $d \geq 1$, and the second one that $d<2$. Hence $d=1$ and we are done.
\end{proof}

Now, we move towards a complete picture of $\Lrad(c)$ and $\Rrad(c)$. Remember the notation
$$P=(S x)^{\bot_s}$$
and the fact that $P$ has dimension $2$ and contains $x$. Moreover, we know that $\mathbf{Y}(z)=0$ for all $z \in P$,
which comes from identity \eqref{eq:P} and the fact that $\mathbf{Y}=\mathbf{B_2}$.

Here, the parameter space for the generic matrices is $V$, and in the notation of Section \ref{section:generic2}
the linear mapping $\Phi$ is $y \in V \mapsto \widehat{y} \in \widehat{S}$, so we simply have, by \eqref{eq:Lradequation},
$$\Lrad(c)=\{y \in V : \; \widetilde{\mathbf{Y}}(y)=0\}.$$

\begin{claim}\label{claim:intersectLradP}
One has $\Lrad(c) \cap P=\F x_0$ for some vector $x_0 \in P \setminus \F x$.
\end{claim}

\begin{proof}
The first point in Claim \ref{claim:specialY} yields in particular that $\widetilde{\mathbf{Y}}(x) \neq 0$, and hence $x \not\in \Lrad(c)$.
Remembering that $\mathbf{Y}(z)=0$ for all $z \in P$, we get from the second point of Claim \ref{claim:specialY} that
$\Theta : z \in P \mapsto (\mathbf{c}(z),\mathbf{d}(z)) \in \F^2$ is not surjective, and hence combining the two points shows that $\Theta$ has rank $1$.
Hence $\Ker \Theta=\Lrad(c) \cap P$ has dimension $1$.
\end{proof}

From now on, we fix the vector $x_0$ from the previous claim. Note that $P=\Vect(x_0,x)$.

\begin{claim}\label{claim:inclusionLrad}
The projection of $\Lrad(c)$ on $Sx$ along $P$ is $W$.
\end{claim}

\begin{proof}
Denote by $\Pi$ the projection of $\Lrad(c)$ on $Sx$ along $P$.
Let $y \in \Lrad(c)$, and split $y=y_1+y_2$ with $y_1 \in S x$ and $y_2 \in P$.
Since $\mathbf{Y}(y_2)=0$ we find $\mathbf{Y}(y_1)=0$, and hence $y_1 \in \Lrad(b)=W$. Hence $\Pi \subseteq W$.
Now, applying the rank theorem to $x \in \Lrad(c) \mapsto x_1 \in \Pi$ yields $\dim \Pi=\dim \Lrad(c)-\dim(\Lrad(c) \cap P)$,
and thanks to Claim \ref{claim:intersectLradP} we obtain $\dim \Pi=\dim \Lrad(c)-1=\dim W$. Hence $\Pi=W$.
\end{proof}

Now, we can also give very minimal (and essentially obvious) information on the right-radical $\Rrad(c)$.

\begin{claim}\label{claim:inclusionRrad}
One has $W=\Rrad(c) \cap S x$.
\end{claim}

\begin{proof}
Let $y \in S x$. Denote by $C \in \F^{2n}$ its matrix in the basis $(e_1,\dots,e_{2n})$ of $S x$.
Hence the completion $\widetilde{C}:=\begin{bmatrix}
C \\
0 \\
0
\end{bmatrix}$ represent $y$ in the basis $(e_1,\dots,e_{2n+2})$ of $V$.
We deduce that
$$y \in \Rrad(c) \Leftrightarrow \widetilde{\mathbf{Y}}\widetilde{C}=0 \Leftrightarrow
\mathbf{Y} C=0 \Leftrightarrow y \in \Rrad(b) \Leftrightarrow y \in W.$$
The conclusion ensues because $W \subseteq S x$.
\end{proof}

\subsection{The case where $S^{(x)}$ is reducible (3)}\label{section:badcase3}

At this point of the proof, we have drawn everything we could from the computation on generic matrices,
and we need to return to the geometric viewpoint in order to move forward.

We will frequently use the basic remark that every $u \in S$ maps $\Lrad(c)$ into $\Rrad(c)$,
a classical consequence of the fact that $c$ is an alternator of $S$.

Our main goal, in the next paragraphs, is to prove that $\Lrad(c) \subseteq \{x_0\}^{\bot_s}$.
So, towards a contradiction, we assume now that $\Lrad(c) \not\subseteq \{x_0\}^{\bot_s}$, on top of the previous assumptions.

We will use a dimension reduction principle that is reminiscent to the one used at the very start of the inductive step.
We introduce the linear subspace
$$S'':=\{u \in S : \; u(x_0)=0\}.$$
By the rank theorem,
$$\dim S''=\dim S-\dim (S x_0).$$
Every element of $S''$ has its range included in $\{x_0\}^{\bot_s}$ and induces an endomorphism $\overline{u}$
of the quotient space $\{x_0\}^{\bot_s}/\F x_0$, an endomorphism that has trivial spectrum and that is $s'$-alternating for
the symplectic form $s'$ induced by $s$ on $\{x_0\}^{\bot_s}/\F x_0$.
The mapping
$$\Psi : u \in S'' \mapsto \overline{u} \in \End(\{x_0\}^{\bot_s}/\F x_0)$$
is then linear, its range is a trivial spectrum subspace of $\calA_{s'}$, and its kernel
is the set of all $u \in S$ that vanish at $x_0$ and map $\{x_0\}^{\bot_s}$ into $\F x_0$.
By Theorem \ref{theo:majdim}, we have $\dim \Psi(S'') \leq (n-1)n$, and hence by the rank theorem
\begin{equation}\label{eq:minodimPsi}
\dim \Ker \Psi \geq (n+1)n-(n-1)n-\dim(S x_0)=2n-\dim (S x_0).
\end{equation}
In the notation of Section \ref{section:totallysingularinvariant},
we have $\Ker \Psi=\calN_{\F x_0} \cap S$, and by Lemma \ref{lemma:NWspantensor} and more specifically the remark thereafter,
the elements of $\calN_{\F x_0}$ have a very simple form: they are exactly the $s$-alternating $2$-tensors
$x_0 \wedge_s y$ with $y \in \{x_0\}^{\bot_s}$, and the mapping $y \in \{x_0\}^{\bot_s} \mapsto x_0 \wedge_s y$ is linear with kernel
$\F x_0$. Remember finally that when $y \in \{x_0\}^{\bot_s} \setminus \F x_0$ the range of $x_0 \wedge_s y$ is $\Vect(x_0,y)$.

\vskip 3mm
\noindent \textbf{Step 1.}
\textbf{For every $y \in \{x_0\}^{\bot_s} \setminus \F x_0$ such that $S$ contains $x_0 \wedge_s y$, one has $x_0 \in \Rrad(c)$ and $y \in \Rrad(c)$.}

Let indeed $y \in \{x_0\}^{\bot_s} \setminus \F x_0$ be such that $x_0 \wedge_s y \in S$.
Since $\Lrad(c)\not\subseteq \{x_0\}^{\bot_s}$ (our starting assumption in this section), we can choose $z \in \Lrad(c)$ such that $s(x_0,z) \neq 0$.
Then $(x_0 \wedge_s y)(z)$ is nonzero and belongs to $\Rrad(c)$, and in particular some nonzero element of $\Vect(x_0,y)$ belongs to $\Rrad(c)$.
It follows that the bilinear form $(z_1,z_2)\in V^2 \mapsto c(z_1,(x_0 \wedge_s y)(z_2))$ has rank at most $1$.
Yet this bilinear form is alternating, so its rank must be even. We conclude that this rank is zero, and hence
the range of $x_0 \wedge_s y$ is included in $\Rrad(c)$. In particular $x_0 \in \Rrad(c)$ and $y \in \Rrad(c)$.

\vskip 3mm
\noindent \textbf{Step 2.}
\textbf{One has $\Rrad(c)=\F x_0 \oplus W$.}

Remembering that $x_0 \in \Lrad(c)$ by Claim \ref{claim:intersectLradP}, we deduce that $S x_0 \subseteq \Rrad(c)$, and
hence $\dim(S x_0) \leq \dim \Rrad(c)= \dim W+1<2n$, which shows that $\Ker \Psi \neq \{0\}$ thanks to \eqref{eq:minodimPsi}.
By Step 1, we deduce that $x_0 \in \Rrad(c)$.

Combining this with $W \cap \F x_0=\{0\}$,  $W \subseteq \Rrad(c)$ and $\dim \Rrad(c)=\dim W+1$, we reach the claimed step.

\vskip 3mm
\noindent \textbf{Step 3.}
\textbf{The space $S$ contains $x_0 \wedge_s y$ for all $y \in W$.}

Combining the previous two steps leads to $\Ker \Psi \subseteq \{x_0 \wedge_s y \mid y \in W\}$.
To reach the new step, it will suffice to prove that $\dim \Ker \Psi \geq \dim W$.
Hence, we take a closer look at $S x_0$.
We know that $S x_0 \subseteq S x$ from the onset of our proof (Claim \ref{claim:actiononP}).

Next, for all $u \in S$ the inclusion
$u(\Lrad(c)) \subseteq \Rrad(c)$ yields $u(\Rrad(c)^{\bot_s}) \subseteq \Lrad(c)^{\bot_s}$.
Besides, by Step 2 and the fact that $W \subset Sx=P^{\bot_s}$, we observe that $x_0 \in \Rrad(c)^{\bot_s}$.
Hence $S x_0 \subseteq \Lrad(c)^{\bot_s}$, and we deduce that
$S x_0 \subseteq Sx \cap \Lrad(c)^{\bot_s}$. Yet, Claim \ref{claim:inclusionLrad}
yields $Sx \cap \Lrad(c)^{\bot_s}=W^{\bot_{\overline{s}}}$.
Therefore $\dim (S x_0) \leq \dim(W^{\bot_{\overline{s}}})=2n-\dim W$.

Applying \eqref{eq:minodimPsi}, we deduce that $\dim \Ker \Psi \geq \dim W$, and Step 3 is reached.

\vskip 3mm
Towards the next step, we introduce the orthogonal complement
$W':=W^{\bot_{\overline{s}}}$
of $W$ in $Sx$ with respect to $\overline{s}$, i.e., $W'=(Sx) \cap W^{\bot_s}$.
Remembering that $W$ is the greatest proper invariant subspace for $S^{(x)}$, we know that
$W'$ is the smallest one, and in particular it is totally $\overline{s}$-singular.
In particular $W' \subseteq W$.

\vskip 3mm
\noindent \textbf{Step 4.}
\textbf{The space $S$ contains all the alternating tensors $y \wedge_s z$ with $y \in \F x_0 \oplus W'$ and $z \in \F x_0 \oplus W$.}

Since $W'$ is a totally $\overline{s}$-singular invariant subspace for $S^{(x)}$,
we find by Proposition \ref{prop:decompbase} that
$S^{(x)}$ includes the space $\calN_{W'} \subseteq \calA_{\overline{s}}$ introduced in Lemma \ref{lemma:NW}.

Let $y \in W'$ and $z \in W$. Then $y \wedge_s z$ vanishes on $P$, and hence it induces an endomorphism of $S x$
which clearly equals $y \wedge_{\overline{s}} z$. As the latter belongs to $S^{(x)}$, we deduce that $y \wedge_s z \in S$.

Since $(-)\wedge_s (-)$ is bilinear and alternating, we conclude, for all $(\alpha,\beta)\in \F^2$, that
$$(\alpha x_0+y) \wedge_s (\beta x_0+z)=\alpha\, x_0 \wedge_s z-\beta\, x_0 \wedge_s y+y \wedge_s z=x_0\wedge_s (\alpha z-\beta y)+y \wedge_s z \in S$$
by combining the previous result with the one of Step 3. Step 4 is therefore reached.

\vskip 3mm
Note finally that $(\F x_0 \oplus W')^{\bot_s}=\F x_0 \oplus W$. Hence $\F x_0+W'$ is totally $s$-singular, and we have just proved that $\calN_{\F x_0 \oplus W'} \subseteq S$
(thanks to Lemma \ref{lemma:NWspantensor}).
By the Invariant Subspace Lemma (Lemma \ref{lemma:invariantsubspacelemma}), we conclude that $\F x_0 \oplus W'$ is $S$-invariant, which contradicts the assumed irreducibility of $S$.

This contradiction shows that the latest assumption, i.e., $\Lrad(c) \not\subseteq \{x_0\}^{\bot_s}$, was false.
We can then conclude:

\begin{claim}\label{claim:equalLrad}
One has $\Lrad(c)=\F x_0 \oplus W$.
\end{claim}

\begin{proof}
Indeed, we have just proved that $\Lrad(c) \subseteq \{x_0\}^{\bot_s}$. Since $\Lrad(c) \subseteq W \oplus \Vect(x_0,x)$ by Claim \ref{claim:inclusionLrad},
it follows that $\Lrad(c) \subseteq \F x_0 \oplus W$, and we conclude because the dimensions are equal.
\end{proof}

\subsection{Completing the proof}\label{section:badcase4}

Combining Claims \ref{claim:inclusionRrad} and \ref{claim:equalLrad} leads to
$$W \subseteq \Lrad(c) \cap \Rrad(c).$$
Since $\dim \Lrad(c)=\dim \Rrad(c)=\dim W+1$,
having $W \neq \Lrad(c) \cap \Rrad(c)$ would lead to $\Rrad(c)=\Lrad(c)$, thereby contradicting the fact that $c$ is not straight
(see Claim \ref{claim:cstraight}).
Hence:
$$W=\Lrad(c) \cap \Rrad(c).$$
At this point the reader could fear that we are still far from the conclusion, but we are actually extremely close to it!

Remember that $W \subseteq S x \subset \{x\}^{\bot_s}$ and that
$\Lrad(c) \cap \Rrad(c)$ does not depend on the way we constructed $c$, because by Claim \ref{claim:structureAltdeg}
the form $c$ is, up to multiplication with a nonzero scalar, the only non-zero and degenerate form in $\Alt(S)$.

Now the trick is to modify the choice of the starting vector $x$, which we had fixed from the start of the proof.
So, we take another vector $x' \in V \setminus \{0\}$ such that $\rk \widehat{x'}=2n$.
Then we can retrace the previous proof, and this time around we know from the start that
$\Alt(S)$ has a nonzero degenerate element, so the work performed in Section \ref{section:goodcase}
allows us to immediately infer that $S^{(x')}$ is reducible, and then we can follow the previous three sections to obtain the inclusion
$$\Lrad(c) \cap \Rrad(c) \subseteq \{x'\}^{\bot_s}.$$

Let us sum up:

\begin{claim}\label{claim:lastclaim}
For every $z \in V \setminus \{0\}$ such that $\rk \widehat{z}=2n$, the inclusion
$\Lrad(c) \cap \Rrad(c) \subseteq \{z\}^{\bot_s}$ holds.
\end{claim}

And now we can finally reach an ultimate contradiction.

Indeed, Claim \ref{claim:lastclaim} yields that the linear subspace $(\Lrad(c) \cap \Rrad(c))^{\bot_s}$ contains all the vectors
$z \in V \setminus \{0\}$ such that $\rk \widehat{z}=2n$. Hence by the Quasitransitivity Lemma (Lemma \ref{lemma:quasitransitivity}),
this subspace equals $V$. By taking $s$-orthogonal complements, we deduce that
$\Lrad(c) \cap \Rrad(c)=\{0\}$, thereby contradicting the inclusion $W \subseteq \Lrad(c) \cap \Rrad(c)$.

This final contradiction shows that $S^{(x)}$ was irreducible from the start, which completes the inductive step
as seen in Section \ref{section:goodcase}. Hence Proposition \ref{prop:keyprop} is proved at last!
As explained earlier, it follows that all the results we announced in the introduction are now proven.

\section{An application to nilpotent spaces of alternating endomorphisms}\label{section:nilpotent}

\subsection{Introduction}\label{section:lastsectionintro}

This last section is devoted to another application of our classification of trivial spectrum subspaces of $\calA_s$.
Here, we replace the trivial spectrum assumption with the nilpotence assumption.
The nilpotence property was studied in \cite{MeshulamRadwan}, over algebraically closed fields, as a special
case of a general theorem on ad-nilpotent subspaces of Lie algebras. In \cite{structuredGerstenhaber1} and \cite{structuredGerstenhaber3}, it was studied more generally over arbitrary fields
(for fields with characteristic other than $2$ in \cite{structuredGerstenhaber1}, and for fields with characteristic $2$ in \cite{structuredGerstenhaber3}).
Let us restate the previously known results:

\begin{theo}
Let $(V,s)$ be a symplectic space with dimension $2n>0$. Let $S$ be a nilpotent subspace of $\calA_s$, i.e., a linear subspace of
$\calA_s$ in which all the elements are nilpotent.
Then $\dim S \leq n(n-1)$.

If equality occurs and $\car(\F) \neq 2$, then $S$ is represented by the matrix space $\NT_n(\F) \triangle \{0\}$ in some symplectic
basis of $(V,s)$, where $\NT_n(\F)$ denotes the space of all strictly upper-triangular $n$-by-$n$ matrices with entries in $\F$.
\end{theo}

Here, we will consider the restricted setting where $|\F|>2n-2$ (it is necessary to use our results on trivial spectrum spaces)
and we will give a quick proof of the following:

\begin{theo}\label{theo:nilpotentalternating}
Assume that $|\F|>2n-2 \geq 2$.
Let $(V,s)$ be a symplectic space with dimension $2n>0$. Let $S$ be a nilpotent subspace of $\calA_s$.
Then $\dim S \leq n(n-1)$.

If equality occurs then $S$ is represented by the matrix space $\NT_n(\F) \triangle \{0\}$ in some symplectic
basis of $(V,s)$.
\end{theo}

Note that here we have entirely removed the restriction on the characteristic from the conclusion.
First of all, the dimension inequality is immediately obtained thanks to Theorem \ref{theo:majdim}.
Next, take a nilpotent subspace $S$ of $\calA_s$ with dimension $n(n-1)$.
Hence it is an optimal trivial spectrum subspace of $\calA_s$.
Let us take a maximal totally $s$-singular $S$-invariant subspace $W$, so that
$S=S_W \triangle \overline{S}^W$ and $\overline{S}^W$ is an optimal irreducible trivial spectrum subspace of $\calA_{\overline{s}}$.
Since nilpotence is preserved in  taking induced endomorphisms, we see that $S_W$ and $\overline{S}^W$
are nilpotent.

The idea then is to prove that $S_W$ is triangularizable and that $\overline{S}^W=\{0\}$.
We will start with the former, which will give an easy introduction to the strategy for the latter.

\subsection{From trivial spectrum spaces of endomorphisms to nilpotent spaces}\label{section:nilpotentstandard}

The fact that $S_W$ is triangularizable could be derived directly from Gerstenhaber's theorem (see \cite{Gerstenhaber}, and alternative proofs in \cite{Mathesetal,dSPGerstenhaberskew,Serezhkin}).
Actually, we will do as if Gerstenhaber's theorem was not known. Instead, we use the classification of optimal trivial spectrum subspaces
of endomorphisms, as recalled in Section \ref{section:reviewstandard}. Using the matrix form, we find nonisotropic matrices $P_1,\dots,P_p$, of respective
sizes $n_1,\dots,n_p$, such that $S_W$ is represented by the space of all block upper-triangular matrices in which the diagonal cells equal,
respectively, $P_1^{-1} A_1,\dots,P_p^{-1} A_p$ for arbitrary alternating matrices $A_1,\dots,A_p$.
Since $S_W$ is nilpotent, it follows that all the spaces $P_k^{-1} \Mata_{n_k}(\F)$ are nilpotent.
If we can prove that they are spaces of $1$-by-$1$ matrices, then they will all vanish and we will have found a basis in which $S_W$
is represented by $\NT_m(\F)$ for some $m$.
So, all we need is to prove the following lemma:

\begin{lemma}\label{lemma:nilpotentstandard}
Let $b$ be a nonisotropic bilinear form on a nonzero vector space $U$ of dimension $k \geq 1$. If $\calA_b$ is nilpotent then $k=1$.
\end{lemma}

In \cite{dSPlargeaffinenonsingular}, this lemma was proved thanks to a trace argument. Instead, here we will use the rank to sort things out.

\begin{proof}
We note that $\calA_b$ is isomorphic to the vector space $\Mata_k(\F)$, through a bijection that preserves the rank.
Hence the possible ranks of the elements of $\calA_b$ are the even integers in $\lcro 0,k\rcro$. If $k$ is even and $k \geq 2$, then $\calA_b$
contains a rank $k$ element, and hence such an element is not nilpotent.

Assume next that $k$ is odd and $k \geq 3$. Choose an arbitrary nonzero vector $x \in U$, and consider the subspace $S:=\{u \in \calA_b : u(x)=0\}$.
Note that $\calA_b x$ is included in the right $b$-orthogonal complement of $x$, whence $\dim (\calA_b x) \leq k-1$.
Every element $u \in S$ satisfies $\forall y \in U, \; b(x,u(y))=-b(y,u(x))=0$ and hence maps into the left $b$-orthogonal $U'$ of $x$.
We note that $x \not\in U'$ because $b$ is nonisotropic.
Hence $\Phi : u \in S' \mapsto u_{|U'} \in \End(U')$ is a rank preserving linear mapping; finally $b$ induces a nonisotropic bilinear form $b'$ on $U'$,
and obviously $\Phi(S') \subseteq \calA_{b'}$. By the rank theorem $\dim \Phi(S')=\dim S'=\dim S-\dim S x \geq \dbinom{k-1}{2}$, whereas
$\dim \calA_{b'}=\dbinom{k-1}{2}$. It follows that $\Phi(S')=\calA_{b'}$, and hence $\Phi(S')$ contains an element $\Phi(u_0)$ of rank $k-1$, so $\Phi(u_0)$ is not nilpotent.
Hence $u_0$ is not nilpotent either.
\end{proof}

\subsection{Analyzing the irreducible trivial spectrum spaces}

It remains to prove that no irreducible optimal trivial spectrum subspace of $\calA_s$ is nilpotent, besides the trivial case where $s$ has rank $0$.

The idea for this is similar to the one in the proof of Lemma \ref{lemma:nilpotentstandard}. We will use the results on the possible ranks in a $\calU_{s,\L}$ space
(Corollary \ref{cor:evenranks}).

\begin{prop}\label{prop:nilpotentalternating}
Let $(V,s)$ be a symplectic space of dimension $2n\geq 4$, with $|\F| >2n-2$. Then no irreducible optimal trivial spectrum subspace of $\calA_s$ is nilpotent.
\end{prop}

\begin{proof}
Let $S$ be an irreducible optimal trivial spectrum subspace of $\calA_s$. By Theorem \ref{theo:classoptimalirrTspectrum},
there exists a quadratic element $a \in \End(V)$, with irreducible minimal polynomial, such that
$b : (x,y) \mapsto s(x,a(y))$ is nonisotropic and $S=\calU_{s,\L}=\calA_s \cap \calA_b$.
If $n$ is even, then Corollary \ref{cor:evenranks} immediately yields that $S$ contains an element of rank $2n$, which cannot be nilpotent.

Assume now that $n$ is odd, and choose an arbitrary vector $x_0 \in V \setminus \{0\}$. Denote by $a^\star$ the $s$-adjoint of $a$,
so that
$$\forall (x,y)\in V^2, \; s(a^\star(x),y)=s(x,a(y))=b(x,y).$$
The endomorphism $a^\star \in \End(V)$ has the same minimal polynomial as $a$, and hence it has no eigenvalue in $\F$.
Hence the space $P:=\Vect(x_0,a^\star(x_0))$ has dimension $2$, and it is $s$-regular because $b$ is nonisotropic.
Moreover $S x_0 \subseteq P^{\bot_s}$ because all the elements of $S$ are both $s$-alternating and $b$-alternating.
In particular $\dim S x_0 \leq 2n-2$.
Set
$$S':=\{u \in S : u(x_0)=0\}.$$
As the elements of $S'$ are both $s$-alternating and $b$-alternating, they map into $P^{\bot_s}$:
indeed, for all $u \in S$ and all $y \in V$, we write
$$s(x_0,u(x))=-s(x,u(x_0))=0$$
and
$$s(a^\star(x_0),u(x))= b(x_0,u(x))=-b(x,u(x_0))=0.$$
The subspace $P^{\bot_s}$ is $s$-regular because $P$ is $s$-regular. Moreover $P$ is invariant under $a^\star$ because $a^\star$
is quadratic, and hence $P^{\bot_s}$ is invariant under $a$.
Hence $s$ induces a symplectic form $s'$ on $P^{\bot_s}$, and $a$ induces an endomorphism
$a'$ of $P^{\bot_s}$, again quadratic with irreducible minimal polynomial; finally we denote by $b'$ the restriction of $b$ to $(P^{\bot_s})^2$, and we note that $b'$ remains nonisotropic.
This yields an injective restriction mapping
$$\Phi : u \in S' \mapsto u_{P^{\bot_s}} \in \calA_{b'} \cap \calA_{s'.}$$
Hence $\dim \Phi(S') \leq \dim(\calA_{b'} \cap \calA_{s'})=(n-1)(n-2)$ by Proposition \ref{prop:UL}.
And finally $\dim S=\dim (S x_0)+\dim S' \leq 2n-2+\dim \Phi(S')$, which yields
$\dim \Phi(S') \geq (n-1)(n-2)$. Hence $\Phi(S')=\calA_{b'} \cap \calA_{s'}$.

Finally $\calA_{b'} \cap \calA_{s'}$ contains an element of rank $2n-2$, because $2n-2$ is a multiple of $4$, and this element is not nilpotent.
Hence there is a corresponding non-nilpotent element in $S'$. This completes the proof.
\end{proof}

By coming back to the setup of Section \ref{section:lastsectionintro}, we find that the space $W^{\bot_s}/W$ is zero, i.e., $\overline{S}^W$ is void,
and $S_W$ is represented by $\NT_n(\F)$, as seen in Section \ref{section:nilpotentstandard}. Hence $S$ is represented by $\NT_n(\F) \triangle \{0\}$ in some symplectic basis.
In particular, there is only one optimal nilpotent subspace of $\calA_s$ up to conjugation by the symplectic group $\Sp(s)$, and it is represented in a well-chosen symplectic basis
by the space of all matrices of the form
$$\begin{bmatrix}
U & B \\
0 & U^T
\end{bmatrix} \quad \text{with $U \in \NT_n(\F)$ and $B \in \Mata_n(\F)$.}$$
This proves Theorem \ref{theo:nilpotentalternating}.


\begin{thebibliography}{1}
\bibitem{AtkinsonPrim}
M.D. Atkinson,
{Primitive spaces of matrices of bounded rank II.}
J. Austral. Math. Soc. (Ser. A)
{\bf 34} (1983) 306--315.

\bibitem{AtkLloydPrim}
M.D. Atkinson, S. Lloyd,
{Primitive spaces of matrices of bounded rank.}
J. Austral. Math. Soc. (Ser. A)
{\bf 30} (1980) 473--482.

\bibitem{DraismaKraftKuttler}
J. Draisma, H. Kraft, J. Kuttler,
{Nilpotent subspaces of maximal dimension in semisimple Lie algebras.}
Compos. Math.
{\bf 142} (2006) 464--476.

\bibitem{FLR}
P. Fillmore, C. Laurie, H. Radjavi,
{On matrix spaces with zero determinant.}
{Linear Multilinear Algebra}
{\bf 18} (1985) 255--266.

\bibitem{Gerstenhaber}
M. Gerstenhaber,
{On nilalgebras and linear varieties of nilpotent matrices (I).}
{Amer. J. Math.}
{\bf 80} (1958) 614--622.

\bibitem{HuangLandsberg}
H. Huang, J.M. Landsberg,
{On Linear spaces of matrices of bounded rank.}
{Sel. math., New ser.}
{\bf 32:30} (2026)  

\bibitem{BukovsekOmladic}
D. Kokol Bukov\v{s}ek, M. Omladi\v{c},
{Linear spaces of symmetric nilpotent matrices,}
Linear Algebra Appl.
{\bf 530} (2017) 384--404.

\bibitem{Mathesetal}
B. Mathes, M. Omladi\v c, H. Radjavi,
{Linear spaces of nilpotent matrices.}
Linear Algebra Appl.
{\bf 149} (1991) 215--225.

\bibitem{Meshulam}
R. Meshulam,
{On two extremal matrix problems.}
Linear Algebra Appl.
{\bf 114-115} (1989) 261--271.

\bibitem{MeshulamRadwan}
R. Meshulam, N. Radwan,
{On linear subspaces of nilpotent elements in a Lie algebra.}
Linear Algebra Appl.
{\bf 279} (1998) 195--199.

\bibitem{Quinlan}
R. Quinlan,
{Spaces of matrices without non-zero eigenvalues in their field of definition, and a question of Szechtman.}
{Linear Algebra Appl.}
{\bf 434} (2011) 1580--1587.

\bibitem{Riehm}
C. Riehm,
{The equivalence of bilinear forms.}
{J. Algebra}
{\bf 31} (1974) 45--66.

\bibitem{Rubei}
E. Rubei,
{Affine subspaces of antisymmetric matrices with constant rank.}
{Linear Multilinear Algebra}
{\bf 72-11} (2024) 1741--1750.

\bibitem{Scharlaupairs}
W. Scharlau,
{Paare alternierender Formen.}
Math. Z.
{\bf 147} (1976) 13--19.

\bibitem{dSPgivenrank}
C. de Seguins Pazzis,
{On the matrices of given rank in a large subspace.}
Linear Algebra Appl.
{\bf 435-1} (2011) 147--151.

\bibitem{dSPlargeaffinenonsingular}
C. de Seguins Pazzis,
{Large affine spaces of non-singular matrices.}
Trans. Amer. Math. Soc.
{\bf 365} (2013) 2569--2596.

\bibitem{dSPAtkinsontoGerstenhaber}
C. de Seguins Pazzis,
{From primitive spaces of bounded rank matrices to a generalized Gerstenhaber theorem.}
Quart. J. Math.
{\bf 65-2} (2014) 319--325.


\bibitem{dSPLLD1}
C. de Seguins Pazzis,
{Local linear dependence seen through duality I.}
J. Pure Appl. Algebra
{\bf 219} (2015) 2144--2188.

\bibitem{dSPLLD2}
C. de Seguins Pazzis,
{Local linear dependence seen through duality II.}
Linear Algebra Appl.
{\bf 462} (2014) 133--185.


\bibitem{dSPGerstenhaberskew}
C. de Seguins Pazzis,
{On Gerstenhaber's theorem for spaces of nilpotent matrices over a skew field.}
Linear Algebra Appl.
{\bf 438-11} (2013) 4426--4438.

\bibitem{dSPRC}
C. de Seguins Pazzis,
{Range-compatible homomorphisms on matrix spaces.}
 Linear Algebra Appl.
{\bf 484} (2015) 237--289.

\bibitem{dSPPrimitiveF2}
C. de Seguins Pazzis,
{Primitive spaces of matrices with upper rank two over the field with two elements.}
Linear Multilinear Algebra
{\bf 64} (2016) 1321--1353.

\bibitem{structuredGerstenhaber1}
C. de Seguins Pazzis,
{The structured Gerstenhaber problem I.}
Linear Algebra Appl.
{\bf 567} (2019) 263--298.

\bibitem{structuredGerstenhaber2}
C. de Seguins Pazzis,
{The structured Gerstenhaber problem II.}
Linear Algebra Appl.
{\bf 569} (2019) 113--145.

\bibitem{structuredGerstenhaber3}
C. de Seguins Pazzis,
{The structured Gerstenhaber problem III.}
Linear Algebra Appl.
{\bf 601} (2020) 134--169.

\bibitem{dSPaffinealt}
C. de Seguins Pazzis,
{On affine spaces of alternating matrices with constant rank.}
Linear Multilinear Algebra
{\bf 73-3} (2025) 577--591.

\bibitem{dSPtriangularizable}
C. de Seguins Pazzis,
{Spaces of triangularizable matrices.}
Acta Sci. Math. (Szeged)
{\bf 91} (2025) 369--399.

\bibitem{dSPaffineunits}
C. de Seguins Pazzis,
{Affine spaces of units in simple algebras.}
Preprint (2025), arXiv: https://arxiv.org/abs/2508.06934

\bibitem{Serezhkin}
V.N. Serezhkin,
{Linear transformations preserving nilpotency (in Russian).}
Izv. Akad. Nauk BSSR, Ser. Fiz.-Mat. Nauk
{\bf 125} (1985) 46--50.

\bibitem{Sergeichuk}
V.V. Sergeichuk,
{Classification problems for systems of forms and linear mappings.}
Math. USSR-Izv.
{\bf 31} (1988) 481--501.


\end{thebibliography}
\end{document}